\documentclass[a4paper,fleqn,10pt]{article}

\usepackage{fullpage,color}
\usepackage{dsfont}
\usepackage[T1]{fontenc}
\usepackage[utf8]{inputenc}
\usepackage[english]{babel}
\usepackage{amsmath}
\usepackage{amsfonts}
\usepackage{amssymb}
\usepackage{amsthm}
\usepackage{mathpazo}
\usepackage{eqnarray}
\usepackage{comment}
\usepackage{ulem}
\usepackage[pdftex]{graphicx}
\usepackage[margin=1in]{geometry}
\usepackage{float}
\usepackage{cite}
\usepackage{authblk}
\usepackage{mathtools}
\usepackage{empheq}
\usepackage{multirow}
\usepackage[table]{xcolor}
\usepackage{tikz}
\usetikzlibrary{calc,
				fadings,
				decorations.pathreplacing,
				arrows,
				positioning,
				shapes,
				intersections,
				quotes,
				patterns,
				automata,
				fit}
\usepackage{pgfplots}
\definecolor{totocolor}{rgb}{1,0.6,1}
\colorlet{LightRed}{red!30}
\colorlet{ColorPink}{blue!20}
\colorlet{lightgreen}{green!50}
\tikzstyle{terminal}=[circle,draw]
\pgfplotsset{width=7cm,compat=1.15}

\tikzset{
    state/.style={
           rectangle,
           rounded corners,
           draw=black, thick,
           minimum height=2em,
           inner sep=2pt,
           text centered,
           },
}

\usepackage{caption}
\usepackage{subcaption}
\usepackage{manfnt}
\definecolor{couleur_r}{RGB}{168, 173, 0}
\usepackage{soul}

\usepackage{accents}

\usepackage{comment}
\usepackage{hyperref}

\newcommand{\cqfd}{\nobreak \ifvmode \relax \else
      \ifdim\lastskip<1.5em \hskip-\lastskip
      \hskip1.5em plus0em minus0.5em \fi \nobreak
      \vrule height0.75em width0.5em depth0.25em\fi}

\makeatletter
\newcommand{\leqnomode}{\tagsleft@true}
\makeatother
\usepackage[skipempty]{credits}

\newtheorem{hyp}{Hypothesis}[section]

\def\mm#1{{\color{red}#1}}
\usepackage{stmaryrd}
\newcommand{\inhib}{\relbar\mapsfromchar}

\def\<{\langle}
\def\>{\rangle}
\def\t{\widetilde}
\def\i{\infty}

\def\asymp{{\textsc{asym}}}
\def\adm{{\textsc{adm}}}
\def\panin{{\textsc{panin}}}
\def\pdac{{\textsc{pdac}}}

\def\tasymp{{t_\asymp}}
\def\tadm{{t_\adm}}
\def\tpanin{{t_\panin}}
\def\tpdac{{t_\pdac}}

\def\pdacadvanc{{t_\pdac^{advanced}}}
\def\pdacearly{{t_\pdac^{early}}}
\def\tpaninadvanc{{t_\panin}}

\newcommand{\verteq}{\rotatebox{90}{$\,=$}}

\def\Chi{\raise .3ex
\hbox{\large $\chi$}}

\def\({\Bigl (}
\def\){\Bigr )}

\newcommand{\bea}{$$ \begin{array}{lll}}
\newcommand{\eea}{\end{array} $$}
\newcommand{\bi}{\begin{itemize}}
\newcommand{\ei}{\end{itemize}}

\newtheorem{prop}{Proposition}
\newtheorem{theorem}{Theorem}

\newtheorem{lemma}{Lemma}

\newtheorem{definition}{Definition}
\newtheorem{remark}{Remark}

\newcommand{\dt}{\frac{d}{dt}}
\newcommand{\ds}{\frac{d}{ds}}
\newcommand{\NC}{N_c}
\newcommand{\F}{\mathcal{F}}
\newcommand{\X}{\mathcal{X}}

\newcommand{\defeq}{\vcentcolon=}
\renewenvironment{proof}{\noindent{\bf Proof.}}{\hfill
  $\blacksquare$\par\noindent}

\providecommand{\keywords}[1]
{
  \small	
  \textbf{\textit{Keywords: }} #1
}
\credit{M.J.C}  {1,0,1,0,1,1,0,0,1,0,0,1,1,1}
\credit{S.C}    {1,1,0,1,1,0,0,0,0,0,0,1,1,1}
\credit{F.H}    {1,0,1,1,1,1,0,0,0,0,0,1,1,1}
\credit{F.M}    {1,1,0,1,1,0,0,0,0,0,0,1,1,1}
\credit{M.M}    {1,0,1,0,1,1,0,0,1,0,0,1,1,1}
\credit{P.P}    {1,0,1,1,1,1,0,0,0,0,0,1,1,1}
\begin{document}

\title{
A continuous approach of modeling tumorigenesis and axons regulation for the pancreatic cancer.}

\date{}
\author[2]{\small Marie-Jose Chaaya}
\author[1]{\small Sophie Chauvet}
\author[2]{\small Florence Hubert}
\author[1]{\small Fanny Mann}
\author[2, 3]{\small Mathieu Mezache}
\author[2]{\small Pierre Pudlo}

\affil[1]{\footnotesize Aix Marseille Univ, CNRS, IBDM (UMR 7288), Turing Centre for Living systems, Marseille, France}
\affil[2]{\footnotesize Aix Marseille Univ, CNRS, I2M (UMR 7373), Turing Centre for Living systems, Marseille, France}
\affil[3]{\footnotesize Université Paris-Saclay, INRAE, MaIAGE (UR 1404), 78350 Jouy-en-Josas, France}

\maketitle
\keywords{Dynamical system, Partial differential equations, Cancer, Parameter calibration, In silico denervation}
\begin{abstract}

The pancreatic innervation undergoes dynamic remodeling during the development of pancreatic ductal adenocarcinoma (PDAC). 
Denervation experiments have shown that different types of axons can exert either pro- or anti-tumor effects, but conflicting results exist in the literature, leaving the overall influence of the nervous system on PDAC incompletely understood. 
To address this gap, we propose a continuous mathematical model of nerve-tumor interactions that allows in silico simulation of denervation at different phases of tumor development. 
This model takes into account the pro- or anti-tumor properties of different types of axons (sympathetic or sensory) and their distinct remodeling dynamics during PDAC development. 
We observe a “shift effect” where an initial pro-tumor effect of sympathetic axon denervation is later outweighed by the anti-tumor effect of sensory axon denervation, leading to a transition from an overall protective to a deleterious role of the nervous system on PDAC tumorigenesis. 
Our model also highlights the importance of the impact of sympathetic axon remodeling dynamics on tumor progression. 
These findings may guide strategies targeting the nervous system to improve PDAC treatment.

\end{abstract}
\section{Introduction}\label{sec:introduction}
{
\noindent The nervous system plays an important role in regulating various bodily functions and disease processes, including cancer development and progression \cite{winkler2023cancer}.
Denervation studies have shown that the peripheral nervous system (PNS) can either promote or inhibit cancer growth, depending on the specific types of nerves and cancers involved. 
For example, in mouse models of pancreatic ductal adenocarcinoma (PDAC), selective ablation of pancreatic sympathetic innervation has been associated with accelerated tumor growth, increased metastasis and decreased survival \cite{guillot2022sympathetic}. 
Conversely, removal or silencing of sensory neurons has been shown to slow tumor growth and improve survival (\cite{saloman2016ablation}; \cite{sinha2017panin}). 
These findings, together with others, suggest a broad model in which the autonomic nervous system (including both its sympathetic and parasympathetic branches) has an anti-tumor property in PDAC, whereas the sensory nervous system has an inverse pro-tumor activity (Figure \ref{fig:scheme_axons_pdac}). \\
}

{
\noindent According to the data presented, the combined effect of the PNS on PDAC tumour progression results from a mixture of pro- and anti-tumour activities. 
This combined effect can be assessed in experimental models by surgical denervation of the pancreas, which disrupts the mixed nerve of sympathetic and sensory axons supplying the pancreas. 
Such an intervention mimics procedures performed in patients (eg. celiac neurolysis and splanchnicectomy) to manage abdominal pain in unresectable PDAC. 
However, surgical denervation studies in mice have led to divergent results. 
A pro-tumor effect was observed when denervation was performed before the onset of pathology \cite{guillot2022sympathetic}, whereas an anti-tumor effect was observed when denervation was performed after tumor establishment \cite{renz2018beta2}. 
The reasons behind these conflicting results are not yet understood.
One hypothesis suggests a switch in sympathetic function over the course of pathology. 
Initially, sympathetic axons may have an anti-tumor effect in the early (pre-)cancer stages, before possibly switching to a promoting role on tumor growth at a later stage. 
Another interpretation is that the relative abundance of pro- and anti-tumoral axons may vary at different denervation times. 
Indeed, studies have shown different remodeling patterns for sympathetic and sensory innervation. 
Higher levels of sympathetic innervation are found in pre-cancerous lesions compared to both healthy and cancerous tissues, while sensory axons are markedly increased within the cancerous lesions (\cite{guillot2022sympathetic}; \cite{demir2015neural}). \\
}

{
\noindent These findings highlight the importance of the PNS as a potential therapeutic target for the modulation of PDAC. 
However, they also highlight the need for a deeper understanding of the individual and combined effects of the different axon types, taking into account their remodeling patterns and influence on the tumor, in order to develop strategies to deplete or inhibit the PNS in PDAC. 
To address this need, we develop a mathematical model that allow us to study the consequences of denervating sensory and/or sympathetic axons at different times during tumor development and progression. \\
}

{
\noindent So far, two mathematical models have been developed to better understand the influence of the PNS on cancer. 
The first model investigated the pro-tumoral effect of the autonomic nervous system in prostate cancer \cite{lolas2016tumour}. 
However, due to the opposite function of the sympathetic innervation in pancreatic cancer, another mathematical model was developed to specifically study PDAC \cite{chauvet2023tumorigenesis}. 
This model formalized the interactions between cancer progression and axons using a compartmentalized differential equations model, where each compartment corresponds to a stage of cancer progression (healthy, pre-cancerous, cancerous). 
The asymptotic behavior of the system shows that the pathological state, where only cancer cells persist, is globally asymptotically stable. 
The impact of denervation was simulated in silico and recapitulated the biological data of denervation performed at early stage, before the onset of pathology. 
However, other times of denervation have not been investigated. \\
}

\noindent The aim of this paper is to introduce a new continuous mathematical model describing the relationship between axons and cancer, and to simulate in silico denervation of sympathetic and sensory neurons at different times. 
The continuous model considers the cell phenotype (from healthy to cancerous) as a continuous variable. 
Consequently, the description of cancer progression and axon remodeling occurring during this process becomes more precise. \\

\noindent We organized the paper as follows. 
In Section \ref{sec:model}, we set up the model illustrating the effect that PNS axons have on cancer development and progression. 
A theoretical study on the well-posedness of the model is given in Section \ref{sec:wellposedness}. 
This is followed by a qualitative study in Section \ref{sec:qualitative_study} where we show that the final state of the disease can belong to one out of three cases: either no cancer cells exist, either cancer and healthy cells co-exist, or only cancer cells exist. 
Also, explicit bounds on the time of appearance of the first cancer cells are given under some hypotheses on the parameters. 
Next, in Section \ref{sec:param}, we give precision on the parameters of the model and explain how to denervate in silico. 
We then calibrate the model in Section \ref{sec:calibration} by minimizing a 2-dimensional criterion. 
We obtain several sets of parameters that fit the data. 
Using some of the sets of parameters obtained, we apply numerical simulations to study in silico denervation using an indicator of invasive potential detailed in Section \ref{sec:indic}. 
In Section \ref{sec:axonsrole}, we perform in silico denervation separately for each axon type (sympathetic or sensory) or simultaneously at defined time points.
The dynamical system and its in silico denervation are implemented using a finite volume approach and the algorithm can be found at (\href{https://github.com/MarieJosec/PDE_Axons_Innerv}{\texttt{https://github.com/MarieJosec/PDE\_Axons\_Innerv}}).
In Section \ref{sec:heatmap}, denervation is performed for each axon type or for both at different time points.  
The results recapitulate the different outcomes of surgical denervation observed in \cite{guillot2022sympathetic} and  \cite{renz2018beta2} [Renz et al., 2018a] and support a ’shift effect’ of PNS function in PDAC from anti-tumor to pro-tumor. 
The model predicts that this transition does not occur through a change in sympathetic function, but rather depends on the remodeling dynamics of the sympathetic and sensory axons and on the strength of the sympathetic inhibition on tumor growth.

\begin{figure}[ht!]
\centering
\scalebox{0.9}{\begin{tikzpicture}[scale=0.6,
  box/.style={rectangle, draw=blue, thick, fill=blue!20, text width=13em,align=center, rounded corners, minimum height=4em},
    box2/.style={rectangle, draw=red, thick, fill=red!20, text width=13em,align=center, rounded corners, minimum height=4em},
    boxt/.style={rectangle, draw=black, thick, fill=black!10, text width=20em,align=center, rounded corners, minimum height=8em},
     boxtt/.style={rectangle, draw=black, thick, fill=black!2, text width=6em,align=center, rounded corners, minimum height=3em},
    line/.style ={ ultra thick, -latex', shorten <=2pt,shorten >=2pt},
    fleche/.style={->, >=latex', shorten >=2pt, shorten <=2pt,ultra thick},
    ]
     \draw (0,0) node[boxt] (PDAC){{\bf PDAC}};
     \draw (3,-1) node[boxtt] (pdac_prog){{\bf Progression}};
     \draw (-3,-1) node[boxtt] (pdac_gro){{\bf Growth}};
     \draw (0,10) node[box] (para){{\bf Parasympathetic neurons}};
     \draw (10,10) node[box] (symp){{\bf Sympathetic neurons}};
     \draw (-10,10) node[box2] (sensor){{\bf Sensory neurons}};
     
 
 \draw[fleche](sensor.south) to[bend right=20] (-11,-5)  --node[above,midway,sloped]{\small \cite{saloman2016ablation}} (3,-5)-- (pdac_prog.south);
 
 \draw[fleche](sensor.south) -- (-10,-3) node[below, midway, sloped]{\small \cite{banh2020neurons}} --(-3,-3)--(pdac_gro.south);
 
 \draw[fleche](symp.south) to [bend right = 20]  node[above,midway,sloped]{\small \cite{renz2018beta2}} (4.9,2.4)
 ;
\draw [line,-|] (para.south) -- (PDAC.north) node[above,midway,sloped]{\small \cite{hayakawa2017nerve}} node[below,midway,sloped]{\small \cite{ renz2018cholinergic}};

\draw [fleche] (sensor.south) -- (PDAC.north west) node[above,midway,sloped]{\small \cite{liddle2007role}} node[below,midway,sloped]{\small \cite{sinha2017panin}};

\draw [line, -|] (symp.south) -- (PDAC.north east) node[below,midway,sloped]{\small \cite{guillot2022sympathetic}};

\end{tikzpicture}
}
\caption{Graphical description of the state-of-the-art concerning the interactions between sensory axons, parasympathetic neurons, sympathetic axons and tumor cells (PDAC). The arrow with the symbol $\rightarrow$ denotes a promoting effect and the arrow with the symbol $\inhib $ denotes an inhibiting effect. 
}\label{fig:scheme_axons_pdac}
\end{figure}
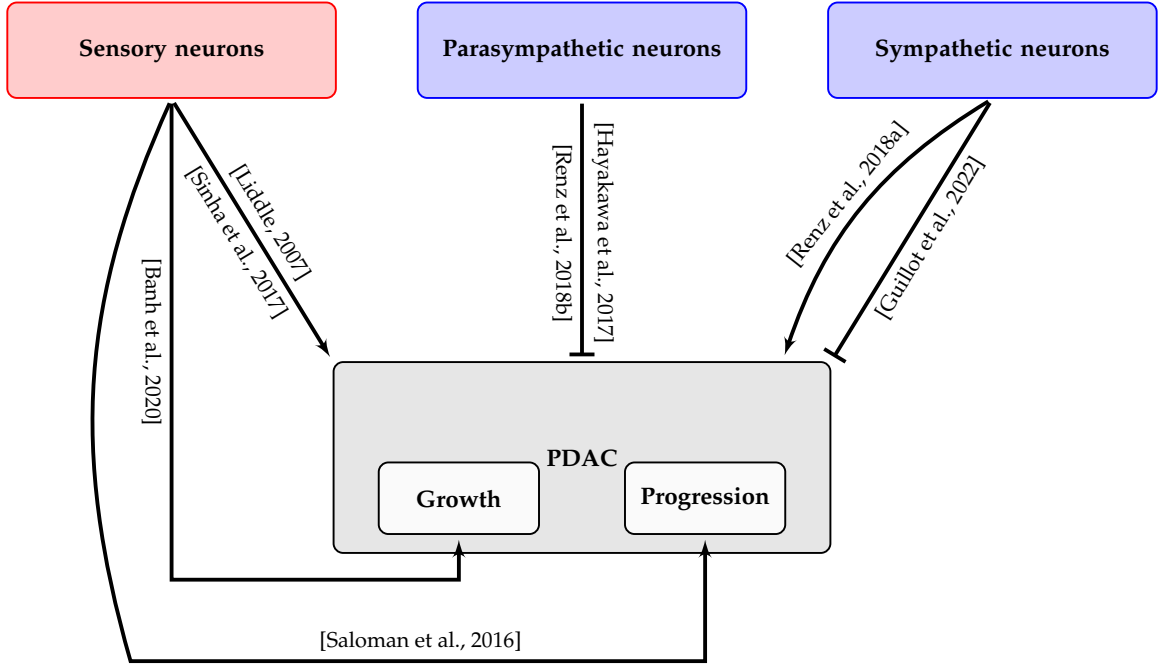

\section{Mathematical model} \label{sec:math}

In this section, we establish and explore an original continuous model for studying cancer progression and its regulation by axons. Initially, a compartmental model was introduced in \cite{chauvet2023tumorigenesis} to investigate the interactions. However, this model, based on ordinary differential equations (ODEs), focused on discrete stages of cancer progression from a healthy state to  a tumoral stage. Each stage was treated as a variable, compartmentalizing the different steps, which failed to capture the continuous nature of the biological process. To address this limitation, we introduce another mathematical model that adopts a continuous framework for cell phenotypes, spanning from a healthy state to a tumoral stage, thereby bringing it closer to the reality of the biological process. This model also considers the heterogeneity of the cell phenotype. In alignment with recent literature (cf. \cite{eftimie2020kinetic}, which introduces a phenotype-structured model to describe the heterogeneous population of macrophages), we propose employing structured population dynamics in the model. 

\subsection{Model presentation}\label{sec:model}
We focus on the pancreas as the domain of our model,
by taking into account the pancreatic cells and the nerve axons. In this model, we denote the unknowns $Q$, $A_1$ and $A_2$
where:
\begin{itemize}
\item  $Q(t,x)$ is the density of cells structured by the progression state of the disease. The variable $t$ corresponds to time with $t\in [0,+\infty)$ and the variable $x$ corresponds to the phenotype 
of the cell where  $x\in \Omega$ with $\Omega\defeq \mathbb{R}$ or $\Omega\defeq (-L,L)$ with the constant $L>0$ (finite truncated domain).  Cells with a positive phenotype correspond to cancerous
cells. Whereas for $x\in(-L,0)$, we start from healthy cells undergoing 
precancerous phenotypes before reaching the cancerous state. The lower the phenotype variable $x$ is, the healthier the cells $Q(t,x)$ are.
\item  The variable $A_1(t)$ is the normalized density of the sympathetic axons with respect to time $t\in[0,+\infty)$. 
\vspace*{1mm}
\item  The variable $A_2(t)$ is the normalized density of sensory axons with respect to time $t\in[0,+\infty)$. 
\vspace*{1mm}
\end{itemize}
Hence, the variables $A_i$ for $i=1,\, 2$ are unit-less and are bounded, i.e.   $A_i(t)\in (0,1)$ for $t\geq 0$. 

\noindent The total amount of pancreatic cells $N(t)$ and the total amount of cancer cells $\NC(t)$ can be obtained from the density $Q$ by 
$$N(t) = \int_\Omega Q(t,x) dx,\quad \NC (t)  = \int_\Omega\psi(x) Q(t,x) dx,  $$
where $\psi$ is an indicator function whose support is a subset of $\mathbb{R}_+$.
These two macroscopic quantities play significant roles in the dynamics  of the model. They allow to compute crucial indicators such as the proportion of cancer cells in the system or the growth of the size of the pancreas induced by cancer cells. 
For instance, the ratio ${\NC(t)}/{N(t)}$ corresponds to the proportion of cancer cells in the system.
If the ratio is equal to $1$ then all the cells are cancerous ones and if the ratio is equal to 0, then no cancer cells are present in the system.
Also the ratio ${N(t)}/{N(0)}$ gives information on the growth of the size of the pancreas.\\
Moreover, we introduce the notation $\X \defeq (Q,A_1,A_2)$ to group the unknowns into a tuple. Hence, the dynamic of $Q(t,x)$ is given by the following transport-growth equation :
$$\partial_t Q(t,x)  + \underbrace{\partial_x \big[f(t,x;{\mathcal X})Q(t,x) \big]}_{\text{Progression of the disease}} = \underbrace{g(t,x;{\mathcal X})}_{\text{Growth term}}.$$
\\

\noindent{\bf Progression of the disease.} 
We model the evolution of the disease as a transport term on the phenotype axis for the cell densities $Q$ in the partial differential equation governing the dynamics of cells. 
The speed of progression of the disease (i.e. the transport speed) denoted $f(t,x;{\mathcal X})$ is regulated by the presence of the axons and the cancer cells. It takes the following form
\begin{equation}
f(t,x;{\mathcal X}): = \pi(x)  \big[1- \beta(x)\rho(A_1(t)) + \delta(x)A_2(t)\big] +\eta \left(x, \tfrac{\NC (t)}{N(t)}\right)
\label{eq:transport_term_pde}
\end{equation}
where
\begin{itemize}
    \item $\pi(x) $ represents a basal amplitude for the speed of disease progression. Since we expect the 
    transformation of healthy cells to be slow at early stages, we consider the transfer to be almost negligible for $x\ll 0$. This transport is then expected to increase up to a plateau observed during the early PDAC stage. Moreover, we naturallly assume that at the boundary of the phenotype axis ($+\infty$ or $L$), the function $\pi$ vanishes.  
    \item The basal amplitude is modulated by the presence of the sympathetic axons that slow down the disease progression (cf. \cite{guillot2022sympathetic}).
    {The function $\beta$ modulates the maximum rate of the regulation depending on the phenotype variable $x$. 
    The function $\rho$ modulates the basal amplitude negatively because of the inhibiting mechanism of $A_1$, under the assumption that a sufficient density of sympathetic axons is required to have an impact on the cancer progression process.
    Thus, the function $\rho$ vanishes for small values of $A_1$.}
    \item The basal amplitude is also regulated by the presence of sensory axons  that speed up the disease progression (cf. \cite{saloman2016ablation}). 
    {The function $\delta$ modulates the maximum rate of the regulation depending on the phenotype variable $x$.}
    \item Moreover, we assume that cancer cells in the pancreas are inducing an acceleration of the cancer progression of the healthy cells in the late stages of the disease. {This assumption is formalizing the crosstalk between PDAC cells and other cells where the tumorigenic environment facilitates the ability of cancer cells to survive and proliferate to the detriment of healthy cells (cf. \cite{ALEXANDER201680}).}
In order to model this response, we introduce the function $\eta$ which contributes to the cancer progression acceleration. This function takes as variables the ratio ${\NC(t)}/{N(t)}$ which quantifies the presence and the proportion of pathological cells in the system and the phenotype variable $x$. Also, we assume that $\eta$ is compactly supported and its support is located in a neighbourhood corresponding to the healthy cells on the phenotype axis.
\end{itemize}
An additional assumption in our model is that the cancer progression is non-reversible. For instance, once cells start to 
progress towards a cancerous state, they cannot recover, meaning that they cannot have an healthy phenotype state later in time. This implies that the transport term $f$ remains non-negative. 
\\

\noindent{\bf Growth term.} {In the model, the growth of the cancer cell densities is modelled  by a logistic-type growth which is given by the growth term $g$:}
\begin{equation}
g(t,x;{\mathcal X}):=r(x)Q(x)\left(1- \frac{N(t)}{\tau_C} - \mu_1 A_1(t)+\mu_2A_2(t)\right).
\label{eq:growth_term_pde}
\end{equation}

\noindent The function $r$ of the phenotype variable $x$ is the basal growth rate of the proliferating cells. 
Since {growth starts at pre-cancerous stages and accelerate around cancerous stages,} (cf. \cite{klein2002direct}), the support of the function $r$ is located in its neighbourhood on the phenotype axis. This growth process is regulated by the axons (cf. \cite{biankin2012pancreatic}). 
On the one hand, {the presence of the sympathetic axons inhibits or promotes the growth of cancer cells depending on the sign of $\mu_1$ which is the amplitude of the modulation (cf. \cite{renz2018beta2} and \cite{guillot2022sympathetic})
This mechanism is modelled by the term $-\mu_1 A_1(t)$, if $\mu_1>0$, then the sympathetic axons play an inhibitory role, whereas if $\mu_1<0$, the sympathetic axons play a promoting role on tumor growth. 
On the other hand, the presence of sensory axons amplifies the growth of cancer cells (cf. \cite{banh2020neurons})
, and this is modelled by the term $+\mu_2 A_2(t)$ where $\mu_2 \geq 0$.
Moreover, the saturation rate of the cells densities is linked to the parameters $\tau_C$ (the carrying capacities), $\mu_1$ and $\mu_2$.\\

\noindent{\bf Sympathetic axons growth dynamics.} 
One interesting dynamics for the sympathetic axons is that a small increase of the density of these axons is observed early in precancerous stage
of cancer progression process, whereas once cancer is established
the density of sympathetic axons is reduced (see the biological data in Figure \ref{fig:boxplot_syaxons} in Section \ref{sec:data}). 
Hence, in order to model the time evolution of the sympathetic axons, we use a logistic law with an Allee effect.
We denote by $\theta$ the function that enables the Allee effect.
The function $\theta$ takes as argument the ratio ${N(t)}/{N(0)}$ and it allows to change the dynamics of the sympathetic axons evolution during disease progression.
Starting from the healthy state, $N(t)=N(0)$ at least from small time $t$ i.e. the total concentration of cells remains constant, then taking $\theta(1)<A_1(0)$ and the amount of sympathetic axons increases. Moreover, $\theta$ is an increasing function such that if $N(t)>N(0)$ i.e. there are proliferative cells in the system, then at some point $\theta\left({N(t)}/{N(0)} \right)>1$ and the amount of sympathetic axons decreases. Thus the dynamics of the sympathetic axons are given by the following differential equation:
$$
\dt A_1(t) =  \underbrace{ r_{A_1} A_1(t)\left(\frac{A_1(t)}{\theta\left(\frac{N(t)}{N(0)}\right)}-1 \right)(1-A_1(t))}_{\text{Logistic law with Allee effect}},\\
$$
{where $r_{A_1}>0$ is the growth rate of sympathetic axons.}\\

\noindent{\bf Sensory axons growth dynamics.} {The remodeling of the sensory axons start at cancerous stages 
(see the biological data in Figure \ref{fig:boxplot_seaxons} in Section \ref{sec:data}). Hence, one natural way to model the time evolution of sensory axons is to use a logistic law with a growth rate which is modulated by the presence of cells with non-negative phenotype values. 
We denote $r_{A_2}$ the increasing function taking as argument ${\NC(t)}/{N(t)}$ which model the growth rate of sensory axons.
The dynamics of the sensory axons are given by the following differential equation: }
$$\dt A_2(t) = \underbrace{r_{A_2}\left(\frac{\NC(t)}{N(t)} \right)}_{\text{Impact of cancer cells on growth}} \underbrace{A_2(t)(1-A_2(t))}_\text{Logistic law}. $$
For instance, if there is no proliferative cells, then $$r_{A_2}\left(\frac{\NC(t)}{N(t)} =0 \right)=0$$ 
and the amount of sensory axons remains constant. As soon as ${\NC(t)}/{N(t)} >0 $ which implies $ r_{A_2}\left({\NC(t)}/{N(t)} \right)>0$ then the sensory axons follow the logistic law. Since we assume that $r_{A_2}$ is monotonous, then the more cancer cells there are in the system, the faster the growth of sensory axon is.\\

\noindent{\bf Complete dynamical system.} {The system that mathematically formalizes the impact of axons on tumor progression consists in a partial differential equation (PDE) for cell dynamics and two differential equations for axon dynamics. It couples a growth-transport equation for $Q$ with two ODEs with non-local terms for $A_1$ and $A_2$:}
\begin{equation}
\left\{
\begin{array}{c@{}l}
\partial_t Q(t,x)  &+ \partial_x \big[f(t,x;{\mathcal X})Q(t,x) \big] = g(t,x;{\mathcal X}),\, t>0,\, x\in \Omega\\
\\
\dt A_1(t) &=  r_{A_1} A_1(t)\left(\frac{A_1(t)}{\theta\left(\frac{N(t)}{N(0)}\right)}-1 \right)(1-A_1(t)),\, t>0,\\
\\
\dt A_2(t) &= r_{A_2}\left( \frac{\NC (t)}{N(t)}\right) A_2(t) \left(1-A_2(t)\right),\, t>0.
\end{array}
\right.
\label{eq:model_continuous}
\end{equation}

\noindent The system \eqref{eq:model_continuous} is completed with the following initial conditions:
$$Q(0,x)=Q_0(x)\text{ for } x\in \Omega, \quad A_1(0)\cong A_1^{eq}\quad \text{and}\quad A_2(0)=A_2^0,$$
where $A_1^{eq}>0$ corresponds to the average density of sympathetic axons in a healthy pancreas which has been normalized and $0 < A_2^0 \ll 1$ corresponds to the average density a sensory axons in a healthy pancreas also normalized. 
For modeling purposes, if there is no observations showing the presence of sensory axons in a healthy pancreas, we still consider that there is a negligible amount of sensory axons in the system at the initial state. Otherwise, the growth dynamics could not take place since $A_2(t)=0$ is an unstable steady-state.  
One point to note is that if we consider initially that the pancreas is essentially composed of healthy cells then one way to cope with this assumptions is to assume that $Q_0$ is compactly supported and that the following holds on the support of the initial datum $supp(Q_0)\subset \Omega /\mathbb{R}_+$.
Finally, we add the following boundary condition on the PDE
$$f(t,x;\X)Q(t,x) =0, x\in\partial \Omega,  $$
which ensures Dirichlet boundary conditions and allows us to neglect any processes acting outside of our domain. 
{This boundary condition is a strong assumption which is sufficient to prove the well-posedness of the model and the fact that no mass is lost on the boundary of the domain. It is ensured by the compact support of the transport term $f$ (cf. Hypothesis \ref{hyp:cond_advec_wellpo}).}
\subsection{Well-posedness of the model}\label{sec:wellposedness}
We denote by $\mathcal{X}\defeq (Q,A_1,A_2)$ the solution of the system \eqref{eq:model_continuous}.
{We note that $\mathcal{X}$ is solution of a non-conservative system that is a particular case of the following system :}
\begin{subequations}
\label{eq:model_abstract}
\begin{empheq}[left = {\empheqlbrace\,}]{align}
& \partial_t Q(t,x)  +\, f(t,x;\mathcal{X})\partial_x Q(t,x) + c(t,x;\mathcal{X})Q(t,x) = 0 &\quad \text{for } t>0,\, x\in \Omega,\label{eq:mod_abs_pde}\\
& f(t,x;\mathcal{X})Q(t,x) =  0 &\quad \text{for } x\in\partial\Omega, \,t\geq 0,\label{eq:mod_abs_bdd}\\
& Q(0,x) = Q_0(x) & \quad \text{for } x \in \Omega,\label{eq:mod_abs_Q0}\\
& \dt A_1(t)  =  r_{A_1} A_1(t)\left(\frac{A_1(t)}{\theta\left(\tfrac{N(t)}{N(0)}\right)}-1 \right)(1-A_1(t)) &\quad  \text{for } t > 0,\label{eq:mod_abs_a1} \\
& A_1(0) = A_1^0 \in (0,1), & \label{eq:mod_abs_a10}\\
& \dt A_2(t) = r_{A_2} \left(\frac{\NC (t)}{N(t)}\right) A_2(t) \left(1-A_2(t)\right) &\quad \text{for } t > 0,\label{eq:mod_abs_a2}\\
& A_2(0) = A_2^0 \in (0,1), & \label{eq:mod_abs_a20}
\end{empheq}
\end{subequations}
where $N(t) \defeq \int_\Omega Q(t,x)dx $ and $\NC (t) \defeq \int_\Omega \psi(x)Q(t,x)dx $ such that $\psi$ is given and nonnegative and {$\text{supp}(\psi)\subset \Omega\cap \mathbb{R}_+$ and that $\|\psi Q(t)\|_{\mathcal{L}^1(\Omega)}\leq \|Q(t)\|_{\mathcal{L}^1(\Omega)}$ for $t\geq 0$}. We also introduce the function $g$ which allows to rewrite the equation \eqref{eq:mod_abs_pde} in a conservative form :
\begin{equation}
g(t,x;\mathcal{X}) \defeq c(t,x;\mathcal{X}) - \partial_x f(t,x;\mathcal{X}),
\label{eq:model_abstract_g}
\end{equation}
implying that
$$ \partial_t Q(t,x)  +\, \partial_x \big(f(t,x;\mathcal{X}) Q(t,x)\big) + g(t,x;\mathcal{X})Q(t,x) = 0 \quad \text{for } t>0,\, x\in \Omega. $$
In the following, we {study the well-posedness of system \eqref{eq:model_abstract} and start with the different hypotheses required for that. 
}
\begin{hyp}[Initial condition]
Assume $Q_0 \in \mathcal{C}^1 (\Omega) $ is nonnegative such that {
$$0<C(N(0))<\int_\Omega Q_0(x) dx  = \| Q_0 \|_{\mathcal{L}^1(\Omega)}< C(\tau_C) <\infty $$
with $C(N(0))$ and $C(\tau_C)$ positive constants.}
\label{hyp:initial_cond_wellpo}
\end{hyp}
Hence, we denote $\mathcal{P}$ the following set:
\begin{equation}
\begin{split}
\mathcal{P}(T)\defeq &\bigg\{(Q, A_1, A_2)\in \mathcal{C}^1\left([0,T];\mathcal{C}^1(\Omega)\cap\mathcal{L}^1(\Omega)\right) \times \mathcal{C}^1([0,T]) \times \mathcal{C}^1([0,T]) \text{ such that }\\
& C(N(0))\leq \int_\Omega Q(t,x)dx\leq C(\tau_C) \text{ and } 0\leq A_i(t)\leq 1,\; i =1,2,\; t\in[0,T] \bigg\}.
\end{split}
\label{eq:admissible_set}
\end{equation}
We also assume the following conditions on the functions given by the transport term and the growth term. Given $\mathcal{X} \in \mathcal{P}(T), \text{ for } T>0$, we have the following hypotheses :
\begin{hyp}[Transport term]
Assume that $\mathcal{X}_i \in \mathcal{P}(T), \text{ for } T>0$ for $i=1,2$ then
$$f(\cdot ; \mathcal{X}) \in \mathcal{C}\left([0,T]; \mathcal{C}_c^1 (\Omega)\right) \text{ with }|\partial_x f(t,x ; \mathcal{X})|<\infty,\; f(t,x ; \mathcal{X})\geq 0$$
{$\text{for } t\in[0,T] \text{ and } x\in \Omega$, 
$$ f(t,x;\X) = \partial_x f(t,x;\X) = 0 \text{ for } t\in[0,T] \text{ and } x\in \partial\Omega$$
and there exist constants $C_l(f)>0$, $C_l(\partial_x f)>0$ such that  
$$\|f(t;\X_1) - f(t;\X_2) \|_{\mathcal{L}^\infty(\Omega)} \leq C_l(f) \|\X_1 -\X_2\|_\mathcal{P}\quad \text{and} \quad \|\partial_x (f(t;\X_1) - f(t;\X_2)) \|_{\mathcal{L}^\infty(\Omega)} \leq C_l(\partial_x f) \|\X_1 -\X_2\|_\mathcal{P}.$$}
\label{hyp:cond_advec_wellpo}
\end{hyp}
\begin{hyp}[Growth term]
{Assume that $\mathcal{X}_i \in \mathcal{P}(T), \text{ for } T>0$ for $i=1,2$ then
$$g(\cdot ; \mathcal{X})\in \mathcal{C}\left([0,T]; \mathcal{C}_c^1 (\Omega)\right) \text{ with }  \; |g(t,x;\X)|<\infty, \; |\partial_x g(t,x;\X)|<\infty  \text{ for } t\in[0,T] \text{ and } x\in \Omega$$
and there exists a constant $C_l(g)>0$ such that  
$$\|g(t;\X_1) - g(t;\X_2) \|_{\mathcal{L}^\infty(\Omega)} \leq C_l(g) \|\X_1 -\X_2\|_\mathcal{P}.$$
Moreover, assume that the growth term is a logistic-type growth term, i.e
$$-g(t,x;\X) \defeq r(x)\left(h\big(A_1(t),A_2(t)\big)-\int_\Omega Q(t,y)dy \right)$$
where there exist constants $0\leq r_{-} < r_{+} <\infty$ and $C(\tau_C)>0$ such that 
$$ r_{-} \leq r(x) \leq r_{+}\quad \text{ for } x\in\Omega$$
and there exists a small perturbation $0<\varepsilon$ such that 
$$C(\tau_C) - \varepsilon \leq h(x,y)\leq C(\tau_C)\quad  \text{ for } (x,y) \in [0,1]^2. $$ 
}
\label{hyp:cond_growth_wellpo}
\end{hyp}
\begin{hyp}[Coupling axons and cell densities]
Assume $\theta \in \mathcal{C}(\mathbb{R_+})$ and $r_{A_2} \in \mathcal{C}(\mathbb{R_+})$ such that $$ r_{A_2}(x) \geq 0, \quad \forall x\in \mathbb{R} $$
and that for any compact set $K\subsetneq \mathbb{R_+}$, there exist a constant $C_l(\theta)>0$ and a constant $C_l(r_{A_2})>0$ such that $\forall (x,y)\in K\times K$
$$|\theta(x) -\theta(y)|\leq C_l(\theta)|x-y| \quad \text{ and } \quad |r_{A_2}(x) -r_{A_2}(y)|\leq C_l(r_{A_2})|x-y|.$$
{Moreover, the image of $\theta$ satisfies
$$\text{Im}(\theta) \subset [\theta_{-}, \theta_{+}] \text{ with } \theta_{-} \in (0,A_1^0), \, \theta_{+}>1. $$
}
\label{hyp:cond_wellpo_theta_r}
\end{hyp}
{
\begin{remark}
The Hypothesis \ref{hyp:initial_cond_wellpo} allows to consider an initial state of healthy cells for the system. As for Hypotheses \ref{hyp:cond_advec_wellpo} and \ref{hyp:cond_growth_wellpo}, they give sufficient conditions on the tumor progression term (the transport term $f$) and on the proliferation term (the reaction term $c$) in order to ensure the well-posedness. The Hypothesis \ref{hyp:cond_wellpo_theta_r} states some regularity assumptions on the terms formalizing the effect of the cells on the axons densities as well as additional conditions in order to have a coupling term biologically relevant (cf. Section \ref{sec:model}).
\end{remark}
\begin{remark}
In our model, $$ f(t,x;\mathcal{X}) = \pi(x)  (1- \beta(x)\rho(A_1(t)) + \delta(x)A_2(t)) +\eta\left(x,\tfrac{\NC (t)}{N(t)}\right) $$
and
$$c(t,x;\mathcal{X}) =  \partial_x f(t,x;\mathcal{X}) - r(x)\left(1- \frac{N(t)}{\tau_C} - \mu_1 A_1(t)+\mu_2A_2(t)\right).$$
Hence the conditions on $f$ and $c$ stated in Hypotheses \ref{hyp:cond_advec_wellpo} and \ref{hyp:cond_growth_wellpo} are enforced when the following detailed assumptions are fulfilled : 
\begin{itemize}
    \item Let $\pi\in \mathcal{C}_c^1 (\Omega) $, $\beta \in \mathcal{C}^1 (\Omega)$, $\delta \in \mathcal{C}^1 (\Omega)$ and $\eta\left( \cdot, \tfrac{N_C(t)}{N(t)}\right) \in \mathcal{C}_c^1 (\Omega)$ such that 
$$\sup\limits_{x\in \Omega} |\pi'(x)| + |\beta'(x)| + |\delta'(x)| + |\eta'(x)| \leq C $$
for a constant $C>0$.
    \item Let $r\in \mathcal{C}_c^1(\Omega)$ and Lipschitz continuous. 
\end{itemize}
More details about these functions will be given in Section \ref{sec:precision}.
\end{remark}
}
\begin{theorem}\label{thm:well_po}
Assume that the Hypotheses \ref{hyp:initial_cond_wellpo}, \ref{hyp:cond_advec_wellpo}, \ref{hyp:cond_growth_wellpo} and \ref{hyp:cond_wellpo_theta_r} are satisfied, then the system \eqref{eq:model_abstract} admits a unique solution $$(Q,\; A_1,\; A_2) \in \mathcal{C}^1([0,T];\mathcal{L}^1(\Omega)) \times \mathcal{C}^1([0,T]) \times \mathcal{C}^1([0,T]), \text{ for all } T>0$$
such that $$\forall s\in [0,T] \quad  0\leq A_i(s) \leq 1, \; i=1,2, $$ 
and there exists a positive constant $C(\tau_C)>0$
$$\forall s\in [0,T] \quad \int_\Omega Q_0(x) dx \leq \int_\Omega Q(s,x) dx \leq C(\tau_C).  $$
\end{theorem}
\paragraph{Sketch of the proof :} The details of the proof of Theorem \ref{thm:well_po} are postponed in the Appendix \ref{sec:app_wellpo}. The system \eqref{eq:model_abstract} couples an transport--reaction partial differential equation (PDE) for the cancer progression with two differential equations for the axons densities. Axons dynamic is governed by non-local terms that depend on the solution of the PDE. Inversely, the dynamic of tumor progression modelled by the transport term and the reaction term is also governed by the axons. The proof of the well-posedness of the solution of this non-linear system relies on the contraction mapping theorem.\\
Consider $\X =(Q,A_1,A_2) \in \mathcal{C}^1([0,T];\mathcal{C}^1(\Omega)\cap\mathcal{L}^1(\Omega))\times \mathcal{C}^1([0,T])\times \mathcal{C}^1([0,T])$ given. {Define $N(s) \defeq \int_\Omega Q(s,x)dx $ and $\NC (s) \defeq \int_\Omega \psi(x)Q(s,x)dx $} such that $0 < N(0)\leq N(s) \leq C(\tau_C)$ for $ s\in [0,T]$. We introduce the linear system \eqref{eq:model_abstract}:
\begin{equation}
\left\lbrace
\begin{array}{c@{}l@{}l}
\partial_s u(s,x) & +\, f(s,x;\X)\partial_x u(s,x) + c(s,x;\X)u(s,x) = 0,\,&\, \text{for } s\in(0, T),\, x\in \Omega,\\
\\
f(s,x)u(s,x)& =  0,&\, \text{for } x \in \partial\Omega,\; s\in(0, T),\\
\\
u(0,x) &= Q_0(x), &\, \text{for } x \in \Omega,\\
\\
\ds \widetilde{A_1}(s) & =  r_{A_1} \widetilde{A_1}(s)\left(\frac{\widetilde{A_1}(s)}{\theta\left(\tfrac{N(s)}{N(0)}\right)}-1 \right)(1-\widetilde{A_1}(s)) , & \, \widetilde{A_1}(0) = A_1^0,\\
\\
\ds \widetilde{A_2}(s) & = r_{A_2} \left(\frac{N_c(s)}{N(s)}\right)\widetilde{A_2}(s) \left(1-\widetilde{A_2}(s)\right),& \, \widetilde{A_2}(0) = A_2^0.
\end{array}
\right.
\label{eq:model_abstract_linear}
\end{equation}
We denote 
$$\mathcal{B} \defeq \left\lbrace Q \in \mathcal{C}^1([0,T];\mathcal{C}^1(\Omega)\cap\mathcal{L}^1(\Omega))  \, | \, \int_\Omega Q_0(x)dx \leq \int_\Omega Q(s,x)dx \leq C(\tau_C) \, \text{for} \, s\in[0,T]  \right\rbrace$$ 
and $\mathcal{S}$ the following mapping :
\begin{equation}
\mathcal{S} :\, \mathcal{B}\times \mathcal{C}^1([0,T])\times \mathcal{C}^1([0,T]) \rightarrow \mathcal{B}\times \mathcal{C}^1([0,T])\times \mathcal{C}^1([0,T]), \quad (Q, A_1, A_2) \mapsto (u, \widetilde{A_1}, \widetilde{A_2}).
\label{eq:mapping_contraction}
\end{equation} 
Hence, we note that the solution of \eqref{eq:model_abstract} is a fixed point of $\mathcal{S}$.
The Lemmas \ref{lem:logistic_allee}, \ref{lem:logistic} and \ref{lem:wellpo_linear_pb}  in Appendix \ref{sec:app_wellpo} prove the well-posedness of the solution of \eqref{eq:model_abstract_linear} and give the estimations needed to prove the contraction property of the mapping which implies the existence and uniqueness of the solution of \eqref{eq:model_abstract}. 
\begin{remark}
The Theorem \ref{thm:well_po} state the well-posedness on any finite time interval of the solution of \eqref{eq:model_abstract}. However, one can prove the global well-posedness with the same arguments assuming additionnal regularity assumptions (on the time variable $t\in \mathbb{R}_+$ instead of $t\in[0,T]$) on the functions $\theta$ and $r_{A_2}$ (e.g. no singularities). The main reason allowing the extension of the result about the well-posedness comes from the fact that assuming the Hypothesis \ref{hyp:initial_cond_wellpo} holds and $0<A_i^0<1 $ for $i=1,2$ then the trajectories $(\int_\Omega Q(\cdot,x)dx, A_1(t),A_2(t))$ remain bounded away from $0$. However, finer stability estimates are required to conclude. Since this study focuses on the the interactions between cells and axons during the tumorigenesis for the pancreatic cancer, the transient behaviour (instead of the asymptotic behaviour) is the main point of interest. The result of Theorem \ref{thm:well_po} is sufficient to pursue this goal. 
\end{remark}

\subsection{Qualitative study of the system}\label{sec:qualitative_study}
{The aim of this section is to study how the transport term \eqref{eq:transport_term_pde} and the proliferation term \eqref{eq:growth_term_pde} of the system \eqref{eq:model_abstract} impact the final state of the model's dynamics.   
First, we show that, depending on the assumptions made about the functions $\pi$ (the basal speed of disease progression), $\eta$ (the speed of disease progression due to the presence of cancer cells), $r$ (the basal proliferation rate of cancer cells) and $Q_0$ (the initial distribution of cells with respect to their phenotypes), the different behaviours of the model's dynamics can be summed up in three categories.
Secondly,  we obtain explicit bounds on the first time of appearance of cells with cancerous phenotypes. This result is obtained under some assumptions on the transport term and relies on the characteristic curves of the equation. 

\subsubsection{Behaviour of the model}
{We recall that our phenotype domain $\Omega \subset \mathbb{R}$ can either be the whole real line or a segment centered in $0$. Hence, a distinct separation for the cells with positive and negative phenotypes is assumed. 
The origin separates the cancer-induced proliferative cells from the healthy
one. 
Moreover, at time $t=0$ the cells are supposed to be at a non-cancerous and thus at a non-proliferating state. This assumption is formally translated in our model by the following hypothesis on the location on the phenotype axis of the initial distribution $Q_0$ (i.e. its support) and on the support of the proliferation rate $r$ :} 
\begin{hyp}
Assume $Q_0$ is the initial condition of \eqref{eq:model_abstract} which fulfills Hypothesis \ref{hyp:initial_cond_wellpo}. Moreover, let $Q_0$ be compactly supported and $$\text{supp}(Q_0)\bigcap \text{supp}(r) = \emptyset.$$ 
\label{hyp:supp_Q0_r}
\end{hyp}
\noindent Hypothesis \ref{hyp:supp_Q0_r} implies that the population of cells in a healthy pancreas remain constant until the apparition of cancer cells.
Moreover, we also assume the following :
\begin{hyp}
Let 
$$\pi(x)\geq 0, \quad \eta(x) \geq 0, \quad \delta(x) \geq 0 \quad \text{and} \quad 1-\beta(x)\rho(y) \geq 0 \quad \forall x \in \Omega,\; \forall y \in [0,1]. $$
\label{hyp:pos_advec}
\end{hyp}
\noindent Hypothesis \ref{hyp:pos_advec} implies that the transport term $f(\cdot;\X)$ is nonnegative. It ensures the modeling assumptions stating that the phenotypic transfer is unidirectional: from healthy to cancer cells.
The immediate consequence of this assumption is that the behaviour of the system is mainly dictated by the location on the phenotype axis of the transport term support (i.e. the supports of $\pi$ and $\epsilon$).\\
 Once the Hypotheses \ref{hyp:supp_Q0_r} and \ref{hyp:pos_advec} are assumed, {the model dynamics can exhibit several types of behaviour that can be associated to one of the following states: }
\begin{itemize}
    \item a \textit{stationary state }where there is no cancer progression,
    \item a\textit{ bimodal state}, where the cancer progression takes place but the system converges towards a bimodal distribution (i.e. an asymptotic state in which healthy cells and cancer cells are present),
        \item a \textit{pathological state} where the healthy cells are totally replaced by cancer cells.
\end{itemize}
These three behaviours are determined by more refined  assumptions on the functions involved in the system \mm{\eqref{eq:model_abstract}}. The assumptions categorizing the dynamical behaviour are summarized in Figure \ref{fig:supports}.
}
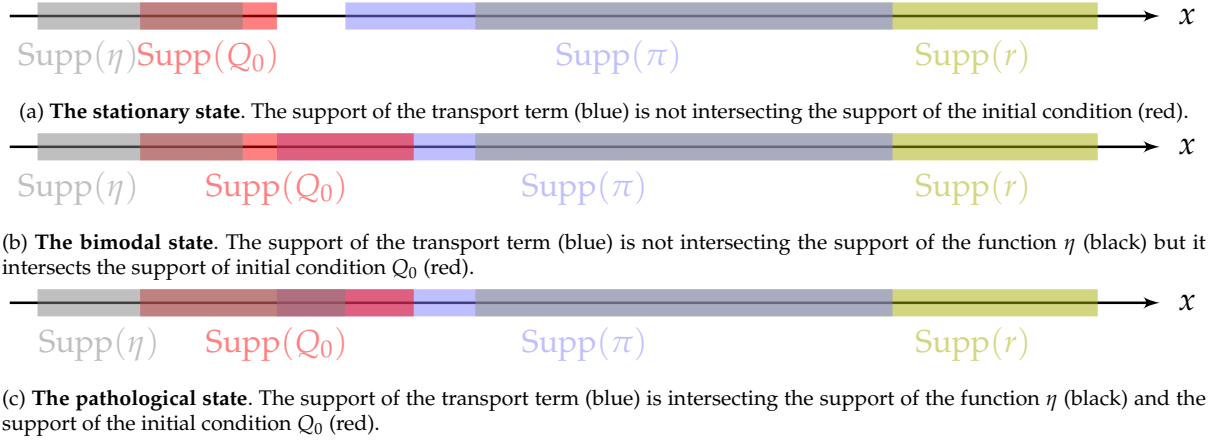
\begin{figure}[ht!]
\centering
     \begin{subfigure}[b]{0.995\textwidth}
         \centering
\scalebox{1.2}{\begin{tikzpicture}[scale=0.75,
                    fleche/.style={->, >=latex', shorten >=2pt, shorten <=2pt, thick},
                    doublefleche/.style={<->, >=latex', shorten >=2pt, shorten <=2pt, thick},
                    pi/.style={-,line width =3mm,blue!50,opacity=0.5},
                    QO/.style={-,line width =3mm, red,opacity=0.5},
                    eta/.style={line width =3mm, black!50,opacity=0.5},
                    r/.style={line width =3mm, couleur_r,opacity=0.5},
                    caract/.style={dashdotted, very thick},
                    ]

\draw[fleche] (-7,0)--(10,0) node[right] {$x$};

\draw[r] (-0.1,0.0)--(9,0.0) node[below,pos=0.8] {${\rm Supp} (r)$};
\draw[pi] (-2,0.0)--(6,0.0) node[below,midway] {${\rm Supp} (\pi)$};
\draw[QO] (-5,0)--(-3,0)node[below,midway] {${\rm Supp} (Q_0)$};
\draw[eta] (-6.5,0.0)--(-3.5,0.0)node[below,pos=0.2] {${\rm Supp} (\eta)$};

\end{tikzpicture}
}
    \caption{\textbf{The stationary state}. The support of the transport term (blue) is not intersecting the support of the initial condition (red).}\label{fig:schem_supp_stationary}
     \end{subfigure}
     \hfill
     \begin{subfigure}[b]{0.995\textwidth}
         \centering
         \scalebox{1.2}{\begin{tikzpicture}[scale=0.75,
                    fleche/.style={->, >=latex', shorten >=2pt, shorten <=2pt, thick},
                    doublefleche/.style={<->, >=latex', shorten >=2pt, shorten <=2pt, thick},
                    pi/.style={-,line width =3mm,blue!50,opacity=0.5},
                    QO/.style={-,line width =3mm, red,opacity=0.5},
                    eta/.style={line width =3mm, black!50,opacity=0.5},
                    r/.style={line width =3mm, couleur_r,opacity=0.5},
                    caract/.style={dashdotted, very thick},
                    ]

\draw[fleche] (-7,0)--(10,0) node[right] {$x$};

\draw[r] (-0.1,0.0)--(9,0.0) node[below,pos=0.8] {${\rm Supp} (r)$};
\draw[pi] (-3,0.0)--(6,0.0) node[below,midway] {${\rm Supp} (\pi)$};
\draw[QO] (-5,0)--(-1,0)node[below,midway] {${\rm Supp} (Q_0)$};
\draw[eta] (-6.5,0.0)--(-3.5,0.0)node[below,pos=0.2] {${\rm Supp} (\eta)$};

\end{tikzpicture}
}

    \caption{\textbf{The bimodal state}. The support of the transport term (blue) is not intersecting the support of the function $\eta$ (black) but it intersects the support of initial condition $Q_0$ (red).}\label{fig:schem_supp_bimodal}
    \end{subfigure}
    \hfill
    \begin{subfigure}[b]{0.995\textwidth}
         \centering
        \scalebox{1.2}{\begin{tikzpicture}[scale=0.75,
                    fleche/.style={->, >=latex', shorten >=2pt, shorten <=2pt, thick},
                    doublefleche/.style={<->, >=latex', shorten >=2pt, shorten <=2pt, thick},
                    pi/.style={-,line width =3mm,blue!50,opacity=0.5},
                    QO/.style={-,line width =3mm, red,opacity=0.5},
                    eta/.style={line width =3mm, black!50,opacity=0.5},
                    r/.style={line width =3mm, couleur_r,opacity=0.5},
                    caract/.style={dashdotted, very thick},
                    ]

\draw[fleche] (-7,0)--(10,0) node[right] {$x$};

\draw[r] (-0.1,0.0)--(9,0.0) node[below,pos=0.8] {${\rm Supp} (r)$};
\draw[pi] (-3,0.0)--(6,0.0) node[below,midway] {${\rm Supp} (\pi)$};
\draw[QO] (-5,0)--(-1,0)node[below,midway] {${\rm Supp} (Q_0)$};
\draw[eta] (-6.5,0.0)--(-2,0.0)node[below,pos=0.2] {${\rm Supp} (\eta)$};

\end{tikzpicture}
}
    \caption{\textbf{The pathological state}. The support of the transport term (blue) is intersecting the support of the function $\eta$ (black) and the support of the initial condition $Q_0$ (red).}\label{fig:schem_supp_pathological}
    \end{subfigure}
\caption{Scheme of the supports of the functions governing the dynamics of the cancer cells progression}\label{fig:supports}
\end{figure}

\begin{figure}[ht!]
     \centering
     \begin{subfigure}[b]{0.495\textwidth}
         \centering
         \includegraphics[width=\textwidth]{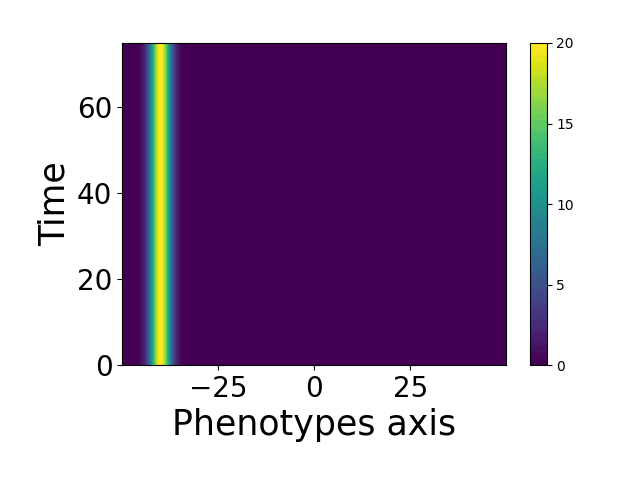}
         \caption{{\textbf{The stationary state.}} Evolution of the cell densities}
         \label{fig:heatmap_stationary}
     \end{subfigure}
     \hfill
     \begin{subfigure}[b]{0.495\textwidth}
         \centering
		 \includegraphics[width=\textwidth]{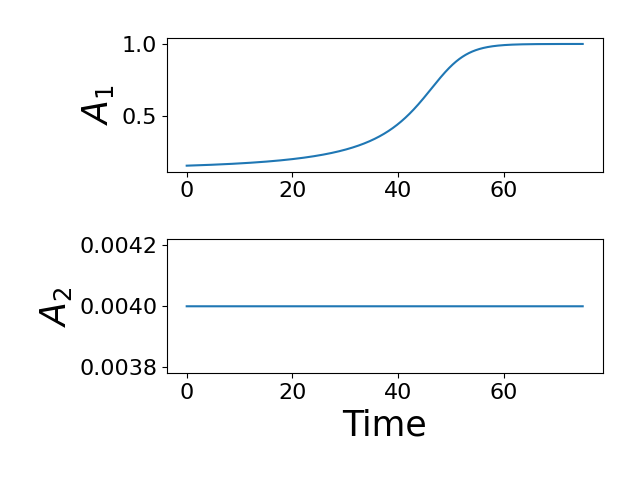}
         \caption{{\textbf{The stationary state.}} Evolution of the axons}
         \label{fig:axons_stationary}
     \end{subfigure}
     \centering
     \begin{subfigure}[b]{0.495\textwidth}
         \centering
         \includegraphics[width=\textwidth]{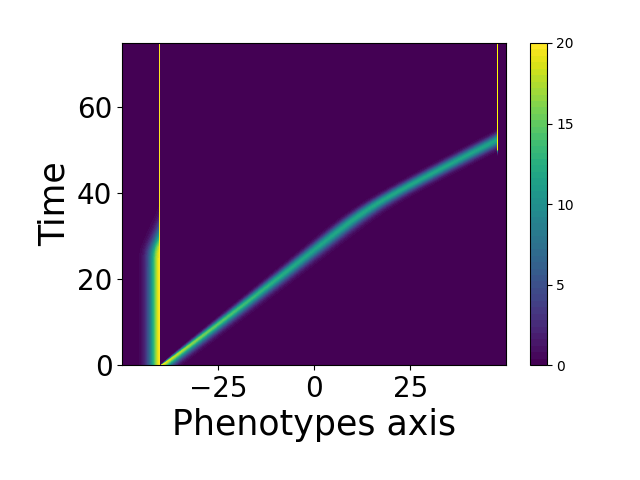}
         \caption{\textbf{The bimodal state. }Evolution of the cell densities}
         \label{fig:heatmap_bimodal}
     \end{subfigure}
     \hfill
     \begin{subfigure}[b]{0.495\textwidth}
         \centering
		 \includegraphics[width=\textwidth]{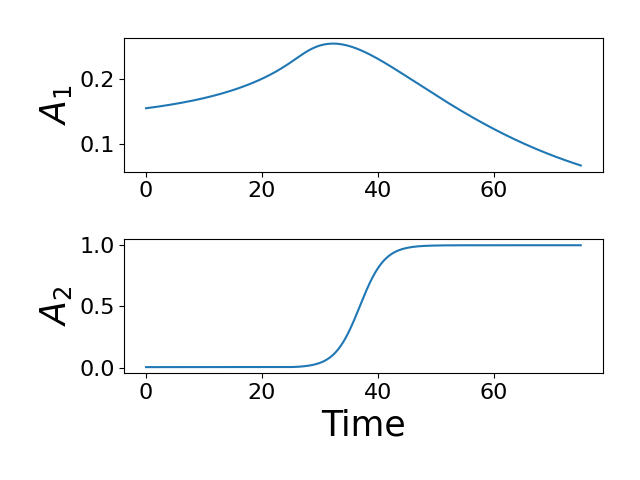}
         \caption{\textbf{The bimodal state. }Evolution of the axons}
         \label{fig:axons_bimodal}
     \end{subfigure}
     \centering
     \begin{subfigure}[b]{0.495\textwidth}
         \centering
         \includegraphics[width=\textwidth]{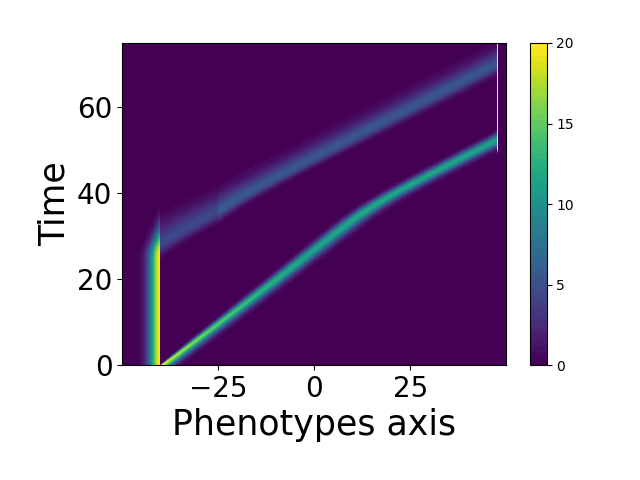}
         \caption{\textbf{The pathological state. }Evolution of the cell densities}
         \label{fig:heatmap_pathological}
     \end{subfigure}
     \hfill
     \begin{subfigure}[b]{0.495\textwidth}
         \centering
		 \includegraphics[width=\textwidth]{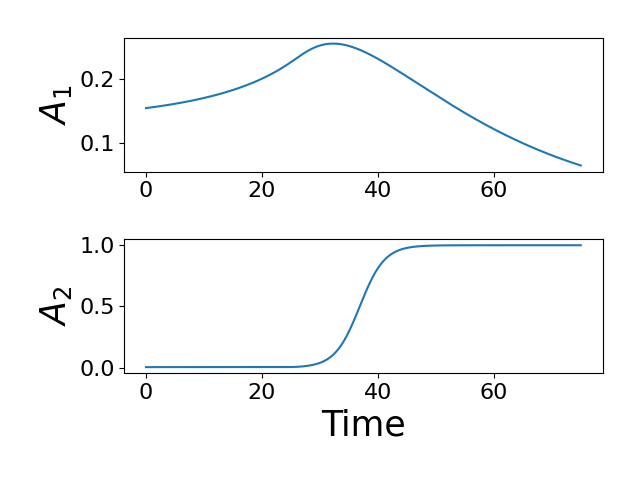}
         \caption{\textbf{The pathological state.} Evolution of the axons}
         \label{fig:axons_pathological}
     \end{subfigure}
     \caption{Numerical simulations of the system \eqref{eq:model_abstract} (cf. Appendix \ref{app:numerics} for the details of the numerics). The subfigures \ref{fig:heatmap_stationary}, \ref{fig:heatmap_bimodal} and \ref{fig:axons_pathological} show in the form of a heatmap the evolution of cell populations over time as a function of their phenotypes. The vertical axis is the time and the horizontal axis is the phenotype. The color map is an indicator of the cells densities. The subfigures \ref{fig:axons_stationary}, \ref{fig:axons_bimodal} and \ref{fig:axons_pathological} describe the evolution of the {normalized densities} of sympathetic axons denoted $A_1$ and sensory axons denoted $A_2$. }
    \label{fig:num_stationary}
\end{figure}

\paragraph{The stationary state.} The conditions ensuring this state are schematized in Figure \ref{fig:schem_supp_stationary}. 
Since there is no intersection between the support of the initial condition and the support of the transport term, there is no possible progression of the cancer. This behaviour corresponds to the dynamics of the model when there are only healthy cells without proliferative cells (cf. Figure \ref{fig:heatmap_stationary}). {Moreover, since there is no proliferative cells, the sensory axons density remains constant and the sympathetic axons density stays constant or follows a logistic growth (cf. Figure \ref{fig:axons_stationary}).} 

\paragraph{The bimodal state.} The conditions ensuring the bimodal state are schematized in Figure \ref{fig:schem_supp_bimodal}. 
The support of the initial density of cells on the phenotype axis is intersecting with the support of the functions implied in the progression towards the cancerous phenotype (the functions $\pi$ and $\eta$ in \eqref{eq:transport_term_pde}). {The function $\pi$ transports a proportion of the cell densities toward the cancer phenotype whereas the function $\eta$ starts accelerating the transport of the cell densities from the healthy state to the cancerous state only when cancer cells are present in the system. } 
Nevertheless, in the bimodal state, the supports of $\pi$ and $\eta$ are not intersecting each other. In that case, the system dynamics tend towards a bimodal population distribution.
The bimodal state corresponds to a state where there are cancer cells in the system after some time . However a proportion of non-cancerous cells is also present and persists through time (cf. Figures \ref{fig:heatmap_bimodal} and \ref{fig:scheme_bimodal}). 
Concerning the axons dynamics, in contrast to the stationary case, the sympathetic axons are first increasing thanks to the logistic growth and then decreasing thanks to the Allee effect (cf. Figure \ref{fig:axons_bimodal}) and the sensory axons are following a logistic growth (since there are cancer cells in the system).

\paragraph{The pathological state.} The asymptotic behaviour corresponding to the pathological case is the opposite of the one from the stationary case. After a finite time, the cancer cells are present in the system and a progression towards cancerous phenotypes occurs for the remaining healthy cells (cf. Figure \ref{fig:heatmap_pathological}). Then, after some time, there are only cancer cells in the system. 
If a proportion of healthy cells remains after the start of the cancer progression (i.e. $\text{supp}(Q_0)\cap \text{supp}(\pi)\neq \emptyset$ and $\text{supp}(\eta)/ \text{supp}(\pi)\neq \emptyset $), a  second transport phenomenon toward the cancer phenotypes takes place because the supports of the function $\pi $ and the function $ \eta $ are intersecting each other (cf. Figures \ref{fig:schem_supp_pathological}, \ref{fig:heatmap_pathological} and \ref{fig:scheme_pathological}). This means that in this state, the cancer cells play a more prominent role in the cancer progression. {Also, the cancer cell proliferation is controlled by the function $r$, its support on the phenotype axis and the various parameters involved in the saturation phenomena. In this case, since all cells are cancer cells after a finite time, the total amount of cancer cells is higher as compared to the one in the bimodal state (cf. Figures \ref{fig:heatmap_bimodal} and \ref{fig:heatmap_pathological})}.  However, the axons dynamics in the bimodal state and the pathological state are similar since in both state cancer cells are present (cf. Figure \ref{fig:axons_pathological}).  

\subsubsection{Explicit bounds on the time to appearance of the first cancer cells}
To obtain explicit bounds on the time to appearance of the first cancer cells, we state the following hypothesis which hold either on the bimodal case or the pathological case.
\begin{hyp}
Let  $$\text{supp}(Q_0)\bigcap \text{supp}(\pi) \neq \emptyset, \quad \text{supp}(r)\bigcap \text{supp}(\pi) \neq \emptyset$$
and $$\pi:K\subsetneq \big(\text{supp}(Q_0)\bigcap \text{supp}(\pi)\big)/\text{supp}(r) \rightarrow \mathbb{R}_+^*,\quad x \mapsto \pi(x)  $$ be monotonous and increasing.
\label{hyp:supp_Q0_pi}
\end{hyp}
Hypothesis \ref{hyp:supp_Q0_pi} is stating that at least a proportion of cells from the initial condition progress toward the cancerous phenotype and that the cells can reach a proliferative state. 
Moreover, it also state that the function $\pi$ involved in the transport term is positive and increasing in a specific subset of its definition domain. 
The second statement ensures that the progression toward cancer cells happens in finite time and that progression is faster when the location on the phenotype axis is towards the pre-cancerous and cancerous states than when it is towards the healthy state. Hence, we consider the time $t^*$, the time to appearance of the first cancer cells, with the following definition.
\begin{definition}\label{def:time_prolif}
We denote $t^*$ the first time when cells can proliferate, i.e. 
$$t^* \defeq \inf \left\lbrace  t\in \mathbb{R}_+^* \, | \, \min(\text{supp}(r)) \in \text{supp}(Q(t,\cdot)) \right\rbrace.$$
\end{definition}
\noindent Finally, we obtain an upper bound and a lower bound on the time $t^*$. 
The bounds depend on the assumptions made about the functions in the transport term ($\pi$, $\beta$, $\rho$, $\delta$ in \eqref{eq:transport_term_pde}), those in the growth term ($r$ in \eqref{eq:growth_term_pde}) and the initial condition $Q_0$ (cf. Figures \ref{fig:schem_supp_bimodal} or \ref{fig:schem_supp_pathological}). 
\begin{figure}[ht!]
\centering
     \begin{subfigure}[b]{0.495\textwidth}
         \centering
         \scalebox{0.75}{\begin{tikzpicture}[scale=0.75,
                    fleche/.style={->, >=latex', shorten >=2pt, shorten <=2pt, thick},
                    doublefleche/.style={<->, >=latex', shorten >=2pt, shorten <=2pt, thick},
                    pi/.style={-,line width =3mm,blue!50,opacity=0.5},
                    QO/.style={-,line width =3mm, red,opacity=0.5},
                    eta/.style={line width =3mm, black!50,opacity=0.5},
                    caract/.style={dashdotted, very thick,red},
                    ]
  \draw[red!20,fill=red!20] (-1,4)--(6,9)--(8,9)--(1,4);
\draw[red!20,fill=red!20] (-3,0)-- (-1,4)--(1,4)--(-1,0);

\draw[red!20,fill=red!20] (-5,4)-- (-3.5,4)--(-3.5,5.5);    
\draw[red!20,fill=red!20] (-5,4)-- (-3.5,4) -- (-3.5,0)--(-5,0);
\draw[red!20,fill=red!20] (-3.5,0.1)--(-3,0.1)--(-3,9)--(-3.5,9);

\draw[fleche] (-7,0)--(7,0) node[below] {$x$};
\draw[fleche] (1,-0.1)--(1,10) node[above] {$t$};
\draw[dotted] (-7,4)--(7,4) node[above, pos=0.52] {$t^*$};

\draw[caract] (-5,0)--(-5,4);
\draw[caract] (-3.5,0)--(-3.5,4);
\draw[red, ultra thick] (-3.5,4)--(-3.5,9);
\draw[caract] (-4.25,0)--(-4.25,4);
\draw[caract] (-3,0)--(-3,9);

\draw[caract,red] (-3,0) -- (-1,4);
\draw[caract,red] (-1,4) -- (6,9);
\draw[caract,red] (1,4) -- (8,9);
\draw[caract,red] (-1,0) -- (1,4);

\draw[caract,red] (-5,4)-- (-3.5,5.5);
\draw[caract,red] (-4.25,4) -- (-3.5,4.8);

\draw[pi] (-3,0.0)--(6.7,0.0) node[below,midway] {${\rm Supp} (\pi)$};

\draw[QO] (-5,0)--(-1,0)node[below,midway] {${\rm Supp} (Q_0)$};
\draw[eta] (-6.5,0.0)--(-3.5,0.0)node[below,pos=0.2] {${\rm Supp} (\eta)$};
\draw[fleche] (-0.1,-1.5)--(-1,0) node [below,pos=-0.1]{$M_0$};
\node[draw,red,rotate=65] at (-0.75,2.4)  {${\rm Supp} (Q(t,.))$};
\end{tikzpicture}}
         \caption{Bimodal state}
         \label{fig:scheme_bimodal}
     \end{subfigure}
     \hfill
     \begin{subfigure}[b]{0.495\textwidth}
         \centering
		 \scalebox{0.75}{\begin{tikzpicture}[scale=0.75,
                    fleche/.style={->, >=latex', shorten >=2pt, shorten <=2pt, thick},
                    doublefleche/.style={<->, >=latex', shorten >=2pt, shorten <=2pt, thick},
                    pi/.style={-,line width =3mm,blue!50,opacity=0.5},
                    QO/.style={-,line width =3mm, red,opacity=0.5},
                    eta/.style={line width =3mm, black!50,opacity=0.5},
                    caract/.style={dashdotted, very thick,red},
                    ]
  \draw[red!20,fill=red!20] (-1,4)--(6,9)--(8,9)--(1,4);
\draw[red!20,fill=red!20] (-3,0)-- (-1,4)--(1,4)--(-1,0);

\draw[red!20,fill=red!20] (-5,4)-- (-3,4)--(4,9)--(2.5,9);                 
 \draw[red!20,fill=red!20] (-5,4)-- (-3,4) --(-3,0)--(-5,0);
 
\draw[fleche] (-7,0)--(7,0) node[below] {$x$};
\draw[fleche] (1,-0.1)--(1,10) node[above] {$t$};
\draw[dotted] (-7,4)--(7,4) node[above, pos=0.52] {$t^*$};

\draw[caract] (-5,0)--(-5,4);
\draw[caract] (-3,0)--(-3,4);

\draw[caract,red] (-3,0) -- (-1,4);
\draw[caract,red] (-1,4) -- (6,9);
\draw[caract,red] (1,4) -- (8,9);
\draw[caract,red] (-1,0) -- (1,4);

\draw[caract,red] (-3,4) -- (4,9);
\draw[caract,red] (-5,4)-- (2.5,9);

\draw[pi] (-3,0.0)--(6.7,0.0) node[below,midway] {${\rm Supp} (\pi)$};

\draw[QO] (-5,0)--(-1,0)node[below,midway] {${\rm Supp} (Q_0)$};
\draw[eta] (-6.5,0.0)--(-2,0.0)node[below,pos=0.2] {${\rm Supp} (\eta)$};
\draw[fleche] (-0.1,-1.5)--(-1,0) node [below,pos=-0.1]{$M_0$};
\node[draw,red,rotate=65] at (-0.75,2.4)  {${\rm Supp} (Q(t,.))$};

\end{tikzpicture}
}
         \caption{Pathological state}
         \label{fig:scheme_pathological}
     \end{subfigure}
     \caption{Scheme of the characteristic curves and the supports of the functions $\eta$, $Q_0$ and $\pi$ for the system \eqref{eq:model_continuous} under the assumptions of Hypothesis \ref{hyp:supp_Q0_pi}. The red dashed lines are the characteristic curves. The red region is the region where the density of cells is non-zero. The thick red line is a dirac mass which contains part of the densities of the system (the characteristic curves are equal to 0). The time $t^*$ is the time of appearance of the first cancer cells (cf. Definition \ref{def:time_prolif}). The slope of the characteristic curves is changing at time $t^*$ because of the effect of the sensory axons and the effect of the cancer cells (through the function $\eta$) on the cancer progression.}\label{fig:scheme_charact}
\end{figure}
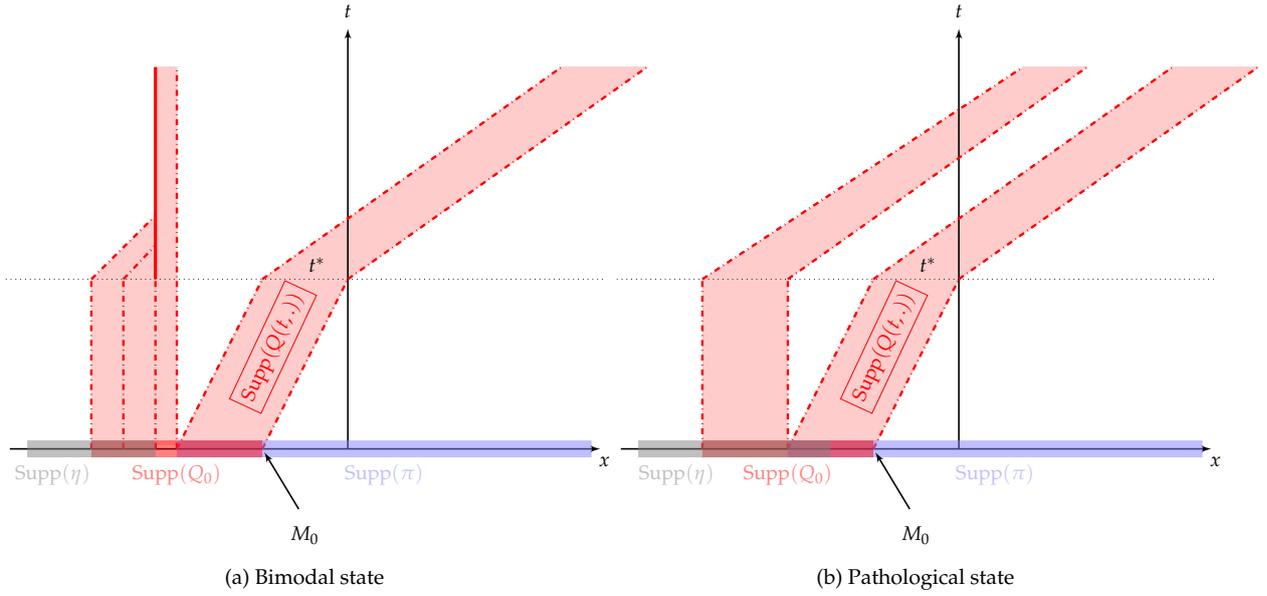

\begin{prop}[Appearance of the proliferative cells]\label{prop:bnd_time_prolif}
Assume Hypotheses \ref{hyp:supp_Q0_r}, \ref{hyp:pos_advec} and \ref{hyp:supp_Q0_pi} hold. We denote (cf. Figure \ref{fig:scheme_charact})
$$M_0 = \max\left( \text{supp}(Q_0)\bigcap \text{supp}(\pi) \right)\quad  \text{and} \quad r_0 =\min(\text{supp}(r)).$$ 
Then there exist two positive constants $0< m_1(\pi(M_0),\beta,\rho) < m_2(\pi, \delta)$ such that
$$ t^* \in \left[\frac{r_0-M_0}{m_2}; \frac{r_0-M_0}{m_1}\right].$$
\end{prop}
\begin{proof}[Proposition \ref{prop:bnd_time_prolif}]
Let $t\in (0,t^*)$, then the following equation holds for $Q$:
$$
\partial_t Q(t,x) + \partial_x \left(f(t,x)Q(t,x) \right) = 0,
$$
since $\text{supp}(Q_0)\bigcap \text{supp}(r) = \emptyset.$ It implies that 
$$Q(t,x) = Q_0(X(0,t,x)) \times \exp\left(-\int_0^t \partial_x f \big|_{(t,x)=(s,X(s,t,x))} ds \right)$$
for $t\in (0,t^*)$ and $x\in \Omega$ where $X$ denote the characteristic curves, i.e. 
\begin{equation*}
    \left\{
    \begin{array}{cl}
       \ds X(s,t,x)  & = f(s,X(s,t,x)), \\
         X(t,t,x) & = x. 
    \end{array}
    \right.
\end{equation*}
Then, assuming $t\in (0,t^*)$, we have the following 
$$\text{supp}(Q(t,\cdot)) = \text{supp}(Q_0(X(0,t,\cdot)). $$
Hence, the bounds on $t^*$ can be found by studying the characteristics. 
Since Hypothesis \ref{hyp:pos_advec} hold, then there exist two functions $f_1:\Omega \rightarrow \mathbb{R}_+^*  $ and $f_2:\Omega \rightarrow \mathbb{R}_+^*$ 
with 
\begin{equation*}
\begin{split}
    f_1(y)  &= \pi(y) (1-\beta(y)\rho(1)),\\
    f_2(y) &= \pi(y)(1+\delta(y))
\end{split}
\end{equation*}
such that 
$$ f_1(X(s,t,x) )\leq \ds X(s,t,x) \leq f_2(X(s,t,x)).$$
In order to estimate the bounds on $t^*$, we focus on the characteristic starting at $M_0 = \max(\text{supp}(Q_0)\bigcap \text{supp}(\pi))$ at time 0 (cf. Figure \ref{fig:scheme_charact}). 
Since $f(t,x)\geq 0$ for $t\geq 0$ and $x\in\Omega$ then $$X(t,0,M_0)\geq M_0,\quad \text{for } t\geq 0.$$
On the first hand, there exists $\varepsilon >0$ arbitrary small such that $\pi'(x)> 0$ for $x\in[M_0-\varepsilon, M_0+\varepsilon] \subset \big(\text{supp}(Q_0)\bigcap \text{supp}(\pi)\big)/\text{supp}(r)$ and $\max\limits_{x\in \Omega} \beta(x)\rho(1)<1$, we have that 
\begin{equation}
    0<\pi(M_0)(1-\max\limits_{x\in \Omega} \beta(x) \rho(1)) \leq f_1(X(t, 0, M_0), \quad \text{for } t\in(0,t^*).
    \label{eq:bdd_f1}
\end{equation}
On the other hand, since $t\in (0,t^*)$ then $\NC(t) = 0$, we have that 
\begin{equation}
     f_2(X(t, 0, M_0)\leq \max\limits_{x\in \Omega}\pi(x)(1+\max\limits_{x\in \Omega} \delta(x)), \quad \text{for } t\in(0,t^*).
    \label{eq:bdd_f2}
\end{equation}
We denote $m_1 = \pi(M_0)(1-\max\limits_{x\in \Omega} \beta(x) \rho(1))$ and $m_2 = \max\limits_{x\in \Omega}\pi(x)(1+\max\limits_{x\in \Omega} \delta(x))$, using the bounds in \eqref{eq:bdd_f1} and \eqref{eq:bdd_f2}, we obtain 
$$M_0 + m_1 t \leq X(t,0, M_0) \leq M_0 + m_2 t .$$
Finally, we obtain a bound for $t^*$ using the estimates on the characteristic starting at $M_0$ :
$$ \frac{r_0-M_0}{m_2} \leq t^* \leq \frac{r_0-M_0}{m_1}.$$
\end{proof}

{The dynamics of the mathematical model illustrate the different states of pancreatic cancer tumorigenesis.
Moreover,
depending on the information available on the pro or anti-tumoral interactions between the axons and the cells, 
our model can be adapted to estimate the bounds on the time of appearance of the first cancer cells (Proposition \ref{prop:bnd_time_prolif}).
}


\section{Parametrization and modelling of the denervation}
\label{sec:param}
In order to confront our model to the data and be able to extract biological information from it, we need to parametrize the dynamical system and introduce degrees of freedom. The parametrization given in Table~\ref{tab_parameters} is described in Section \ref{sec:precision}. To understand the regulation coming from the nervous system, we also need to introduce a mathematical model of the denervation treatment. This is done in Section~\ref{sec:parden}.

\subsection{Details of the model and its parameters}\label{sec:precision}

\begin{table}[ht]
\scalebox{0.9}{
\begin{tabular}{|c|c|c|c|c|}
\hline
&Name &  Range/Value & Units& Description\\
\hline
\multirow{3}{*}{Discretization} & $T$ & 70 & day & life time of a mouse\\
& $L$ & 50 & - & phenotype domain \\
&$dt$&$\min\left(\frac{0.9dx}{\pi_0(1+\delta)},0.005\right)$&day&time step\\
&$dx$&0.05&-&space step\\
\hline
\multirow{2}{*}{Initial condition}&$\overline{ Q}_0$&100&{cells/mm$^3$}&amount of healthy cells at time 0 days\\
&$A_1(0)$&0.15445&-&amount of sympathetic axons at time 0 days\\
&$A_2(0)$&0.004&-& amount of sensory axons at time 0 days\\
\hline
\multirow{7}{*}{ Speed of progression}&
$\pi_0$&$[0.1,10]$&day$^{-1}$&Basal amplitude of the speed\\
&$x_{1,\pi}$&$[30,49]$&-&phenotype at which progression starts\\
&$\epsilon_{1,\pi}$&$[10^{-4},10]$&-&steepness of the switch to progression\\
&$x_{2,\pi}$&10&-&phenotype at which progression stops \\
&$\epsilon_{2,\pi}$&10&-&steepness of the switch to no progression\\
&$\beta$&$[10^{-4},1]$&-&modulation by the sympathetic axons\\
&$\delta$&$[10^{-4},1]$&-&modulation by the sensory axons\\
\hline
\multirow{5}{*}{ Tumor growth}
&$\gamma_r$&$[10^{-2},10]$&day$^{-1}$&maximum proliferation rate\\
&$s_r$&$[0.2,10]$&-&rate of proliferation\\
&$\tau_c$&$[120,300]$&{cells/mm$^3$}& carrying capacity of total concentration of cells\\
&$\mu_1$&$[-1,1]$&-&modulation by the sympathetic axons\\
&$\mu_2$&$[10^{-4},1]$&-&modulation by the sensory axons\\
\hline
\multirow{5}{*}{ Axon growth}&$A_1^{eq}$&0.1544&-&Fixed from the data\\
& $r_{A_1}$&$[10^{-4},5]$&day$^{-1}$&Growth rate of the sympathetic axons\\
& $s_\theta$&$[13,20]$&-&steepness of decrease\\
&$\bar r_{A_2}$&$[10^{-4},5]$&day$^{-1}$&Growth rate of the sensory axons\\
&$ s_{A_2}$&$[10^{-4},10]$&-&steepneess of increase\\
\hline
\end{tabular}
}
\caption{Details on the parameters involved in the model.}\label{tab_parameters}
\end{table}

\noindent \paragraph{Choice of the domain.} 
{We propose to use  functions $\pi$ and $Q_0$ compactly supported so that we can reduce the study of the transport equation to 
a finite domain $\Omega = (-L,L)$ (cf. Section \ref{sec:qualitative_study}).  We are then able to approximate the solution of the model using an upwind  finite volume scheme for  the transport equation and  explicit Euler scheme for the ODE (cf. Appendix \ref{app:numerics}).
In the following, we fix $L=50$. We use a constant time step $dt$ and space step $h$ ensuring CFL conditions for the upwind scheme.}

\noindent \paragraph{Initial Conditions.} 
 We set $Q_0(x)$ to be the concentration of cells at time $t=0$ days. Since initially, only healthy acinar cells are observed, and since negative values of $x$ close to $50$ correspond to healthy cells, we then suppose $Q_0(x)$ follows a gaussian distribution centered at $-x_{init}$ defined by the following : 
$$ Q_0(x) = \overline{ Q}_0\frac{\exp{\{\frac{-(x+x_{init})^2}{4}} \}}{\sqrt{4\pi}}  $$
with $\overline{ Q}_0=100$, and $x_{init}= 40 $.\\
{We note that $N(0)$ is close to $ \overline{ Q}_0$.
Also, the choices of $\overline{ Q}_0$, $L$ and $x_{init}$ are arbitrary and impact the values of many of the parameters.}\\
\noindent As for the axons, at time $t=0$ days, a negligible amount of sensory axons, and a small amount of autonomic axons are present. Moreover, the axon densities in the mathematical model \eqref{eq:model_continuous} are renormalized and unitless. Based on the biological data provided, we consider 
$$A_1(0) = 0.15445,\quad A_2(0) = 0.004 .$$

\paragraph{Cancer progression.} The experimental data that we have is not quantified on the impact that cancer cells have on cancer progression of the healthy cells in the late stages of the disease.  Therefore, we impose in what follows $\eta=0$ such that 
    $$
f(t,x;{\mathcal X})=\pi(x)(1-\beta\rho(A_1)+\delta A_2).$$
This implies that transport to cancer phenotypes is no longer affected by the presence of cancer cells at a very advanced stage of the disease.
The advantage of this modeling approach is that it reduces the number of parameters without significantly impacting the system's transient dynamics.
\begin{figure}[ht!]
	\begin{center}
     	\begin{tikzpicture}[scale=0.9]
			\draw[->] (-6,0) -- (6,0) node[right] {};
    		\draw[->] (0,-0.1) -- (0,2) node[above] {\tiny $\pi(x)$};
			\draw[thick,variable=\t,samples=500,color=red, domain=-6:6]
  					plot ({\t },{((tanh(\t +  5-2)+tanh(-1*\t + 5- 3)))/2}) ;
 			\draw (6,0) node[above] {\tiny $x$};
			\draw (0,1.1) node[right] {\tiny $\pi_0$};
			\draw[dotted] (-3,0)--(-3,0.5);
			\draw (-3,0) node[below] {\tiny $-(L-x_{1,\pi})$};
			\draw (-5,0) node[below] {\tiny $-L$};
			\draw[dotted] (-3,0.5)--(2,0.5);
			\draw[dotted] (2,0)--(2,0.5);
			\draw (5,0) node[below] {\tiny $L$};
			\draw (2,0) node[below] {\tiny $L-x_{2,\pi}$};
			\draw (0,0.5) node[right] {\tiny $\frac{\pi_0}{2}$};
  		\end{tikzpicture}
	\end{center}
     \caption{Illustration of the basal amplitude for the speed of 					disease progression $\pi$.}
     \label{fig:pi_0}
\end{figure}
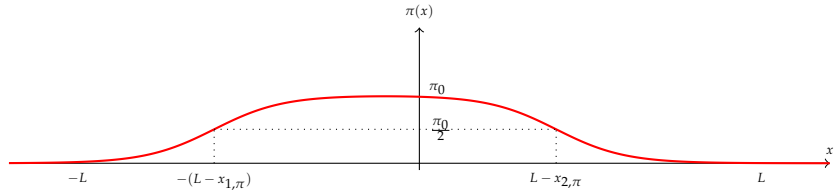
\begin{itemize}
    \item The function $\pi$ represents a basal amplitude for the speed of disease progression:
    $$\pi(x)=\pi_0 \left(\tanh(\epsilon_{1,\pi}(x+L-x_{1,\pi}))+\tanh(\epsilon_{2,\pi}(-x+L-x_{2,\pi}))\right)/2.
$$
Since initially the transdifferentiation of healthy cells is considered to be slow, we require the transport to be almost negligible around the lower bound of the phenotype domain. At some phenotype $-L+x_{1,\pi}\in (-L,0)$, we expect to see an increase in the transport speed. 
A constant amplitude is then observed during the PanIN/PDAC phase. We impose the assumption that after $x=L-x_{2,\pi}$ we have decreasing transport speed in amplitude which vanishes in a neighborhood of the upper bound of the domain. This implies that no more transport takes place since cells have reached their PDAC phenotype for $x\geq L-x_{2,\pi}$ (cf. Figure \ref{fig:pi_0}).
\begin{itemize}
    \item Initially the amount of sensory axons is almost negligible and increases only after the presence of PDAC cancer cells. We then expect the speed before the presence of cancer cells to be 
$$f(t,x;{\mathcal X})\cong \pi_0(1-\beta\rho(A_1)). $$ 
Since the cancer cells are seen around $35$ days, {the position of the initial distribution on the phenotype axis provides information on the approximate value of the parameters linked to cancer progression and the following can be considered }
    $$\pi_0(1-\beta\rho(A_1)) \cong  \frac{x_{init}}{35}.$$
But as $\rho(A_1)<1$, we then have $$ \pi_0 (1-\beta)\leq \pi_0(1-\beta\rho(A_1)) \leq \pi_0 $$
which gives, 
$$\pi_0 \in \left[\frac{x_{init}}{35}, \frac{x_{init}}{35(1-\beta)}\right].$$
	\item The parameter $x_{1,\pi}$ is closely related to the location on the phenotype axis where the cancer progression starts. We note that based on Section \ref{sec:qualitative_study}, we would like to secure that $\text{supp}(Q_0)\bigcap \text{supp}(\pi) \neq \emptyset$ so that tumorigenesis takes place. Since $Q_0$ is constructed such that $Q_0(x) \approx 0$ for $x\geq -20$, we would then like to ensure that $\min\{\text{supp}(\pi)\}\leq -30$ . Thus, we choose $L-x_{1,\pi} \in [30,49] $ so that $x_{1,\pi} \in [1,20]$. 
\begin{figure}[ht!]
    \begin{center}
       \begin{tikzpicture}[scale=0.5]
\draw[->] (-6,0) -- (6,0) node[right] {};
    \draw[->] (0,-0.1) -- (0,3) ;

\draw[thick,variable=\t,samples=500,color=red, domain=-6:6]
  plot ({\t },{((tanh(\t +  4-2)+tanh(-1*\t + 3)))/2}) ;
\draw[thick,variable=\t,samples=500,color=blue, domain=-5:5]
  plot ({\t },{(10*(e^(-(\t + 4)^2) /4))}) ;

  \draw (-5,0) node[below] {\tiny $-L$};
    \draw (5,0) node[below] {\tiny $L$};
  \draw[dotted] (-5,0)--(-5,2);
  \draw[dotted] (-4,0)--(-4,2.5);
  \draw (-4,0) node[below] {\tiny $x_{init}$};
 \draw (0,1.2) node[right] {{\color{red}\tiny $\pi(x)$}};
 \draw (-5,2.7) node[right] {{\color{blue}\tiny $Q_0(x)$}};
  \draw (6,0) node[above] {\tiny $x$};
  \end{tikzpicture}
\end{center}
    \caption{Illustration of the initial distribution of cells on the phenotype axis and the basal amplitude of the transport speed of the disease progression.}
    \label{fig:pi_0Q_0}
\end{figure}
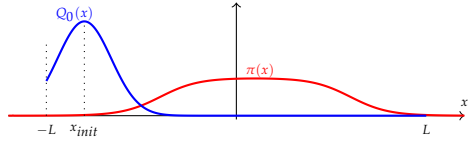  
    \item The parameter $\epsilon_{1,\pi}$ determines the rapidity of increase in the speed of transport. No biological information has been provided on this value but keeping in mind that it is essential for $\text{supp}(Q_0)\bigcap \text{supp}(\pi) \neq \emptyset$ , we then take  $\epsilon_{1,\pi} \in [10^{-4}, 10]$.
    \item The parameter $x_{2,\pi}$ is linked to the location on the phenotype axis where the transport decreases and stops. The idea of introducing $x_{2,\pi}$ is to avoid losing mass at the boundary value $L$. 
Thus, we choose $x_{2,\pi}$ so that $L-x_{2,\pi}<50$. Here, we fix this value to $x_{2,\pi} = 10$.
     \item The parameter $\epsilon_{2,\pi}$ determines the speed at which the transport speed is reduced. Since we do not expect any transport to happen after all susceptible cells have attained their cancerous state, we then choose $\epsilon_{2,\pi}=10$ so that $\pi(x)$ vanishes fast after $L- x_{2,\pi}$.
\end{itemize}
 \item The parameters $\beta$ and $\delta$ represent the maximum rate of the regulation of the progression by the sympathetic and sensory axons respectively. 
 Since no biological information is provided about their value and since they represent a modulation, we choose $(\beta,\delta) \in [10^{-4},1]\times [10^{-4},1]$.
  \item The function $\rho$ is a function that removes the effect of regulation for small values of $A_1$ and that is linear for larger values of $A_1$. This regulation through the function $\rho$ is motivated by the fact that the autonomous axons are non-negligibly present in healthy pancreas. We denote $A_1^{eq}$ the renormalized value of $A_1$ at the healthy state. To ensure that the quantity $A_1$ does not artificially affect the cancer progression when $A_1 < A_1^{eq}$, we construct $\rho$ as the following :

\begin{minipage}{0.55\linewidth}
\centering
$
\rho(x)=\begin{cases}
 0 \hspace{4mm} \text{for} \hspace{2mm}  0 \leq x \leq A_1^{eq} ,\\
\\
y = x-A_1^{eq} \hspace{4mm} \text{for} \hspace{2mm} A_1^{eq}\leq x \leq 1.
\end{cases}$
\end{minipage}
\begin{minipage}{0.4\linewidth}
    \begin{tikzpicture}
[scale=1.8]
    \draw[->] (-0.1,0) -- (1.3,0) node[right] {\tiny $A_1$};
    \draw[->] (0,-0.1) -- (0,0.7) node[above] {\tiny $\rho(A_1)$};
    
    \draw[ultra thick,variable=\t,samples=500,color=red, domain=0:0.5]
      plot ({\t },{0*\t });
      \draw[ultra thick,variable=\t,samples=500,color=red,domain=0.5:1]
      plot ({\t },{(\t-0.5) });
      \draw (0.5,-0.05) node[below] {\tiny $A_1^{eq}$};
      \draw (1,-0.05) node[below] {$1$};
      \draw[dotted] (1,0)--(1,0.5);
      \draw[dotted] (0,0.5)--(1,0.5);
      \draw (-0.05,0.5) node[left] {\tiny $1-A_1^{eq}$};
      \end{tikzpicture}
  \vfill
\end{minipage}
\end{itemize}

\paragraph{Cancer growth.} {The proliferation is generated by a logistic growth. More precisely,}
$$g(t,x;{\mathcal X}):=r(x)Q(t,x)\left(1- \frac{N(t)}{\tau_C} - \mu_1 A_1(t)+\mu_2A_2(t)\right).
$$
\begin{itemize}
 \item The function $r$ represents the basal growth rate of the pre-tumor and tumor cells. 
We expect for proliferation to start taking place in the presence of pre-cancerous cells, i.e. in a neighborhood of $0^-$. Hence, we consider that 
$$r(x)=\frac12 \gamma_r(1+\tanh(s_r x)).$$
\begin{itemize}
    \item The maximum proliferation rate that the cells can exhibit is $\gamma_r$. Here, we take $\gamma_r \in [10^{-2}, 10 ]$.
    \item  The parameter $s_r$ determines the rate at which proliferation takes place. For small values of $s_r$, then the proliferation increases in rate in precancerous lesions,
    whereas for big values of $s_r$, the tumor growth
    increases in rate only in late stages of precancerous lesions.
      We choose  $s_r \in [0.2, 10]$.
\end{itemize}
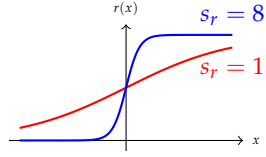
\begin{figure}[ht!]
\begin{center}
          \begin{tikzpicture}[scale=1.4]
\draw[->] (-1.1,0) -- (1.1,0) node[right] {\tiny$x$};
    \draw[->] (0,-0.1) -- (0,1.1) node[above] {\tiny $r(x)$};

\draw[thick,variable=\t,samples=500,color=red, domain=-1:1]
  plot ({\t },{(1+tanh(\t))/2 })
  node[right,below] {\small $s_r=1$};
  \draw[thick,variable=\t,samples=500,color=blue,domain=-1:1]
  plot ({\t },{(1+tanh(\t*8))/2 })
  node[right,above] {\small $s_r=8$};
  \end{tikzpicture}
  \end{center}
    \caption{Illustration of the behaviour of the basal growth rate of the pre-tumor and tumor cells.}
    \label{fig:s_r}
\end{figure}
\item The growth rate of the pre-tumor and tumor cells can be regulated by the presence of axons through the term
$$-\mu_1 A_1(t)+\mu_2 A_2(t).$$
In that, we consider that the sympathetic axons either slow down or accelerate tumor growth
whereas the sensory axons only accelerate tumor growth.
The parameters $\mu_1$ and $\mu_2$ modulate the impact that the axons have on the proliferation. We consider that $\mu_1 \in [-1,1]$ and $\mu_2 \in [10^{-4},1] $.
\item It's important to point out that even if cells proliferate, the total quantity of cells is limited and there is no explosion in finite time. In other words, we require $1- \frac{N(t)}{\tau_C} - \mu_1 A_1(t)+\mu_2A_2(t) < \infty$ with $\tau_C$ being the carrying capacity of the amount of cells in the absence of axons. Thus, using the fact that $A_i \leq 1$ for $i=1,\;  2$ we have, 
\begin{equation*} 
\begin{split}
N(t) &\leq  \tau_C (1 - \mu_1 A_1(t)+\mu_2A_2(t)) \leq    \tau_C (1 + \mu_2 +\mu_1^-), \\
&\leq    C(\tau_C),
\end{split}
\end{equation*}
with $C(\tau_C) = \tau_C (1 + \mu_2 +\mu_1^-)$ where $x^- = \tfrac{1}{2}(|x| - x)$. Since $N(0)=100$, we then require $N(0)\leq \tau_C$. We choose $\tau_C \in [120,300]$.
\end{itemize}

\paragraph{Sympathetic axons growth.}
We introduced an Allee effect in the growth term of the sympathetic axons (similar to what is seen in \cite{lolas2016tumour}). We recall that the dynamics of $A_1$ is given by the following 
$$r_{A_1} A_1(t)\left(\frac{A_1(t)}{\theta\left(\frac{N(t)}{N(0)}\right)}-1 \right)(1-A_1(t)).
$$
 \begin{itemize}
    \item The parameter $r_{A_1}$ is the growth rate of the autonomic axons. The bigger the value of $r_{A_1}$ is, the steeper the increase or decrease of $A_1$ is. We consider $r_{A_1}\in [10^{-4},5]$.
    \item 
    In the light of what is done in \cite{lolas2016tumour}, we also propose to modulate the threshold $\theta$ with the total amount of cells:
\begin{equation}\label{theta1}
    \theta\left(\frac{N(t)}{N(0)}\right) = \frac{A_1^{eq}}{2} + \frac12 \tanh{\left(s_{\theta}\left(\frac{N(t)}{N(0)}-1-0.1\right)\right)}+\frac12
\end{equation}
where $A_1^{eq}$ is the density of the sympathetic axons at the healthy state and $s_\theta$ is the steepness of the decrease speed of $A_1$.
The value of $A_1^{eq}$ is estimated from the data (cf. Table \ref{tab_parameters}).
In order to ensure a biological meaning, the following has to hold: 
for small time $t$, $ \theta\left(\frac{N(t)}{N(0)}\right) \approx \theta(1) $, 
and we want $\theta(1) < A_1(0)$ with $A_1(0) \approx A_1^{eq}$ to observe initially an increase in the autonomic axon density. Thus, we want 
$$s_\theta \geq \frac{\tanh^{-1}(1-A_1^{eq})}{10^{-1}}.$$
It leads us to consider $s_{\theta}\in [13,20]$.
For large values of the ratio $\frac{N(t)}{N(0)} $, we have $\theta\left(\frac{N(t)}{N(0)}\right) \approx 1+ \frac{A_1^{eq}}{2} > 1$ leading to a decrease in the autonomic axon density. 
\end{itemize}

\paragraph{Sensory axons growth.} 
The sensory axons follow a logistic growth regulated by the presence of PDAC
$$r_{A_2}\left( \frac{\NC (t)}{N(t)}\right)A_2(1-A_2)$$
with $r_{A_2}(s)=\bar r_{A_2}\tanh(s_{A_2}s)$.
The parameter $\bar r_{A_2}$ is the maximum growth rate of the sensory axons. We consider $\bar r_{A_2} \in [10^{-4},5]$. The parameter $s_{A_2}$ modulates the steepness of the increase of the growth of $A_2$. Here, we consider $s_{A_2} \in [10^{-4},10]$.

\subsection{Modeling denervation treatment}\label{sec:parden}
\noindent Testing the effects of denervation in preclinical models often requires the use of numerous animals, raising ethical concerns. Therefore, the use of a mathematical model to predict the effects of denervation becomes highly valuable, allowing for the replacement and reduction of the number of animals used. In the literature, the role of axons in cancer progression has been studied through chemical or surgical denervation. In our mathematical model, denervation is implemented by setting parameters related to the targeted axon to $0$ after the time of the intervention.\\

\noindent Specifically, denervating sympathetic axons is equivalent to suppressing the influence of $A_1$ on cell transport and proliferation. This is achieved by setting the parameters as follows after the denervation time:
$$\beta = 0,\hspace{3mm} \mu_1 = 0. $$
\noindent Conversely, denervating sensory axons is equivalent to suppressing the influence of $A_2$ on cell transport and proliferation, after the denervation time, the parameters are set as:
$$\delta = 0,\hspace{3mm} \mu_2 = 0 . $$
\noindent For denervating both axons, the parameters are set as follows after  the denervation time: 
$$\beta = 0,\hspace{3mm} \mu_1 = 0,\hspace{3mm} \delta = 0,\hspace{3mm} \mu_2 = 0.$$

\noindent We have to adapt the notations given above to take into account the above mathematical denervation when useful. The dynamical system defined in Section~\ref{sec:math}, in particular in Eq.~(3) is parameterized by  $\vartheta$ as explained in Section~\ref{sec:precision} and models the control treatment (AA). The total amount of cells and the total amount of cancer cells arising from this dynamical system is denoted by $N(t|\vartheta, \text{AA})$ and $N_c(t|\vartheta, \text{AA})$, respectively.
The dynamical system that models the denervation treatment is obtained from the same dynamical system, yet the value of $\vartheta$ is changed after time of the intervention, where some components are set to $0$ as explained above. The total amount of cells and the total amount of cancer cells arising from this second dynamical system is denoted by $N(t|\vartheta, \text{OHDA})$ and $N_c(t|\vartheta, \text{OHDA})$, respectively.

\section{Calibration} \label{sec:calibration}
In this section, we focus on the calibration of the model to the biological data at our disposal. It is rather complex in our case since
\begin{itemize}
    \item a small number of biological data is provided, with variability in samples
    as described in Section~\ref{sec:data},
    \item the model is far from being linear in the parameters and the predictions returned by the mathematical model cannot be written explicitly as a function of the parameters,
    \item we do not have at our disposal a likelihood function to explain the difference between the predictions of model and the data.
\end{itemize}

\noindent Moreover, the set of parameters is of rather high dimension ($d=14$) and we need to avoid selecting the best set of parameters that would overfit the few observed  data and would not reflect the biological phenomenon we are trying to capture.

\noindent To this aim, we propose in Section \ref{sub:crit} a 2-dimension criterion to measure the quality of the calibration of the model to the data and to our biological knowledge explained in Section~\ref{sec:data}: \textit{(1)} a mean squared error criterion based on the cell distribution and axon evolution, that measures the difference between model outcomes and data, and \textit{(2)} a criterion based on the biological knowledge we have on the chronology of cell evolution. As described in Section \ref{sub:optim}, we managed to extract biologically relevant sets of parameters that are consistent with the cell data and the chronological knowledge by exploring the parameter space with a multi-stage quasi-Monte Carlo algorithm that starts from a flat, uniform prior over the parameter space which is the hypercube given by the range column of Table \ref{tab_parameters}, and ends with the $0.4\%$ best sets of parameters according to the criterion, that are presented in Section~\ref{sec:relevantsets} and in Appendix~\ref{app:calib}.

\subsection{Biological data}\label{sec:data}
{The selection of experimental data for calibrating the mathematical model is crucial and should align with the characteristics of the relevant preclinical model.} 
{The calibration process encompasses various aspects of the animal model, incorporating chronological insights into cancer progression derived from the literature.} Additionally, data on sympathetic nerve axons density, sensory nerve axons density, and cell density at different stages of cancer progression are considered during the calibration process.\\

\paragraph{Mice model and data acquisition methods.} In this study, we use
the $Kras^{LSL-G12D/+};\; Cdkn2a (Ink4a/Arf)^{lox/lox};\; Pdx1$-$Cre$ (KIC) )  mouse model of pancreatic cancer. The overexpression of a mutated form of the oncogene $KRAS$ in pancreatic cells induces transdifferentiation, where healthy acinar cells (the functional unit of the exocrine pancreas) transform into ductal cell-like cells, leading to the formation of acinar to ductal metaplasia (ADM). Subsequently, ADM can progress to form premalignant pancreatic intraepithelial neoplasia (PanIN). The deletion of tumor-suppressor genes Ink4A and Arf further accelerates tumor development, resulting in PDAC itself. The KIC mice model successfully replicates the stepwise progression observed in human pancreatic cancer (cf. \cite{aguirre2003activated}, and Figure \ref{fig:scheme_chronological}). Several mouse pancreases at the time points of $6.5$ weeks and $8$ weeks are utilized in the experiments to accommodate biological heterogeneity.
The technique used to quantify nerve axons density and cell density within lobules of the pancreatic tissue is detailed in the source data file of \cite{guillot2022sympathetic}. 
In short, data comes from the quantification of 3D cleared tissue images obtained by Light Sheet Fluorescence Microscopy. Tissue staining was achieved through immunostaining and tissue clearing using the iDISCO+ protocol. Tyrosine hydroxylase antibody staining (TH), an enzyme involved in the biosynthesis of norepinephrine (one of the main neurotransmitters of the sympathetic nervous system), and Calcitonin gene-related peptide (CGRP) antibody were used to visualize sympathetic and sensory innervation respectively. Regions of Interest (ROI), namely asymptomatic (ASYM), ADM, PanIN, and PDAC, were segmented based on the autofluorescence signal of the tissue, and their volumes were measured. Axonal networks were manually reconstructed to measure the total axon length in each ROI. 
The unitless axon density was calculated as follows: $\((\text{axons length sum} \times 1000) (\mu m) / (\text{volume of each Regions of Interest (ROI))}^{\frac{1}{3}} (\mu m) \) $. 
Densities obtained through experimentation were
normalized by first dividing by the cubic root of the volume in which the density was measured and then by an affine transformation to obtain proportions. This normalization facilitates the comparison of quantities both within the experimental data and with the outputs provided by the mathematical model (cf. Figures \ref{fig:boxplot_syaxons} and \ref{fig:boxplot_seaxons}).\\
The method used to describe the proportion of tissue categorized by their phenotypes in histological sections through the pancreas of a $6.5$-week-old mouse can be found in the source data file of \cite{guillot2022sympathetic} (cf. Figure \ref{fig:boxplot_pdac}).

{\paragraph{Knowledge on the chronological process.}

Based on the previous characterization of the KIC model (seen in \cite{aguirre2003activated} and in our personal observation), 
the time of observation of ADM lobules denoted $\tadm$ is around $17$ days,
the first appearance of relatively advanced-staged PanIN lobules denoted $\tpanin$ is observed around $24.5$ days, 
and the first appearance of PDAC lobules denoted $\pdacearly $ is  observed around 35 days.
The median survival of a KIC mouse model is $9$ weeks and so, at 6.5 weeks ($=\tpdac$) the mice has already developed a tumor, but at $8$ weeks ($=\pdacadvanc$), the cancer is really aggressive. All these chronological assumptions are summarized in Figure \ref{fig:scheme_chronological}. 
Additionally, we include a time interval of plus or minus three days around the above temporal values to account for variability in observations. }

\begin{figure}[ht!]
\centering
\begin{center}
\scalebox{0.9}{
    \begin{tikzpicture}[scale=0.30] 
\draw[->] (4,0) -- (59,0);
\draw[ultra thick,red] (4,0)--(10,0);
\draw[ultra thick,red] (14,0)--(20,0);
\draw[ultra thick,red] (21.5,0)--(27.5,0);
\draw[ultra thick,red] (32,0)--(38,0);
\draw[ultra thick,red] (42,0)--(48,0);
\draw[ultra thick,red] (53,0)--(58.9,0);
\foreach \x in {7,17,24.5,35,45,56}
    {\draw (\x,-0.2)--(\x,0.2) node [above] {\x};}
\foreach \x in {7,17,24.5,35,45,56}
    {\draw (\x,1.5) node [above] {\verteq};}  
    
 \foreach \x/\y in {7/\tasymp,17/\tadm,24.5/\tpaninadvanc,35/\pdacearly,45/\tpdac,56/\pdacadvanc}
    {\draw (\x,3) node [above] {$\y$};}     

\foreach \x/\y in {17/ADM,24.5/PanIN ,35/PDAC ,45/6.5 weeks,56/Advanced PDAC}
     {\draw (\x,-1) node[below]{\y};}
 
\end{tikzpicture}
}
\end{center}
\caption{Diagram outlining the chronological assumptions regarding cancer progression in pancreas of KIC mice, drawing from \cite{aguirre2003activated} and our firsthand observations. The indicated times represent the estimated average time at which each new stage is expected to be observed for the first time. The bold red segments indicate a time interval of plus or minus three days around the mentioned temporal values, encompassing variations in chronological estimates.  }
\label{fig:scheme_chronological}
\end{figure}
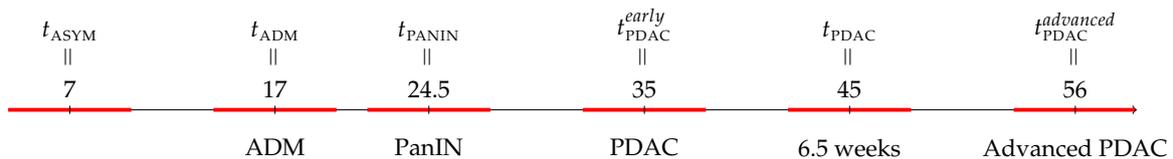

\paragraph{Sympathetic and Sensory  nerve axons quantification.}
{The normalized densities of sympathetic and sensory nerve axons are shown in Figure \ref{fig:boxplot_syaxons} and Figure \ref{fig:boxplot_seaxons}, respectively, quantified at $6.5$ weeks of the KIC age. 
For sensory axons, density measurements were also performed at $8$ weeks for the PDAC lesions.
Also, each sample of density measurements from the data does not only characterize the axon densities but also the
tissue stages from where it was observed (ROI). 
This additional information may not be fully exploited by the model because information about space is not taken into account in the construction of the model. 
One way to use all information of the data is to extrapolate a dynamic based on the ROI type of the data origin. 
Since we know that $\tasymp$, $\tadm$, $\tpanin$, and $\tpdac$ are the times of observation of each ROI type during cancer progression, we then consider that the quantification of the normalized densities of the axons in each ROI type corresponds to those times respectively.
Those timings are stated on the x-axis of the Figure \ref{fig:boxplot_syaxons} and Figure \ref{fig:boxplot_seaxons}. 
Hence, the dynamics of the axons are illustrated by the green curves in Figures \ref{fig:boxplot_syaxons} and \ref{fig:boxplot_seaxons}.}
It allows to combine qualitative information (chronological knowledge) to quantitative information and ensures a more precise calibration of the mathematical model. 
}

\begin{figure}[ht!]
\includegraphics[width=0.63\textwidth]{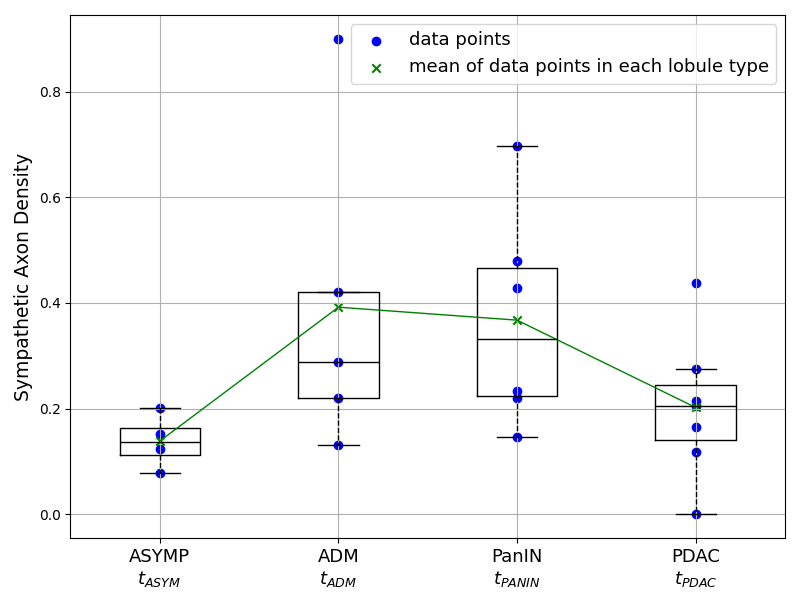}
\centering
\caption{Whisker Plot of Normalized Sympathetic Nerve Axon Density: Data from different KIC mice depict the sympathetic axonal density within various ROI in the pancreas at 6.5 weeks of age, categorized as ASYMP, ADM, PanIN, or PDAC. The quantification involved 4 ASYMP, 5 ADM, 6 PanIN, and 7 PDAC ROIs \cite{guillot2022sympathetic}. The dotted blue points correspond to the data. The green curve represents the mean density in each ROI, with identified stages on the x-axis extrapolated from the chronological knowledge of each ROI type. }\label{fig:boxplot_syaxons}
\end{figure}

\begin{figure}[ht!]
\includegraphics[width=0.63\textwidth]{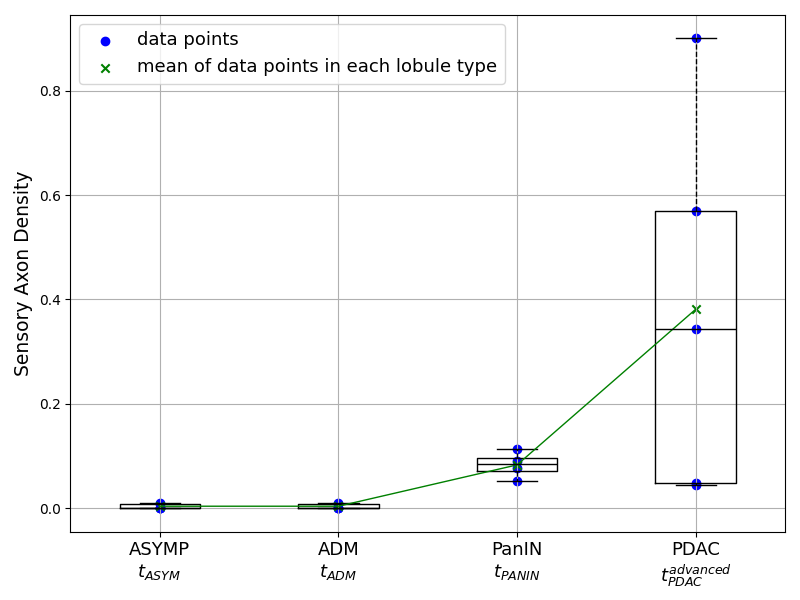}
\centering
\caption{Whisker Plot of Normalized Sensory Nerve Axon Density: Data from different KIC mice depict the sympathetic axonal density within various ROI in the pancreas at $6.5$ weeks and $8$ weeks of age, categorized as ASYMP, ADM, PanIN, or PDAC. The quantification involved $6$ ASYMP, $6$ ADM, $6$ PanIN, and $5$ PDAC ROIs \cite{guillot2022sympathetic}. The dotted blue points correspond to the data. The green curve represents the mean density in each ROI, with identified stages on the x-axis extrapolated from the chronological knowledge of each ROI type.}\label{fig:boxplot_seaxons}
\end{figure}

\paragraph{PDAC Cell Concentration.} 
Figure \ref{fig:boxplot_pdac} illustrates the data extracted for the normalized cancer cell concentration at 6.5 weeks. We set the time of observation for the data and aim to identify model dynamics that closely match these data points at $6.5$ weeks. Through observation, we note that the proportion of PDAC cells is higher in OHDA samples than in AA samples.

\begin{figure}[ht!]
\includegraphics[width=0.63\textwidth]{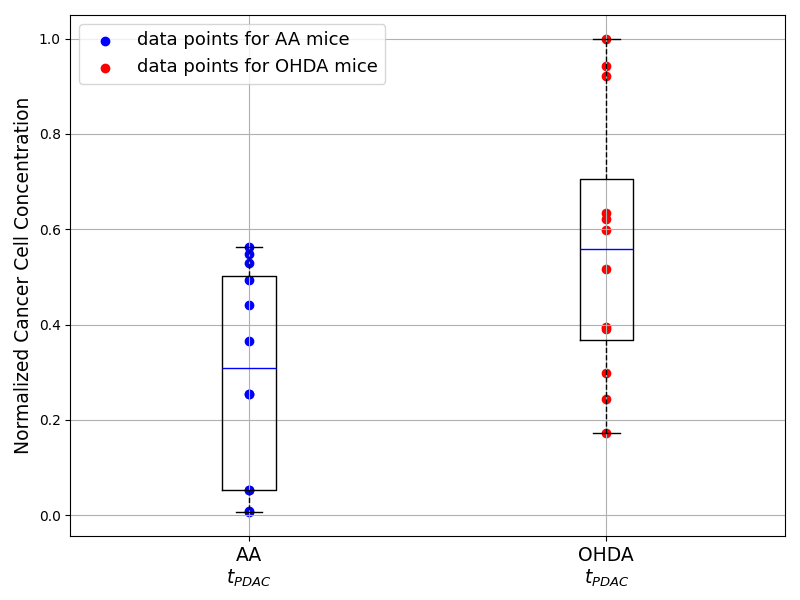}
\centering
\caption{Whisker Plot of Normalized Cancer Cell Concentration: Sympathectomy was conducted through the injection of neurotoxin 6-hydroxydopamine (6-OHDA) vehicle solution between $3.5$ and $4$ weeks of the KIC age. We utilized data published in \cite{guillot2022sympathetic}, where the concentration of cancer cells in $6$ KIC mice was quantified at $6.5$ weeks. The mice were distributed as follows: $3$ mice injected with ascorbic acid (AA) and $3$ mice injected with the neurotoxin 6-hydroxydopamine (OHDA). For each mouse, we obtained $4$ biopsies, and for each biopsy, we collected the proportion of PDAC. Therefore, we consider a total of $6\times 4=24$ quantifications of the proportion of PDAC, with  $3\times 4 = 12$ quantifications from mice injected with AA and $12$ quantifications from mice injected with OHDA. The dotted blue and red points correspond to the data for control mice and denervated mice respectively. }\label{fig:boxplot_pdac}
\end{figure}

\subsection{A 2-dimensional criterion }\label{sub:crit}
The first dimension of our criterion is a mean squared error criterion that measures the difference between model outcomes and data. We consider the following data:
\begin{itemize}
    \item $n$ biological samples of cell distributions $p_i\in(0;1)$ under treatment $u_i\in\{\text{AA}, \text{OHDA}\}$ ($i=1,\ldots, n$) that should be compared to the proportion of proliferating cells $p(t|\vartheta,u)$ given by the model, on average over time $t\in(0,T)$
    \item $m$ biological samples of axon density $a_i\in(0;1)$ in AA ($i=n+1,\ldots, n+m$) that should be compared to the axon density evolution $A(t|\vartheta)$ given by the model, on average over the time range $t\in I_i$ corresponding to the kind of tissues of $i$-th biological sample (ADM, PanIN, PDAC or mature PDAC). Moreover, $z_i\in\{A_1, A_2\}$ {indicates} whether the observed density $a_i$ is a density of sympathetic ($A_1$) or sensory ($A_2$) axons.
\end{itemize}
The first coordinate of our criterion is defined as the mean squared error (MSE) given by
\begin{equation}
G_1(\vartheta) = \frac{2}{n}\sum\limits_{i=1}^n \frac{1}{|I_{\tpdac}|}\int_{I_{\tpdac}} \Bigg(p_i -p\big(t\big|\vartheta, x_i\big) \Bigg)^2dt + \frac{1}{m}\sum\limits_{j=n+1}^{n+m} \frac{1}{|I_j|}\int_{I_j} \Bigg(a_j -A\big(t\big|\vartheta,z_j\big) \Bigg)^2dt.
\label{eq:costsumdata}
\end{equation}

\noindent In the first part of the MSE defined in \eqref{eq:costsumdata}, we have integrated the square difference over a time interval of size 6 because the exact time (on our time axis) at which the cell distributions were observed is not known
\begin{itemize}
    \item $p(t|\vartheta, u_i)\defeq {N_c(t|\vartheta, u_i)}\Big/{N(t|\vartheta, u_i)}$ is the proportion of proliferative cells given by the model at time $t$,  $N_c(t|\vartheta, u_i)$ being the amount of proliferative cells  and $N(t|\vartheta, u_i)$ being the total amount of cells at time $t$ given by the model of Section~\ref{sec:parden}.
    \item $I_{\tpdac} = [\tpdac-3;\tpdac+3]$ is the time range at which the cell distributions were observed.
\end{itemize}

\noindent In the second part of the MSE defined in \eqref{eq:costsumdata}, we have integrated the square difference over a time interval of size 6 because the exact time (on our time axis) at which the axon densities were observed is not known, but we know the kind of tissues we sampled. More precisely,
\begin{itemize}
    \item $I_j\in\Big\{[t_{ADM}-3; t_{ADM}+3], [t_{PANIN}-3; t_{PANIN}+3], [\tpdac-3; \tpdac+3], [\pdacadvanc-3; \pdacadvanc+3]\Big\}$ is the time range corresponding to the observation of the $j$-th biological sample, i.e the sample coming from either the ADM, PanIN, PDAC or mature PDAC {lobules}.
    \item $A(t|\vartheta,z)\defeq A_1(t|\vartheta) \text{ or } A_2(t|\vartheta)$ is the axon density given by the model at time $t$  of type $z$  and parameter set $\vartheta$.
\end{itemize}

\noindent The second component of our criterion is a penalization term to capture the correct behavior of the model according to the biological knowledge we have on the chronology of appearances of PDAC cells(see \cite{aguirre2003activated}) in the control group (AA). We do not expect many PDAC cells at time $t<\pdacearly$ where $\pdacearly$ is the first time of appearance
of PDAC lobules. Thus, the second component of our criterion is defined as
\begin{equation}
G_{2}(\vartheta)  = \frac{1}{\pdacearly} \int_0^{\pdacearly} p(t|\vartheta, \text{AA}) dt,
\label{eq:chrono_penal}
\end{equation}
where $p(t|\vartheta)$ is the proportion of proliferative cells at time $t$ given by the model. The quantity $G_2(\vartheta)$ is the mean proportion of proliferative cells in the control group (AA) over the time range $\left[0, \pdacearly\right]$ and it should be low to respect the biological knowledge we want to introduce in the calibration process.

 \subsection{Seeking for parameter sets with low criterion values} \label{sub:optim}

The range of parameters that makes sense in the mathematical model is the hypercube $\mathcal H$ given in Table \ref{tab_parameters}. Without constraining $\vartheta\in\mathcal H$ with the above criterion $G(\vartheta)$, the dynamics of the model can change strongly based on the value of $\vartheta$. 
Many roads are available to restrain this range and obtain relevant parameter sets based on the 2-dimensional criterion $G(\vartheta) = (G_1(\vartheta), G_2(\vartheta))$. We could have tried to optimize a $1$-dimensional criterion such as $\bar G_\lambda(\vartheta)=G_1(\vartheta) + \lambda G_2(\vartheta)$, where $\lambda$ is a tuning parameter that defines the trade off between the data and the chronological information. Yet tuning $\lambda$ is difficult. Moreover, because of non-convexity, many local minima may exist, exhibiting different dynamics of the model.  Instead of minimizing a $\bar G_\lambda(\vartheta)$, we use a
a multi-stage algorithm that selects first the best sets $\vartheta$ of parameters according to $G_1(\vartheta)$ and then refines the selection according to $G_2(\vartheta)$.\\

\noindent The results given by the multi-stage algorithm given in Appendix \ref{app:numerics} is a (relative large) collection of parameter sets $\vartheta$ that are consistent with the cell data and the chronological knowledge, i.e. that have low values of both $G_1(\vartheta)$ and $G_2(\vartheta)$. We start with a massive quasi-Monte Carlo sampling the hypercube $\mathcal H$ with a uniform distribution that acts here as a non-informative prior. The massive collection is then filtered according to both components of $G(\vartheta)$ and re-sample the hypercube $\mathcal H$ with an instrumental Gaussian distribution fitted on the filtered collection. The filtering step is then repeated to get the final collection of parameter sets. 

\begin{figure}[ht!]
\centering
\begin{center}
\scalebox{0.9}{
    \begin{tikzpicture}[scale=0.30,
    box/.style={rectangle, draw=blue, thick, fill=blue!20, text width=5em,align=center, rounded corners, minimum height=2em},
    box2/.style={rectangle, draw=red, thick, fill=red!20, text width=5em,align=center, rounded corners, minimum height=2em},
    line/.style ={draw, thick, -latex',shorten >=2pt},
    ]
  \draw (0,10) node[box] (J0){${\mathcal J}_0$};
 \draw (0,0) node[box] (J1){${\mathcal J}_1$};
  \draw (10,10) node[box2] (J2){${\mathcal J}_2$};
   \draw (10,0) node[box2] (J3){${\mathcal J}_3$};
 \draw[line] (J0.south)--(J1.north);
 \draw (5,5) node { Double filtering};
 \draw[line] (J2.south)--(J3.north)  ;
 \draw[line] (0,15) -- (J0.north) node [pos=-0.3] {Uniform sampling};
 \path [line] (J1.south)--(0,-5)--(15,-5)--(15,15)--(10,15)--(J2.north);
 \draw (8,-6) node {Mean and covariance};
 \draw (16,16) node {Gaussian sampling};
 \draw (7,18) node {Quasi-Monte Carlo};

\end{tikzpicture}
}
\end{center}
\caption{Schematic view of the step-wise algorithm used to select parameter sets. The collection $\mathcal J_0$ of parameter sets is distributed uniformly over the hypercube $\mathcal H$ of Table~\ref{tab_parameters}. The collection $\mathcal J_1$ is obtained by filtering $\mathcal J_0$ with $G_1(\vartheta)\le g_1$ and then $G_2(\vartheta)\le g_2$. The collection $\mathcal J_2$ is obtained by re-sampling $\mathcal H$ with an instrumental Gaussian distribution fitted on $\mathcal J_1$. The collection $\mathcal J_3$ is obtained by filtering $\mathcal J_2$ with $G_1(\vartheta)\le g_1'$ and then $G_2(\vartheta)\le g_2'$. The thresholds $g_i$ and $g_i'$ are chosen as some quantiles of small order of the collection that is constrained. The final collection $\mathcal J_3$ is a collection of parameter sets that are consistent with the cell data and the chronological knowledge, i.e. that have low values of both $G_1(\vartheta)$ and $G_2(\vartheta)$ (see Appendix \ref{app:calib} for more details).}
\label{fig:select param}
\end{figure}
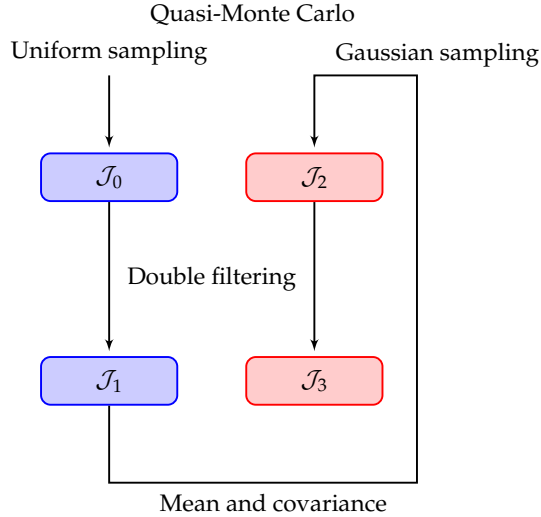

\subsection{Relevant sets of parameters}\label{sec:relevantsets}

\noindent By taking some of the sets of parameters obtained in Section \ref{sub:optim}, as seen in Figure \ref{fig:model_dynamics}, $A_1$ and $A_2$ have variability in their time of remodeling, as well as there is a difference in the time of arrival of cancer cells. 
We observe that despite the variability in dynamics, all parameter sets obtained fit the biological data. In the subsequent sections, we will utilize some of these parameter sets to conduct a detailed investigation of in silico denervation 

\begin{figure}[ht!]
\includegraphics[width=0.8\textwidth]{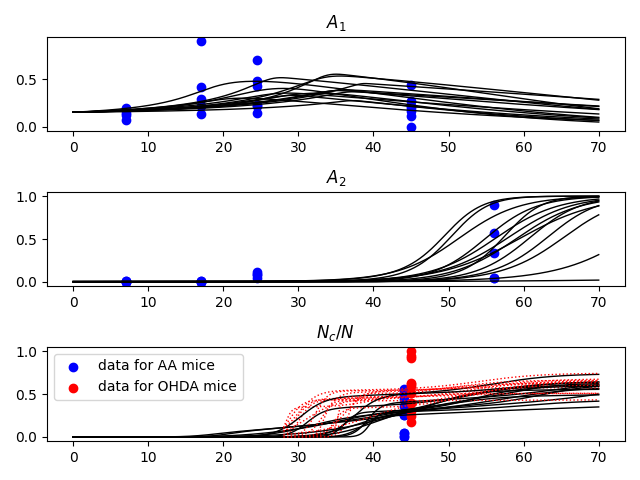}
\centering
\caption{Dynamics of the model using different sets of parameters from the relevant sets of parameters obtained. (Top) Evolution of the sympathetic axons ($A_1$). (Middle) Evolution of the sensory axons ($A_2$). (Bottom) Evolution of the proportion of cancer cells ($N_c/N$). The blue dots in each sub-figure correspond to the biological data provided for the sympathetic axons, sensory axons, and the frequency of cancer cells respectively. 
The red dots in the bottom frame correspond to the frequency of cancer cells observed at $6.5$ weeks after an early denervation at $28$ days with the use of OHDA, and the dashed red curves correspond to the evolution of the concentration of cancer cells after early denervation.  }\label{fig:model_dynamics}
\end{figure}
\section{Impact of the in silico denervation}
\label{sec:in_silico}
The aim of this section is to investigate the effect of denervation on cancer progression using specific parameter sets selected from Section \ref{sec:relevantsets}, chosen for their cost-effectiveness and alignment with biological data outlined in Table 2. 
Their values for the criteria of fitting the data \eqref{eq:costsumdata} are $0.264$, $0.2461$, and $0.246$ respectively. 
Their values for the criteria of fitting the chronological assumption \eqref{eq:chrono_penal} are $ 0.033$, $0.002$, and  $0.014$ respectively. 

\noindent The parameter configurations 2 and 3 outlined in Table \ref{tab:allpar} are selected so that sympathetic axons induce  an inhibitory effect 
on tumor growth ($\mu_1>0$). 
However, they are distinguished by varying durations for axon remodeling, reflecting patterns observed in the literature \cite{guillot2022sympathetic} and accommodating the variability inherent in biological data.
Furthermore, in light of findings from a literature source \cite{renz2018beta2} suggesting a positive impact of $A_1$ on tumor growth, 
we also choose parameter set 1 with $\mu_1<0$ to investigate this case.\\
In what follows, we analyze the effects of varying denervation timings on either sympathetic or sensory axons, or both, in order to elucidate their impacts on tumorigenesis.
 
\begin{table}[H]
\centering
\begin{tabular}{|l|c|c|c|c|c|c|c|c|c|c|c|c|c|c|c|c|c|c|c|c|c|}
 \hline
  & \small{$\pi_0$} & \small{$\beta$} & \small{$\delta$} & \small{$\gamma_r$} & \small{$s_r$} & \small{$\tau_C$} & \small{$\mu_1$} & \small{$\mu_2$} & \small{$r_{A_1}$} & \small{$\bar r_{A_2}$} &  \small{$x_{1,\pi}$} & \small{$\epsilon_{1,\pi}$} & \small{$s_\theta$} &  \small{$s_{A_2}$}  \\
  
  \hline
   \scriptsize{\textbf{Set $1$}}  & \scriptsize $ 2.005$ & \scriptsize  $0.73$ & \scriptsize $0.398$ &\scriptsize  $1.50$& \scriptsize $2.68$ & \scriptsize $172.295$ & \scriptsize $-0.176$ & \scriptsize $0.214$ & \scriptsize $0.055$ & \scriptsize $0.241$ & \scriptsize $32.97$ & \scriptsize $5.357$ & \scriptsize $15.131$ & \scriptsize $4.151$  \\
 \hline
\scriptsize{\textbf{Set $2$}} & \scriptsize $4.589$ & \scriptsize  $0.504$ & \scriptsize $0.829$ &\scriptsize  $1.160$ & \scriptsize $2.775$ & \scriptsize $177.807$ & \scriptsize $0.609$ & \scriptsize $0.139$ & \scriptsize $0.077$ & \scriptsize $0.928$ & \scriptsize $34.56$ & \scriptsize $6.555$ & \scriptsize $14.733$ & \scriptsize $1.105$  \\
  \hline
\scriptsize{\textbf{Set $3$}}  & \scriptsize $1.795$ & \scriptsize  $0.535$ & \scriptsize $0.398$ &\scriptsize  $4.537$& \scriptsize $6.097$ & \scriptsize $150.576$ & \scriptsize $0.176$ & \scriptsize $0.678$ & \scriptsize $0.032$ & \scriptsize $0.29$ & \scriptsize $30$ & \scriptsize $4.385$ & \scriptsize $17.609$ & \scriptsize $6.49$  \\
 \hline

\end{tabular}
\caption{Different sets of parameters }
\label{tab:allpar}
\end{table}

\subsection{Indicator of the {invasive} potential}\label{sec:indic}
To study the effect of denervation and the role of axons in pancreatic cancer progression, we investigate the effect of three types of denervation:
\begin{itemize}
    \item the denervation of sympathetic axons ($A_1$),
    \item the denervation of sensory axons ($A_2$),
    \item the denervation of both types of axons ($A_1$ and $A_2$).
\end{itemize}

\noindent A quantitative measure of the effect of denervation for any of the above three types of denervation can be the difference at a given time between the amount of cancer cells in the control case (with both axons regulating the dynamics) and in the denervated case.
Although this function allows us to measure the effect of denervation, it does not allow us to take into account all the specificities of temporal dynamics.
Indeed, the earlier a cancer cell appears, the higher the probability of accumulating genetic alterations that enhance aggressiveness and metastasis.
So we construct the following function to measure the impact of denervation:
 \begin{equation} \label{eq:indic}
   \hspace{15mm}
\mathcal{I}(T) = \frac{1}{T} \left(  \int_0^T \frac{\NC(t|\vartheta,\text{OHDA})}{N(t|\vartheta,\text{OHDA})}dt -\int_0^T \frac{\NC(t|\vartheta,\text{AA})}{N(t|\vartheta,\text{AA})}dt \right)  
 \end{equation} 
with the notations of Section~\ref{sec:parden}.
{This indicator gives information about the invasive potential at time $T$ because it takes into account the history of existence of cancer cells over time. 

\noindent This indicator is positive for an overall pro-tumoral effect of denervation ($\mathcal{I}(T)>0$), and conversely it is negative for an overall anti-tumoral effect of denervation ($\mathcal{I}(T)<0$).

\subsection{{In silico denervation at defined time points}}\label{sec:axonsrole}
{
In this section, we apply the three types of denervations stated in Section \ref{sec:indic}. These denervations are implemented at two specific times:
\begin{itemize}
    \item at 28 days, which corresponds to the time between the observation of PanIN and the onset of PDAC for the KIC model (cf. Section \ref{sec:data}),
    \item at 50 days, which corresponds to the time between the onset of PDAC and the observation of advanced PDAC (cf. Section \ref{sec:data}). 
\end{itemize}
The dynamics of the model subject to these denervations are illustrated in Figure \ref{fig:den}.

{We denote by $T$ the time of observation, then}
the effects of the denervations are quantified by the indicator $\mathcal{I}(T)$ given by \eqref{eq:indic} and the results are summarized in Table \ref{tab:indic_denerv}.
{The invasive potential can be quantitatively compared with the various results in the literature on the effect of axons on the initiation and progression of PDAC:}
\begin{itemize}
    \item {in \cite{saloman2016ablation}, the authors show that {ablating sensory axons prior to the onset of the pathology (referred to as early denervation)}
    inhibits tumorigenesis of the pancreatic cancer. 
    This result is translated by $\mathcal{I}(T)<0 $.}
    \item {In \cite{guillot2022sympathetic}, {the authors show that ablating sympathetic axons prior to the onset of the pathology (referred to as early denervation) leads to an acceleration of tumor growth and metastasis. This result is translated by $\mathcal{I}(T)>0$ in the case of an early denervation of $A_1$.}}
    
    \item {Finally, in \cite{renz2018beta2} and \cite{guillot2022sympathetic}, {the authors conduct surgical 
    denervation for both $A_1$ and $A_2$  after and before invasive tumor formation, categorized as early and late denervation, respectively. This procedure results in the ablation of mixed sympathetic and sensory axons.}
    \begin{itemize}
        \item In \cite{renz2018beta2} they {show} that removing the axons inhibits tumor growth. Hence, this result is that $\mathcal{I}(T)<0$ in the case of late denervation of $A_1$ and $A_2$.
        \item However, in \cite{guillot2022sympathetic}, early denervation of the splanchnic nerve leads to a pro-tumoral effect with observations of metastasis and a smaller survival time of mice.
        This result is translated by $\mathcal{I}(T)> 0$ in the case of early denervation of $A_1$ and $A_2$.
    \end{itemize}
     }
\end{itemize}
The only nuance to note is that the time at which the effect of denervations is observed in the various articles is not so easily translated into a quantifiable datum T in days. 
For instance, the mice models used may differ in terms of chronological progression of the disease and the different experimental techniques used may also affect the temporal component of the data in their own way.
{To make the most of this variability,} the final time of observation $T$ is set at $T=70$ days as {the rounded upper bound of the median survival of mice ($63$ days in \cite{guillot2022sympathetic}) {which allows us to observe the effect of the denervation
over the entire time range.}}

\begin{table}[H]
\centering
\begin{tabular}{|l|c|c|c|c|c|c|}
 \hline
  &\scriptsize $A_1$ at $28d$ &\scriptsize $A_1$ at $50d$ &\scriptsize $A_2$ at $28d$ &\scriptsize $A_2$ at $50d$ &\scriptsize $A_1$ and $A_2$ at $28d$ &\scriptsize  $A_1$ and $A_2$ at $50d$   \\
  
  \hline
  \scriptsize{\textbf{Set $1$}} & \scriptsize $-2.311$ & \scriptsize  $-0.56$ & \scriptsize $-1.697$ &\scriptsize  $-1.58$ & \scriptsize $-3.961$ & \scriptsize $-2.249$ \\
\hline
\scriptsize{\textbf{Set $2$}} & \scriptsize $5.893$ & \scriptsize  $0.876$ & \scriptsize $-2.163$ &\scriptsize  $-1.85$ & \scriptsize $3.435$ & \scriptsize $-0.788$ \\
  \hline
\scriptsize{\textbf{Set $3$}} & \scriptsize $2.493$ & \scriptsize  $0.441$ & \scriptsize $-6.94$ &\scriptsize  $-6.366$ & \scriptsize $-4.482$ & \scriptsize $-5.47$ \\
\hline

\end{tabular}
\caption{Values in percentage of the invasive potential $\mathcal{I}(70)$ (cf. \eqref{eq:indic}) for the sets of parameters of Table \ref{tab:allpar}. The columns indicate the different denervations, sympathetic axons ($A_1$), sensory axons ($A_2$) or both types of axons ($A_1$ and $A_2$), at early stage (28 days) or late stage (50 days). }
\label{tab:indic_denerv}
\end{table}

\begin{figure}[ht!]
    \centering
    \begin{subfigure}[H]{0.495\textwidth}
 \includegraphics[width=\textwidth]{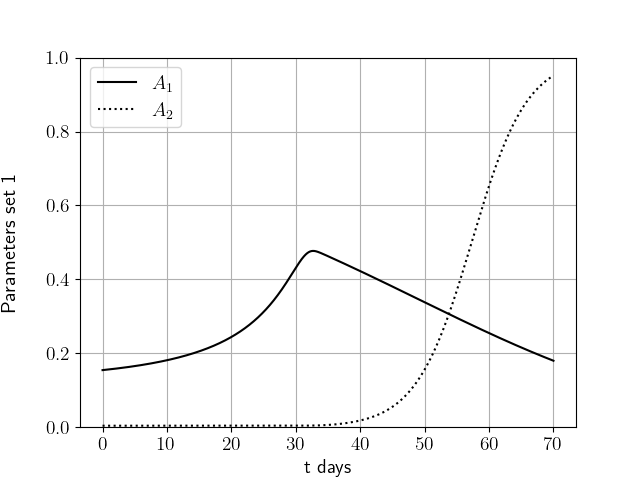}
        \caption{}
      \label{fig:axonremod2}
    \end{subfigure}
    \hfill
    \begin{subfigure}[H]{0.495\textwidth}
    \captionsetup{skip=-0.7ex}
 \includegraphics[width=\textwidth]{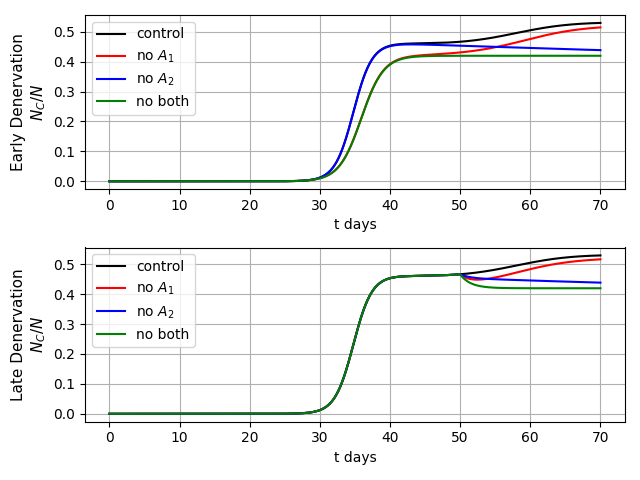}
        \caption{}
        \label{fig:par2}
     \end{subfigure}
     \vfill
    \begin{subfigure}[H]{0.495\textwidth}
        \includegraphics[width=\textwidth]{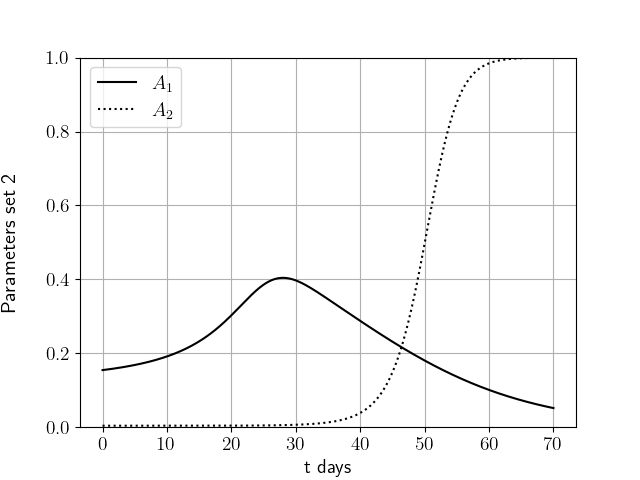}
        \caption{}
        \label{fig:axonremod0}
    \end{subfigure}
    \hfill
    \begin{subfigure}[H]{0.495\textwidth}
       \captionsetup{skip=-0.7ex} \includegraphics[width=\textwidth]{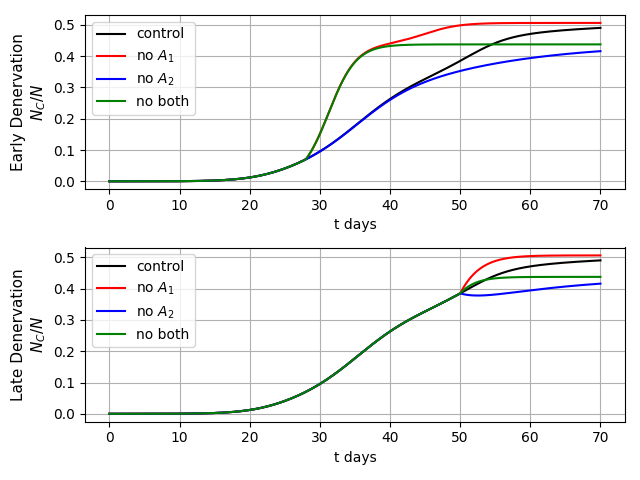}
        \caption{}
        \label{fig:par0}
     \end{subfigure}
     \vfill
    \begin{subfigure}[H]{0.495\textwidth}
        \includegraphics[width=\textwidth]{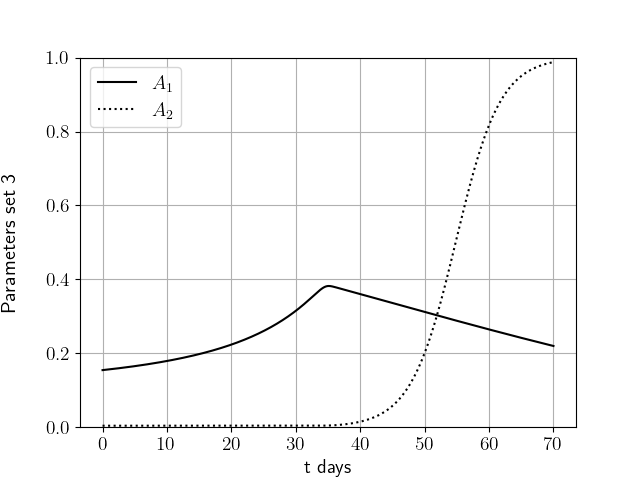}
        \caption{}
        \label{fig:axonremod1}
    \end{subfigure}
    \hfill
    \begin{subfigure}[H]{0.495\textwidth}
    \captionsetup{skip=-0.7ex}
    \includegraphics[width=\textwidth, height=6cm]{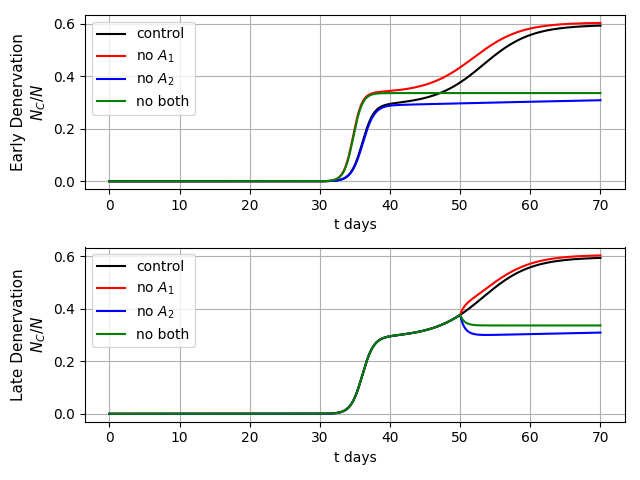}   
    \caption{}
    \label{fig:par1}
    \end{subfigure}

    \caption{Model dynamics subject to denervation. (Left) Evolution of the axons. (Right) Evolution of the proportion of cancerous
    cells $\tfrac{N_C(t)}{N(t)}$. The different colors show the effect of different denervations. The first row (resp. second row and third row) corresponds to the dynamics for the parameters set 1 (resp. 2 and 3) in Table \ref{tab:allpar}.}
        \label{fig:den}
\end{figure}

\subsubsection{Case when $\mu_1<0$}
Based on Table \ref{tab:indic_denerv}, the invasive potential of parameters set $1$ is always negative for the three types of denervations. Indeed,
only an anti-tumoral result is observed for the early and late denervation as seen in Figure \ref{fig:par2}, which do not recapitulate the results shown in \cite{guillot2022sympathetic}.
Therefore, this set can be considered as a false positive estimation due to the high variability of the data and the ill-posedeness of the calibration problem. Thus, in what follows, we focus our study on the parameters sets 2 and 3 of Table \ref{tab:allpar} when $\mu_1>0$ .

\subsubsection{Case when $\mu_1>0$}

In this section, we analyze the impact of axons on tumorigenesis using parameters sets 2 and 3 from Table \ref{tab:allpar}. 
Our model posits that sympathetic axons exert an anti-tumoral effect. 
Therefore, denervation of $A_1$ leads to an increased proportion of cancerous cells revealed by a positive value for invasive potential, whether denervation occurs early or late (see Figure \ref{fig:den} and Table \ref{tab:indic_denerv}). 
Conversely, sensory axons are implicated in promoting tumorigenesis. Thus, denervation of $A_2$ leads to a decrease in proportion of cancerous cells and a negative value for invasive potential (see Figure \ref{fig:den} and Table \ref{tab:indic_denerv}).

\noindent \paragraph{Effect of the double denervation of sympathetic and sensory axons.} 
The impact of denervating sympathetic and sensory axons proves to be complex and subject to diverse outcomes across the two distinct parameter sets, reflecting biological variability.\\

\paragraph{Outcome $1$.} For both set of parameters, early denervation of sympathetic and sensory axons result in a shift from an initial pro-tumoral effect of denervation (increase proportion of cancer cell compare to control condition) to a late anti-tumoral effect of denervation (decrease proportion of cancer cell compare to control condition). This "shift effect" is highlighted by the crossing of the green denervated and the black control curves in Figure \ref{fig:den}.\\

\paragraph{Outcome $2$.} For both set of parameters, late denervation of both axons types inhibits tumorigenesis. In Figure \ref{fig:den}, the green denervated curves are below the control ones meaning that pathological cancerous cells are present at a lower level compared to the control scenario.\\

\paragraph{Outcome $3$.} However significant difference emerges between the two sets. For early denervation, invasive potential is positive for parameter set 2 and negative for parameter set 3.
This suggests a higher net production of cancerous cells over the entire time span for set 2 compared to set 3 (higher area between the green and black curves before the “shift effect” and a lower area after the curves crossing, see Figure \ref{fig:den}).
This difference between the two sets reflects the dominance of a specific axon type (sympathetic for set 2 and sensory for set 3) in pancreatic cancer regulation.\\

{
\noindent The role of axons in cancer progression can be effectively established using the mathematical model and in silico denervation experiments. 
However, the relevance of in silico results is partly correlated with the calibration of the model. 
This calibration becomes more complex when the data is scarce and highly variable. 
Consequently, the parameter sets obtained during calibration exhibit different dynamics and denervation effects. 
Moreover, it is important to note that both the timing of denervation and the duration of observation post-denervation are critical factors in understanding the denervation's effects. 
}

\subsection{In silico denervation for varying times}\label{sec:heatmap}
Using the mathematical model, we investigate the temporal dynamics by denervating both sympathetic and sensory axons at various time points. 
Subsequently, we construct heatmaps at different observation times
to have an evolving illustration of the invasive potential.
\\

\subsubsection{Heatmap construction and interpretation}

{Studying experimentally the effect of different types of denervation can be time consuming and expensive.
However, the mathematical model allows to perform a large number of in silico denervations and hence gives insights on the role of the axons in tumorigenesis. 
In this Section, using the parameters sets 2 and 3 of Table \ref{tab:allpar}, we perform denervations varying the time of denervation of both axons independently and we quantify their pro- or anti-tumoral effect at different times during the tumorigenesis process. 
The results of the in silico denervations are illustrated by the two sequences of heatmaps of Figures \ref{fig:heatmap_0} and \ref{fig:heatmap_1}.}

\noindent First, we introduce the finite sequence $(s_i)\in  \{ 5, \, 10, \hdots, \, 70\}$ which corresponds to the observation times for the effect of denervations. 
Then, we denote by $t_{A_1}$ (resp. $t_{A_2}$) the time of denervation of the sympathetic axons $A_1$ (resp. the sensory axons $A_2$).
Each heatmap of Figure \ref{fig:heatmap_0} or \ref{fig:heatmap_1} corresponds to a grid where the y-axis correponds to the different values taken by $t_{A_1}$ and the x-axis to those taken by $t_{A_2}$.
The following holds 
$$t_{A_j}\in [0,s_i], \quad j=1,2,$$
since it makes no sense to look at the effect of denervation before denervation has taken place.\\
The cell from the heatmap $s_i$ located at $(t_{A_1}, t_{A_2})$ corresponds to the invasive potential at time $s_i$ subject to the denervation of sympathetic axons at time $t_{A_1}$ and subject to the denervation of sensory axons at time $t_{A_2}$ (cf. \eqref{eq:indic} and Section \ref{sec:indic}). 
{We keep track of the denervation time in the notation of the index evaluated at time $s_i$ as follows:}
$\mathcal{I}\left(s_i;t_{A_1},t_{A_2}\right)$.
Hence, two kind of results are observed:
\begin{itemize}
    \item if $\mathcal{I}(s_i;t_{A_1},t_{A_2})>0$, then the cell is colored in blue.
    Hence, the aggressiveness of the tumor 
    is encoded by this color.
    The darker the blue, the stronger the pro-tumoral effect of the denervation. 
    \item If $\mathcal{I}(s_i;t_{A_1},t_{A_2})<0$, then the cell is colored in brown.
    Hence, the possibility of tumor remission through denervation is encoded by this color. 
    The darker the brown, the stronger the anti-tumoral effect of the denervation.
\end{itemize}
}

Moreover, the sequence of heatmaps illustrates the dynamical evolution of the invasive potential. 
For instance, the red, blue, green, and black stars on the heatmaps of Figure \ref{fig:heatmap_0} (resp. Figure \ref{fig:heatmap_1}) correspond to the invasive potential of the early denervated and control curves of the parameters set 2 (resp. set 3) of Table \ref{tab:allpar} illustrated by the red, blue, green and black curves of Figure \eqref{fig:par0} (resp. Figure \eqref{fig:par1}).

\begin{figure}[ht!]
\includegraphics[width=0.9\textwidth]{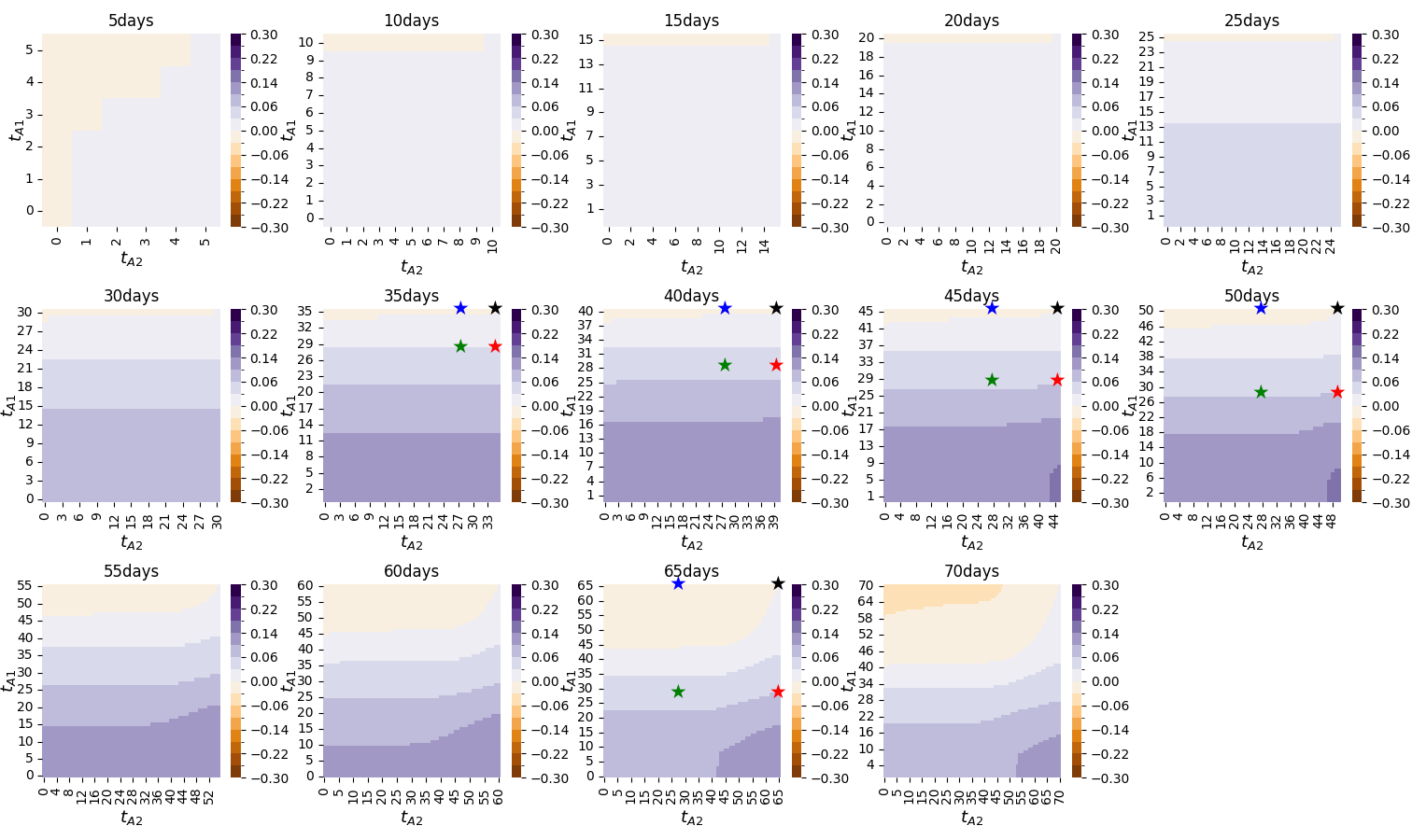}
\centering
\caption{{Evolution of the invasive potential with respect to varying denervation times for the parameters set $2$ from Table \ref{tab:allpar}. The times of denervations of $A_1$ (resp. $A_2$) are indicated on the y-axis (resp. the x-axis). The observation time $s_i $ of the invasive potential is indicated in days above each heatmaps.}
}\label{fig:heatmap_0}
\end{figure}

\begin{figure}[ht!]
\includegraphics[width=0.9\textwidth]{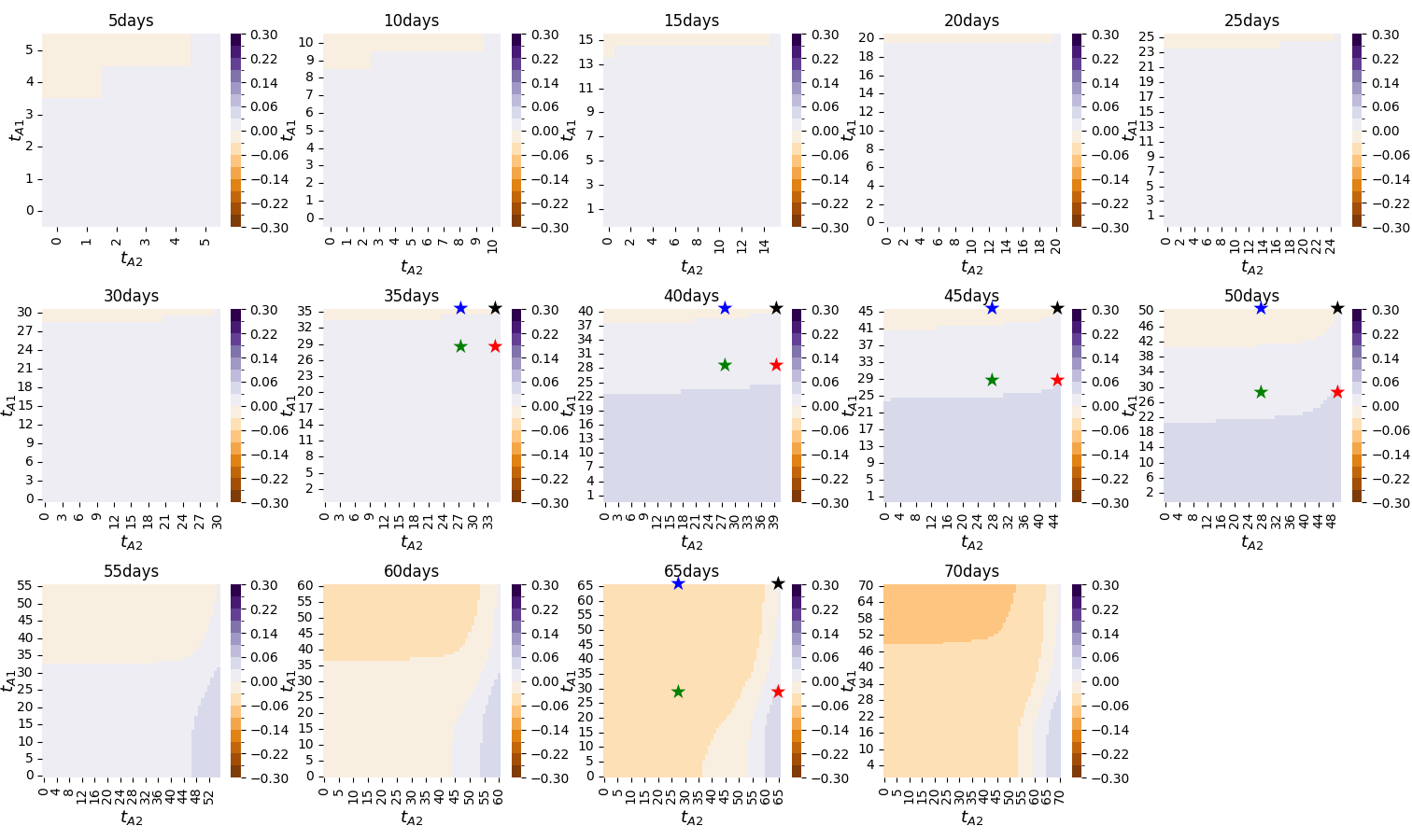}
\centering
\caption{{Evolution of the invasive potential with respect to varying denervation times for the parameters set $3$ from Table \ref{tab:allpar}. The times of denervations of $A_1$ (resp. $A_2$) are indicated on the y-axis (resp. the x-axis). The observation time $s_i $ of the invasive potential is indicated in days above each heatmaps.}
}\label{fig:heatmap_1}
\end{figure}
\subsubsection{Impact of the time-varying denervations on the evolution of the tumorigenesis.}
{In the following, a more \textit{qualitative} approach of model validation is proposed. 
A detailed description of the interesting elements of the model dynamics subject to time-varying denervations is given.
We discuss the possibility of establishing the best time and strategy of denervation.}\\

    {\textbf{Different times of remodeling of sympathetic axons explain the different apparition times of the pro-tumoral effect of denervations.}}
    The pro-tumoral effect of denervation starts to be observed at $s_i=25$ days for parameters set 2 (cf. Figure \ref{fig:heatmap_0}) and at $s_i=40$ days for parameters set 3  (cf. Figure \ref{fig:heatmap_1}) due to the darkening of the blue color in the lower halves of the heatmaps.
    The difference in the time of observation $s_i$ can be explained due to the difference in the time of arrival of sympathetic axons and the difference in the impact they have on cancer growth. 
    By referring to Table \ref{tab:allpar}, the parameters $r_{A1}$ and $\mu_1$ are larger for set 2 than set 3, explaining the larger and earlier effect that sympathetic axons denervation exhibits on tumor growth as seen in Figure \ref{fig:heatmap_0} compared to Figure  \ref{fig:heatmap_1}. \\
    
 \textbf{{In the early stages of the tumorigenesis, the strong pro-tumoral effect of denervation is associated to the early denervation of sympathetic axons.} 
    } 
    For $s_i \leq 50$, the denervations associated with a strong pro-tumoral effect are localized in the lower halves of the heatmaps for both sets of parameters, corresponding to early denervation of sympathetic axons.
    The time of denervation of sensory axons has almost no effect in the lower halves of the heatmaps, as the blue region extends homogeneously along the x-axis . 
    The fact that the denervations of sensory axons do not play a significant role in this time range can be explained due to the late remodeling of sensory axons and the delay between the denervation time of this type of axons and its impact on the dynamics of the model.
    Thus, in that time range, all types of denervations are  mainly impacted by the denervation of the sympathetic axons. \\
    
 \textbf{{In the late stages of the tumorigenesis,} the impact of sensory axons' denervation becomes significant as the dynamics for both parameters sets undergo the shift effect described in the previous section. }
For $s_i \geq 55$, the sensory axons starts to display its inhibiting effect on cancer cells which can be seen through expanding brown areas or lightening blue areas on the upper halves of the heatmaps. 
It illustrates the antagonistic role of both types of axons: the sympathetic axons playing an anti-tumoral role on the early stages of tumorigenesis and the sensory axons playing a pro-tumoral role on the late stages of the tumorigenesis.\\

\textbf{The sympathetic axons' (resp. sensory axons') impact on cancer growth dominates in Figure \ref{fig:heatmap_0} (resp. Figure \ref{fig:heatmap_1}). } 
On the one hand, the regulation of the sympathetic axons plays the most significant role on tumorigenesis for parameters set 2 since the denervation of this type of axons highlights a strong deviation of its dynamic from the control one (illustrated by the opacity level of the blue region in Figure \ref{fig:heatmap_0}).
On the other hand, the regulation of the sensory axons dominates for the parameters set 3 since the anti-tumoral effect of denervation becomes more and more significant as time passes. The brown area in Figure \ref{fig:heatmap_1} expands from the upper region at $s_i=55$ days to three-quarter of the heatmap's area at $s_i=70$ days  (everywhere but the south-east corner corresponding to the latest denervation of sensory axons and the earliest denervation of sympathetic axons).\\

\textbf{The effect of denervation takes time to be seen.}
Although an initial pro-tumoral effect of the denervation of {both types of axons} 
is always seen due to the {dominant} blue areas in the heatmaps for small/intermediate values of $s_i$, a later anti-tumoral effect of denervation is also observed when compared to the control curve after the first remodeling of sensory axons.
If the effect of sensory axons {on cancer growth is strong enough}, then the arrival of sensory axons {may} be sufficient {for} the overall anti-tumoral effect caused by it's denervation {to} compensate the previous pro-tumoral effect 
that has taken place.
{In Figure \ref{fig:heatmap_1}, if we compare the heatmaps at $s_i= 50$ and $s_i=60$ days,} we see areas going from blue to brown.
In that case, the final state of the disease (for large $s_i$) is a reduction of tumor, contrary to what can be observed initially for small $s_i$. 
In Figure \ref{fig:heatmap_0}, the stronger pro-tumoral effect resulting from the denervation of $A_1$ takes longer to be compensated by the anti-tumoral effect of the denervation of $A_2$.
However, it is possible to conjecture that the dynamics obtained of parameters set 2 at a later observation time ($s_i\gg 70$ days) will be similar to the ones displayed in the last frame Figure \ref{fig:heatmap_1} (parameters set 3 at observation time 70 days).

\section{Conclusion}

This study presents mathematical tools to model and simulate the joint effect of PNS axons (promoting and/or inhibiting cancer progression and proliferation) in pancreatic cancer tumorigenesis. 
It extends the previous model presented in \cite{chauvet2023tumorigenesis} by considering the cell phenotype as a continuous variable.
This new mathematical formalism provides a more accurate description of tumour progression and associated neuroplastic changes. 
The mathematical model is then finely calibrated to the available data by measuring the goodness of fit between the model output and the biological knowledge with a two-dimensional criterion and by selecting the relevant calibrated parameter sets with a multi-stage quasi-Monte Carlo algorithm. 
In addition, a quantitative indicator of the balance between the two opposing pro- and anti-tumor effects of denervation is calculated numerically.
This balance can be visualised over time for several parameter sets reflecting biological variability.\\

\noindent From a biological point of view, the mathematical model reconciles all the biological data found in the literature and provides an explanation for the opposite effects of surgical denervation performed at early and late stages of PDAC.
Specifically, the mathematical model shows that when sympathetic axons increase tumor growth, the model does not replicate the biological data from the literature.
This rules out the hypothesis of a functional switch of the sympathetic nervous system during tumorigenesis.
However, when sympathetic axons have a consistently inhibitory effect on tumor growth, the model recapitulates all the data. 
This can be explained by the "shift effect" from a global harmful to a protective role of the PNS, resulting from the compensation of the pro-tumor effect of denervating sympathetic axons by the anti-tumor effect of denervating sensory axons.
\\

\noindent In addition, the model identifies different sets of parameters that may reflect biological variability of tumor innervation.
These different sets highlight the importance of the level of sympathetic inhibition on tumor growth. 
For example, strong sympathetic inhibition on tumor growth masks the anti-tumor effect of sensory axon denervation.
In this case, targetting the PNS may have little or no benefit for PDAC treatment. 
However, if the sympathetic inhibition of tumor growth is weak, the anti-tumor effect of the sensory denervation becomes significant. 
In this scenario, targetting sensory axons would be beneficial. 
However, it is important to note that the beneficial effects are observed with a delay. 
Indeed, the model highlights a latency period between denervation and observable benefits.
Therefore, in terms of clinical applications for the treatment of pancreatic cancer, future knowledge of both the density of tumor innervation and the activity of sympathetic axons will be crucial for adapting the denervation strategy.
This highlights the importance of a patient-specific approach to the timing of nerve block in the treatment of pancreatic cancer.
\\

\noindent This study also highlights the fact that tumorigenesis is a constantly evolving process with inherent biological variability.
To account for the biological variability and complexity of this process, the temporal components of the data need to be further investigated. 
On the one hand, having additional data or at least longitudinal components as data helps to reduce the uncertainty of the parameter estimation problem. 
On the other hand, the predictive quality of the model also depends to a large extent on the data used to calibrate it.\\

\noindent An improvement in the mathematical model could be to include elements of the tumor microenvironment.
The complexity of pancreatic cancer tumorigenesis and its neuroregulation is also related to its interactions with the tumor microenvironment. 
A first step could be to incorporate spatial structure into the model. 
The spatio-temporal dynamics of such a model would certainly make it difficult to establish macroscopic properties and then draw biologically relevant conclusions.
Conversely, a more realistic model will allow more precise investigation of biophysical properties, such as tissue stiffness and its effect on axon growth, or more complex regulatory processes, such as additional regulation by the immune system.\\

\paragraph{Fundings} This research was supported by Centre National de la Recherche Scientifique (CNRS), France; Aix-Marseille University
(AMU), France; 
grant from Institut National Du Cancer (INCa), France,
Fondation Arc and Ligue contre le cancer (PAIR Pancreas, project title:
’’The impact of axonogenesis in pancreatic cancer’’, convention number
186738) to F.M and F.H.

\paragraph{CRediT} \insertcreditsstatement



\begin{appendix}
\section{Results for the well-posedness}\label{sec:app_wellpo}
In this section, we give the mathematical arguments used to prove the well-posedness of the model.
{The proof of Theorem \ref{thm:well_po} will require four steps. First, we prove that the map $\mathcal{S}$ is well defined, i.e the system \eqref{eq:model_abstract_linear} admits a unique solution. 
Second, we prove that the map $\mathcal{S}$ is a contraction. 
Third, we use a bootstrap argument to conclude the existence of a solution of \eqref{eq:model_abstract}.
The final step is the uniqueness result. All these steps relies on the following Lemmas.
}
\begin{lemma}
Let $T>0$, $\theta \in \mathcal{C}([0,T])$ such that Hypothesis \ref{hyp:cond_wellpo_theta_r} holds and $r>0$. {The Cauchy problem :
\begin{equation}
 \left\lbrace
\begin{array}{c@{}l}
\frac{d}{dt} y(t) & =r y(t)\left(1 - y(t)\right)\left(\frac{y(t)}{\theta(t)} -1 \right) ,\, t\in[0,T],\\
\\
y(0) &  = y_0 \in (0,1),
\end{array}
\right.
\label{eq:logistic_allee}
\end{equation}
admits a unique solution $y\in \mathcal{C}^1([0,T])$ associated to  $\theta \in \mathcal{C}([0,T])$ and $y(t) \in (0,1),$ $\forall t\in[0,T]$.}\\
Moreover, assume $y_1$ (resp. $y_2$) is the unique solution of \eqref{eq:logistic_allee} associated to the Allee effect term $\theta_1 \in \mathcal{C}([0,T])$ $\left(\text{resp. } \theta_2 \in \mathcal{C}([0,T])\right)$ where $y_1(0)  = y_2(0)= y_0\in(0,1)$. Let $\theta_i \in \mathcal{C}([0,T]) $ such that Hypothesis \ref{hyp:cond_wellpo_theta_r} holds and $i=1,\, 2$ then { 
    $$|y_1(t) - y_2(t)| \leq C_1(\theta_{-})rte^{C_2(\theta_{-}, \theta_{+})t} \| \theta_1 -\theta_2\|_{\mathcal{L}^\infty ([0,t])},$$
where $C_1>0$ and $C_2>0$ only depend on $\theta_{-}>0$ and $\theta_{+}>1$.}     
\label{lem:logistic_allee}
\end{lemma}
\begin{proof}[Lemma \ref{lem:logistic_allee}]
The local existence is a direct consequence of the Picard--Lindel\"of Theorem. Moreover, since $y(t) =0 $ and $y(t) = 1 $ are two stationary solutions, it implies that the solution with initial datum $y_0$ stays in$(0,1)$ and thus is globally defined. We focus on the proof of the second item.
We introduce the function $\bar{f}:\, (y, \theta) \in(0,1)\times \mathbb{R}_+^* \mapsto r y\left(1 - y\right)\left(\frac{y}{\theta} -1 \right)$.
Let us note that 
\begin{align*}
    |y_1(t) - y_2(t)| & \leq \int_0^t \left|\bar{f}(y_1(s), \theta_1(s)) -\bar{f}(y_2(s), \theta_2(s))\right|ds,\\ 
    & \leq  \int_0^t \left|\bar{f}(y_1(s), \theta_1(s)) -\bar{f}(y_1(s), \theta_2(s))\right|ds + \int_0^t \left|\bar{f}(y_1(s), \theta_2(s)) -\bar{f}(y_2(s), \theta_2(s))\right|ds.
\end{align*}
On the one hand, we have
\begin{align*}
    \int_0^t \left|\bar{f}(y_1(s), \theta_1(s)) -\bar{f}(y_1(s), \theta_2(s))\right|ds & \leq \int_0^t r y_1(s)^2(1- y_1(s))\left|\frac{1}{\theta_1(s)} - \frac{1}{\theta_2(s)}  \right|ds,\\
    &\leq \frac{\|\theta_1 -\theta_2\|_{\mathcal{L}^\infty ([0,t])}}{\min\limits_{\sigma \in [0,t] }|\theta_1(\sigma)\theta_2(\sigma)|} \int_0^t r y_1(s)^2(1- y_1(s)) ds.
\end{align*}
Moreover, the mapping $y\in(0,1) \mapsto y^2(1-y) $ admits a unique maximum in $y=2/3$. 
Hence, we get 
\begin{equation}
    \int_0^t \left|\bar{f}(y_1(s), \theta_1(s)) -\bar{f}(y_1(s), \theta_2(s))\right|ds \leq 
    \frac{4r t}{27\min\limits_{s\in [0,t]}|\theta_1(s) \theta_2(s)|} \| \theta_1 -\theta_2\|_{\mathcal{L}^\infty ([0,t])}.
    \label{eq:lem_bound1}
\end{equation}
On the second hand, we have
\begin{align*}
    \int_0^t \left|\bar{f}(y_1(s), \theta_2(s)) -\bar{f}(y_2(s), \theta_2(s))\right|ds & \leq \int_0^t \frac{|\bar{f}(y_1(s), \theta_2(s)) -\bar{f}(y_2(s), \theta_2(s))|}{|y_1(s) -y_2(s)|}|y_1(s) -y_2(s)|ds,\\
    &\leq \int_0^t |\partial_y \bar{f}(c,\theta_2(s))| |y_1(s) -y_2(s)| ds,
\end{align*}
where $c\in(0,1)$ is a constant coming from the mean value theorem. 
{Moreover, we have 
\begin{align*}
    \left|\left.\partial_y \bar{f}(y,\theta)\right|_{y=c}\right| & = \left|r\left[-1 + 2 \left(1+\tfrac{1}{\theta}\right)c - \tfrac{3}{\theta}c^2 \right]\right|,\\
    &\leq r\max\left(\frac{(\theta -1)^2}{3 \theta} +\tfrac{1}{3},1 \right).
\end{align*}
Hence, we get 
\begin{equation}
    \int_0^t \left|\bar{f}(y_1(s), \theta_2(s)) -\bar{f}(y_2(s), \theta_2(s))\right|ds \leq \int_0^t r\max\left( \frac{(\theta_2(s) -1)^2}{3 \theta_2(s)} +\tfrac{1}{3},1\right) |y_1(s) -y_2(s)|ds.
    \label{eq:lem_bound2}
\end{equation}
Finally, using the bounds \eqref{eq:lem_bound1}, \eqref{eq:lem_bound2} and Gr\"onwall's inequality, we get the following result
$$|y_1(t) - y_2(t)| \leq \frac{4r t\text{e}^{\int_0^t r\max\left( \frac{(\theta_2(\sigma) -1)^2}{3 \theta_2(\sigma)} +\tfrac{1}{3},1\right)d\sigma}}{27\min\limits_{s\in [0,t]}|\theta_1(s) \theta_2(s)|} \| \theta_1 -\theta_2\|_{\mathcal{L}^\infty ([0,t])}.$$ }
\end{proof}
\begin{lemma}
Let $T>0$, $r\in \mathcal{C}([0,T])\bigcap \mathcal{L}^1([0,T])$  and $K\in \mathbb{R}_+^*$. {The Cauchy problem :}
\begin{equation}
 \left\lbrace
\begin{array}{c@{}l}
\frac{d}{dt} x(t) & =r(t) x(t)\left(1 - \frac{x(t)}{K} \right) ,\, t\in[0,T),\\
\\
x(0) &  = x_0 \in (0,K),
\end{array}
\right.
\label{eq:logistic}
\end{equation}
{admits a unique solution associated to $r\in \mathcal{C}([0,T])\bigcap \mathcal{L}^1([0,T])$ given by 
$$x(t) = \frac{K}{1+ Ae^{-\int_0^t r(s) ds}}, $$
where $A=\frac{K-x_0}{x_0}.$}\\
{Moreover, assume $x_1$ (resp. $x_2$) is the unique solution of \eqref{eq:logistic} associated to the growth rate $r_1 \in \mathcal{C}([0,T])\bigcap \mathcal{L}^1([0,T])$ $\left(\text{resp. } r_2 \in \mathcal{C}([0,T])\bigcap \mathcal{L}^1([0,T])\right)$ where $r_i(t)>0,\, \text{ for } t\in [0,T] $ and $i=1,\, 2$ then 
    $$|x_1(t) - x_2(t)| \leq \frac{K A t}{\left(1+A\min\limits_{i=1,2}e^{-t\|r_i \|_{\mathcal{L}^\infty ([0,t])}} \right)^2} \| r_1 -r_2\|_{\mathcal{L}^\infty ([0,t])}. $$
}
\label{lem:logistic}
\end{lemma}
\begin{proof}[Lemma \ref{lem:logistic}]
{The Picard--Lindel\"of Theorem ensures that the Cauchy problem \eqref{eq:logistic} admits a unique solution that stays in $(0,K)$}. 
We focus now on the proof of the second item.
In the following, we use the notation $n_i(t) = e^{-\int_0^t r_i(s)ds}$ for $i=1,2$ and $t\in [0,T]$. Hence, we have 
\begin{align*}
    |x_1(t) - x_2(t)| & \leq K \frac{|A (n_1(t) - n_2(t)) |}{|(1+ A n_1(t)) (1+ An_2(t)|}\\
                    & \leq \frac{KA}{\left(1+A\min\limits_{i=1,2}n_i(t)\right)^2}  \int_0^t |r_1(s) -r_2(s)|ds\\
                    & \leq \frac{K A t}{\left(1+A\min\limits_{i=1,2}e^{-t\|r_i \|_{\mathcal{L}^\infty ([0,t])}} \right)^2} \| r_1 -r_2\|_{\mathcal{L}^\infty ([0,t])}
\end{align*}
since the following inequality holds for $(x, y) \in \mathbb{R_+}\times \mathbb{R_+}$ 
$$|e^{-x} - e^{-y}| \leq |x-y| .$$
\end{proof}
{In order to be able to state the next Lemma on well-posedness of the linear PDE, we introduce the characteristics associated to  \eqref{eq:model_abstract_linear}} by :
\begin{equation}
    \left\lbrace
    \begin{array}{c@{}l}
        \frac{d}{ds} X(s,t,x;\mathcal{X}) &= f(s, X(s,t,x;\X);\X),\, s\in \mathbb{R} \\
        X(t,t,x;\X) &=x. 
    \end{array}
    \right.
    \label{eq:characteristic}
\end{equation}

\begin{lemma}\label{lem:wellpo_linear_pb}
Let $\mathcal{X}\in\mathcal{P}$  (set defined in \eqref{eq:admissible_set}) be given and assume that the Hypothesis \ref{hyp:initial_cond_wellpo} on $Q_0 $ holds. Let $f$ and $c$ defined in \eqref{eq:model_abstract_linear} such that the Hypotheses \ref{hyp:cond_advec_wellpo} and \ref{hyp:cond_growth_wellpo} hold.
Then the PDE defined in \eqref{eq:model_abstract_linear} has unique solution $u\in \mathcal{C}^1\left([0,T], \mathcal{C}^1(\Omega)\cap \mathcal{L}^1(\Omega)\right)$ given by
$$u(t,x;\mathcal{X})=Q_0(X(0,t,x;\mathcal{X}))\times \exp\left( \int_0^t -c(\sigma,X(\sigma,t,x;\mathcal{X}))d\sigma \right) $$
{and we have
\begin{enumerate}
\item {$\int_\Omega Q_0(x)dxe^{-C_1 t} \leq \|u(t)\|_{\mathcal{L}^1(\Omega)} \leq \int_\Omega Q_0(x)dx e^{C_1 t}$, for $t\in [0,T]$},
\item $\|\partial_x u(s)\|_{\mathcal{L}^1(\Omega)} \leq C_2 t e^{C_1 t}$, for $t\in [0,T]$
\end{enumerate}
with $C_i >0$ for $i=1,2$ and {$C_1=r_+ \sup\limits_{0\leq t \leq T}\left|h(A_1(t),A_2(t))-N(t)\right|$.}\\
}
Moreover, let $\mathcal{X}_i = (Q_i, A_1^i, A_2^i) \in \mathcal{P}$  and assume that $u(t,x;\mathcal{X}_i)$ is the unique solution associated to $\mathcal{X}_i$ for $i=1,2$,
then 
$$\sup\limits_{0\leq t\leq T}\|u(t;\mathcal{X}_1)-u(t;\mathcal{X}_2)\|_{\mathcal{L}^1(\Omega)} \leq C(T)T \|\mathcal{X}_1 -\mathcal{X}_2\|_\mathcal{P} $$
where $0<C(T)<\infty$.
\end{lemma}

\begin{proof}

\noindent{\textit{First step}}.The existence of the unique solution of \eqref{eq:model_abstract_linear} under the assumptions of Lemma \ref{lem:wellpo_linear_pb} is a classical result and can be found for instance in Chapter 6 of \cite{perthame2006transport}. 
It relies on the existence and the regularity of the characteristics \eqref{eq:characteristic} that hold since $f\in \mathcal{C}\left([0,T],\mathcal{C}^1_c (\Omega) \right)$.\\

Now, given $\X\in \mathcal{P}$, we rewrite the equation \eqref{eq:model_abstract_linear} in its conservative global form (linear version of \eqref{eq:model_abstract}):
\begin{equation}
\left\lbrace
\begin{array}{c@{}l@{}l}
\partial_t u(t,x) & +\, \partial_x (f(t,x;\X) u(t,x)) + g(t,x;\X)u(t,x) = 0,\,&\, \text{for } t\in(0, T),\, x\in \Omega,\\
\\
f(t,x,\X)u(t,x)& =  0,&\, \text{for } x \in \partial\Omega,\; t\in(0, T),\\
\\
u(0,x) &= Q_0(x), &\, \text{for } x \in \Omega.\\
\end{array}
\right.
\label{eq:model_abstract_linear_conservative}
\end{equation}
Hence, we have that
\begin{align*}
\frac{d}{dt}\int_\Omega u(t,x)dx & = \int_\Omega -g(t,x;\X)u(t,x)dx, 
\end{align*}
with $g(t,x;\X)$ defined in \eqref{eq:model_abstract_g}.
We recall that $g(t, \cdot; \X_i) \in \mathcal{C}_c^1(\Omega)$, $0\leq {A_i}(t)\leq 1$ for $t\in[0,T]$ and $i=1,2$ (cf. the definition of the set $\mathcal{P}$ \eqref{eq:admissible_set}).
\\
{\textit{Second step.} We now prove the estimates on $u$ and on $\partial_x u$. First, one can notice that $u(t,x) \geq 0 $ for $t\in[0,T]$ and $x\in\Omega$ since $Q_0\geq 0$. On the first hand, using the assumptions on $g$ from Hypothesis \ref{hyp:cond_growth_wellpo} and the fact that $\X \in \mathcal{P}$, we have
\begin{align*}
\dt \|u(t)\|_{\mathcal{L}^1(\Omega)} =\dt \int_\Omega u(t,x)dx & \leq \int_\Omega |g(t,x)u(t,x)|dx,\\
&\leq \|g\|_{\mathcal{L}^\infty([0,T]\times \Omega)}\int_\Omega u(t,x)dx,\\
&\leq \|r\|_{\mathcal{L}^\infty(\Omega)}\sup\limits_{0\leq s \leq T}\left|h(A_1(s),A_2(s))-N(s)\right|\int_\Omega u(t,x)dx.
\end{align*}
On the other hand, 
{we have for $0\leq t \leq T$
$$ - \|r\|_{\mathcal{L}^\infty(\Omega)}\sup\limits_{0\leq s \leq T}\left|h(A_1(s),A_2(s))-N(s)\right|\int_\Omega u(t,x)dx \leq \dt \|u(t)\|_{\mathcal{L}^1(\Omega)} .$$}
Then, we obtain the first estimate of Lemma \ref{lem:wellpo_linear_pb} using Gronwall's lemma.
As for the second estimate, since $f(t,x) = \partial_x f(t,x)= 0  $ for $x\in \partial\Omega$, we have 
\begin{align*}
\dt \|\partial_x u(t)\|_{\mathcal{L}^1(\Omega)} &\leq \int_\Omega |g(t,x)| |\partial_x u(t,x)|dx + \int_\Omega |\partial_x g(t,x)|u(t,x)dx,\\
& \leq \|g\|_{\mathcal{L}^\infty([0,T]\times \Omega)} \|\partial_x u(t)\|_{\mathcal{L}^1(\Omega)} + C\left(N(0),\|\partial_x g\|_{\mathcal{L}^\infty([0,T]\times\Omega)}\right) e^{t\|g\|_{\mathcal{L}^\infty([0,T]\times\Omega)}}.
\end{align*}
Once again, using Gronwall's lemma, we obtain the following estimate for $t\in [0,T]$
$$\|\partial_x u(t)\|_{\mathcal{L}^1(\Omega)}\leq C\left(N(0),\|\partial_x g\|_{\mathcal{L}^\infty([0,T]\times\Omega)}\right)te^{t\|g\|_{\mathcal{L}^\infty([0,T]\times\Omega)}}.$$
}

\noindent{\textit{Third step.} Now, we prove the stability condition.} We denote $u_i \defeq u(\cdot; \X_i)$ and $v \defeq u_1 - u_2.$ We obtain that $v$ satisfies the following problem
\begin{equation}
\left\lbrace
\begin{array}{c@{}l@{}r}
\partial_t v & +\, f_1\partial_x u_1 - f_2\partial_x u_2 + c_1u_1 - c_2u_2 = 0,\,&  \text{in}\, (0,T] \times \Omega,\\
\\
v(t,x)& =  0,\,& \text{for } x\in \partial\Omega, t\in [0,T]\\
\\
v(0,x) &= 0, \,& x \in \Omega,\\
\end{array}
\right.
\label{eq:model_abstract_diff}
\end{equation}
where $f_i(t,x) = f(t,x;\X_i) $ and $c_i(t,x) = c(t,x;\X_i).$
It implies that
$$\dt \|v\|_{\mathcal{L}^1(\Omega)} = \int_\Omega \text{sign}(u_1-u_2)\left[\partial_x(f_2 u_2 -f_1 u_1) +(g_1 u_1 -g_2u_2) \right]dx $$
where $g_i(t,x) = g(t, x; \X_i)$ defined in \eqref{eq:model_abstract_g}.\\
{We denote $$I_1(t) \defeq \int_\Omega \text{sign}(u_1 -u_2)\partial_x(f_2 u_2 -f_1 u_1)dx \quad \text{and} \quad I_2(t)\defeq \int_\Omega \text{sign}(u_1 -u_2) (g_1 u_1 -g_2u_2)dx .$$
On the first hand, we have 
\begin{align*}
I_1(t) & = \int_\Omega \text{sign}(u_1-u_2)\partial_x (f_2(u_2-u_1)) + \text{sign}(u_1-u_2)\partial_x((f_2-f_1)u_1)dx,\\
 &\leq \int_\Omega \partial_x \left(-f_2 |u_2-u_1| \right)dx + \int_\Omega \left|\partial_x \left((f_2-f_1)u_1 \right)\right|dx,\\
& \leq \int_\Omega \left| \partial_x (f_2-f_1) u_1 + (f_2-f_1)\partial_x u_1 \right|dx,\\
& \leq \| \partial_x(f_2-f_1)\|_{\mathcal{L}^\infty(\Omega)} \|u_1\|_{\mathcal{L}^1(\Omega)} + \| (f_2-f_1)\|_{\mathcal{L}^\infty(\Omega)} \|\partial_x u_1\|_{\mathcal{L}^1(\Omega)}.
\end{align*}
Using Hypothesis \ref{hyp:cond_advec_wellpo} and the estimates on $u$, we obtain
\begin{equation}
I_1(t) \leq \| \X_1-\X_2\|_\mathcal{P} \big(C_0 + C_2 t \big)e^{tC_1}
\label{ineq:stab_i1}
\end{equation}
On the second hand, we have
\begin{align*}
    I_2(t) & \leq \int_\Omega |g_1 -g_2||u_1|dx +\int_{\Omega} |g_2| |u_1 - u_2| dx,\\
    & \leq \|u_1\|_{\mathcal{L}^1(\Omega)} \|g_1 -g_2\|_{\mathcal{L}^\infty(\Omega)} + \|g_2\|_{\mathcal{L}^\infty(\Omega)} \|u_1 -u_2\|_{\mathcal{L}^1(\Omega)}. \\
\end{align*}
Using Hypothesis \ref{hyp:cond_growth_wellpo} and the estimates on $u$, we obtain
\begin{equation}
I_2(t) \leq \| \X_1-\X_2\|_\mathcal{P} C_3e^{tC_1} + C_1 \|v\|_{\mathcal{L}_1(\Omega)}.
\label{ineq:stab_i2}
\end{equation}
We note that the constant $C_1$ in the inequalities \eqref{ineq:stab_i1} and \eqref{ineq:stab_i2} is the same and depends on the uniform bound of the growth term $g$ which is proportional to the distance between the integral of the initial condition and the saturation constant $C(\tau_C)$.\\
Thanks to the inequalities \eqref{ineq:stab_i1} and \eqref{ineq:stab_i2} and using the notations of Hypotheses \ref{hyp:cond_advec_wellpo} and \ref{hyp:cond_growth_wellpo}, we obtain 
$$
\dt \| v\|_{\mathcal{L}^1(\Omega)} \leq C_1 \|v\|_{\mathcal{L}_1(\Omega)} + \left(C_0 + C_3 +C_2 t \right)e^{tC_1}\| \X_1-\X_2\|_\mathcal{P},
$$ 
where
\begin{itemize}
\item $0 < C_0 = C\left(C_l(\partial_x f), \|Q_0\|_{\mathcal{L}^1} \right)$,
\item {$ 0 < C_1 = r_+ \left|C(\tau_C)-C(N(0))\right|$,}
\item $0 < C_2 = C\left(C_l(f), \|Q_0\|_{\mathcal{L}^1}, \|\partial_x g\|_\infty  \right)$,
\item $0 < C_3 = C\left(C_l(g),\|Q_0\|_{\mathcal{L}^1}\right)$.
\end{itemize}
Finally, using the Gronwall's Lemma on the previous inequality and taking the supremum over the interval $[0,T]$, we obtain 
$$\sup\limits_{0\leq t\leq T}\| v(s)\|_{\mathcal{L}^1(\Omega)} \leq T\left( C_0 + C_3 + \frac{C_2}{2}T \right)e^{TC_1} \| \X_1-\X_2\|_\mathcal{P}. $$
}
\end{proof}
{
\begin{lemma}
    Assume that $\X\defeq (Q,A_1,A_2) \in \mathcal{C}^1\left([0,T];\mathcal{C}^1(\Omega)\cap\mathcal{L}^1(\Omega)\right)\times \mathcal{C}^1([0,T])\times \mathcal{C}^1([0,T])$ is a solution of the system \eqref{eq:model_abstract} and $\mathcal{S}(\X) = \X$ for the mapping defined in \eqref{eq:mapping_contraction}. Assume Hypotheses \ref{hyp:initial_cond_wellpo},\ref{hyp:cond_advec_wellpo}, \ref{hyp:cond_growth_wellpo} hold.
    Then 
    \begin{itemize}
        \item \textit{(nonnegativity)} $Q(t,x)\geq 0$ and $A_i(t)\geq 0$ for $i=1,2$ and $\forall (t,x)\in[0,T]\times \Omega$,
        \item \textit{(boundedness)} $\int_\Omega Q_0(x)dx \leq \int_\Omega Q(t,x) dx \leq C(\tau_C)$ and $A_i(t)\leq 1$ for $i=1,2$ and $t\in [0,T]$.
    \end{itemize}
    \label{lem:sol_estimates}
\end{lemma}
}{
\begin{proof}
The nonnegativity property of the solution of system \eqref{eq:model_abstract} is immediatly obtained using the characteristics for $Q$ and noticing that $0$ is a stationnary solution for $A_i$ for $i=1,2$. The solution can be written as
$$Q(t,x) = Q_0(X(0,t,x))e^{-\int_0^t c(\sigma, X(\sigma,t,x))d\sigma}$$
and has the sign of $Q_0$.
The upper bounds for $A_i$ $i=1,2$ are also obtained noticing that $1$ is a stationary solution of the Cauchy problems. 
Concerning the solution of the PDE $Q$, using Hypotheses \ref{hyp:cond_growth_wellpo} and \ref{hyp:cond_advec_wellpo}, the nonnegativity of $Q$ and the fact that $\int_\Omega Q_0(x)\leq C(\tau_C)$, we first notice that 
$$\int_\Omega r(x)Q(t,x)dx (N(0) - N(t)) \leq \frac{d}{dt} N(t) = \int_\Omega r(x)Q(t,x)dx (h(A_1(t),A_2(t)) - N(t)) \leq \int_\Omega r(x)Q(t,x)dx (C(\tau_C) - N(t))$$
for $t\in[0,T]$ where we denote $N(t) = \int_\Omega Q(t,x) dx$.\\
Hence, we obtain the following inequalities
$$N(0) \leq N(t) \leq  C(\tau_C).  $$
\end{proof}}
Now, we prove Theorem \ref{thm:well_po}.\\
\begin{proof}[Theorem \ref{thm:well_po}]
\textit{Step 1. }We prove the well-posedness of the system \eqref{eq:model_abstract}. To this end, we apply a fixed point theorem for the mapping defined in \eqref{eq:mapping_contraction} on the set $\mathcal{P}(T)$ defined in \eqref{eq:admissible_set}:
\begin{equation*}
\mathcal{P} \defeq \mathcal{B}\times \left\lbrace A\in \mathcal{C}^1([0,T])\; | \;0\leq A(t) \leq 1, \text{ for } t\in[0,T] \right\rbrace^2,
\label{eq:admissible_set_solPDE}
\end{equation*}
where
\begin{equation*}  
\mathcal{B} \defeq \left\lbrace Q \in {\mathcal{C}^1\left([0,T];\mathcal{C}^1(\Omega)\cap\mathcal{L}^1(\Omega)\right)}  \, | \, C(N(0)) \leq \int_\Omega Q(s,x)dx \leq C(\tau_C) \, \text{for} \, s\in[0,T]  \right\rbrace
\end{equation*}
where $C(N(0))$ and $C(\tau_C)$ are positive constants and $$\| Q\|_{\mathcal{B}} \defeq \sup\limits_{0\leq s\leq T}\|Q(s)\|_{\mathcal{L}^1(\Omega)}.$$ 
The norm associated to this functional space is 
$$\| (Q,A_1,A_2) \|_{\mathcal{P}} \defeq  \left(\sup\limits_{0\leq s\leq T}\|Q(s)\|_{\mathcal{L}^1(\Omega)} + \sup\limits_{0\leq s\leq T} |A_1(s)| + \sup\limits_{0\leq s\leq T} |A_2(s)|\right).$$
For $\X= \left(Q, A_1, A_2\right) \in \mathcal{P}$ given, thanks to the Lemmas \ref{lem:logistic_allee}, \ref{lem:logistic} and \ref{lem:wellpo_linear_pb},  the linear problem \eqref{eq:model_abstract_linear} admits a unique solution 
$$(u, \widetilde{A_1}, \widetilde{A_2})\in {\mathcal{C}^1\left([0,T];\mathcal{C}^1(\Omega)\cap\mathcal{L}^1(\Omega)\right)} \times \mathcal{C}^1([0,T])\times \mathcal{C}^1([0,T]). $$
{Also, thanks to Lemma \ref{lem:wellpo_linear_pb}, choosing $T$ such that
\begin{equation}
    T<\frac{\log\left(\tfrac{C(\tau_C)}{N(0)} \right)}{\|r\|_\infty |C(\tau_C) - N(0)|} \times \frac{C(\tau_C) -N(0)}{C(\tau_C)-C(N(0))}= \frac{\lambda_1}{\|r\|_\infty \xi_1}
    \label{eq:ineq_T_linear_pb}
\end{equation}
where $\xi_1 \in (N(0); C(\tau_C)) $ and $\lambda_1 = \frac{C(\tau_C) -N(0)}{C(\tau_C)-C(N(0))}$
and that 
\begin{equation}
    T<\frac{\log\left(\tfrac{N(0)}{C(N(0))} \right)}{\|r\|_\infty |N(0) - C(N(0))|}\times \frac{N(0)-C(N(0)}{C(\tau_C)-C(N(0))} = \frac{\lambda_2}{\|r\|_\infty \xi_2}
    \label{eq:ineq_T_linear_pb_2}
\end{equation}
where $\xi_2 \in (C(N(0)); N(0)) $ and $\lambda_2 = \frac{N(0)-C(N(0))}{C(\tau_C)-C(N(0))}$
then the solution $u$ of the linear system \eqref{eq:model_abstract_linear} is such that $u\in\mathcal{B}.$
}
In order to prove that $\mathcal{S}$ has a fixed point, we introduce the following maps :
$$\F : \; \mathcal{P} \rightarrow \mathcal{B}, \quad \F(\mathcal{X})=u, $$ 
$$ \Gamma_1 : \; \mathcal{P} \rightarrow \mathcal{C}^1([0,T]),\quad  \Gamma_1(\mathcal{X})=\widetilde{A_1}, $$
$$ \Gamma_2 : \; \mathcal{P} \rightarrow \mathcal{C}^1([0,T]), \quad \Gamma_2(\mathcal{X})=\widetilde{A_2}. $$
The contraction property of $\mathcal{S}$ is {implied by some stability properties on the mappings} $\F, \; \Gamma_1$ and $\Gamma_2$.
In the following, we denote $\mathcal{X}_i = (Q_i,A_1^i,A_2^i)$ for $i=1,2$.\\

First, using the results in Lemma \ref{lem:logistic_allee} and the assumptions of the function $\theta$ stated in Hypothesis \ref{hyp:cond_wellpo_theta_r}, we note that
\begin{align}
\begin{split}
\sup\limits_{0\leq t \leq T}|\Gamma_1(\mathcal{X}_1) - \Gamma_1(\mathcal{X}_2)| & \leq  C\left( r_{A_1},\tfrac{1}{ (\theta_-)^2}\right)Te^{C(\theta )T} \sup\limits_{0\leq t \leq T}  \left|\theta\left(\frac{N(t;\mathcal{X}_1)}{N(0)} \right) -\theta\left(\frac{N(t;\mathcal{X}_2)}{N(0)} \right) \right|,\\ 
& \leq C\left( r_{A_1},\tfrac{1}{ (\theta_-)^2},C_l(\theta)\right)Te^{C(\theta )T} \sup\limits_{0\leq t \leq T}\|Q_1(t) - Q_2(t)\|_{\mathcal{L}^1(\Omega)},  
\end{split}\label{eq:wellpo_ineq_gamma1}
\end{align}
where $ 0<C\left( r_{A_1},\tfrac{1}{(\theta_-)^2},C_l(\theta)\right)<\infty$ and $0<C(\theta )<\infty $ are two constants.
\\

Similarly, using the results in Lemma \ref{lem:logistic} and the assumptions on the function $r_{A_2}$ stated in Hypothesis \ref{hyp:cond_wellpo_theta_r}, we note that
\begin{align}
\begin{split}
\sup\limits_{0\leq t \leq T}|\Gamma_2(\mathcal{X}_2) - \Gamma_2(\mathcal{X}_1)| & \leq C(A_2^0)T \sup\limits_{0\leq t \leq T} \left|r_{A_2}\left(\frac{N_c(t;\mathcal{X}_1)}{N(t;\mathcal{X}_1)}\right) -r_{A_2}\left(\frac{N_c(t;\mathcal{X}_2)}{N(t;\mathcal{X}_2)}\right) \right|,\\
&\leq C(A_2^0,C_l(r_{A_2}))T\sup\limits_{0\leq t \leq T} \left|\frac{N_c(t;\mathcal{X}_1)N(t;\mathcal{X}_2)-N_c(t;\mathcal{X}_2)N(t;\mathcal{X}_1)}{N(t;\mathcal{X}_1)N(t;\mathcal{X}_2)} \right| ,\\
& \leq C\left(A_2^0,C_l(r_{A_2}),\tfrac{1}{C(N(0))^2}, C(\tau_C) \right)T\sup\limits_{0\leq t \leq T} \|Q_1(t) -Q_2(t) \|_{\mathcal{L}^1(\Omega)},
\end{split}\label{eq:wellpo_ineq_gamma2}
\end{align}
where $0<C\left(A_2^0,C_l(r_{A_2}),\tfrac{1}{C(N(0))^2}, C(\tau_C) \right)<\infty$ is a constant. Moreover, one can check that the constant $C\left(A_2^0,C_l(r_{A_2}),\tfrac{1}{C(N(0))^2}, C(\tau_C) \right) $ is bounded by above and below by strictly positive constants independant of the initial data $A_2^0$ and $Q_0$ since there exist constants $m_1>0$ and $m_2>0$ 
{such that $$m_1<A_2^0<1 \quad \text{and} \quad m_2<C(N(0))<C(\tau_C).$$} 
Also, thanks to the stability result of Lemma \ref{lem:wellpo_linear_pb}, we note that
\begin{equation}
\sup\limits_{0\leq t\leq T}\|\mathcal{F}(\X_1)-\mathcal{F}(\mathcal{X}_2)\|_{\mathcal{L}^1(\Omega)} \leq C(T)T \|\mathcal{X}_1 -\mathcal{X}_2\|_\mathcal{P}.
\label{eq:wellpo_ineq_f}
\end{equation}
Finally, using the inequaties \eqref{eq:wellpo_ineq_gamma1}, \eqref{eq:wellpo_ineq_gamma2} and \eqref{eq:wellpo_ineq_f}, we obtain that there exit two constants $0<C_1<\infty$ and $0<C_2<\infty$ such that
$$\|\mathcal{S}(\X_1) - \mathcal{S}(\X_2)\|_\mathcal{P} \leq C_1(T) T e^{C_2 T} \|\mathcal{X}_1 -\mathcal{X}_2\|_\mathcal{P}  $$
{where, using the notations previously introduced and those of Hypotheses \ref{hyp:initial_cond_wellpo}, \ref{hyp:cond_advec_wellpo}, \ref{hyp:cond_growth_wellpo} and \ref{hyp:cond_wellpo_theta_r}, we have 
\begin{itemize}
    \item $C_1(T) = C\left(C_l(r_{A_2}),A_2^0, \tfrac{1}{C(N(0))^2}, C(\tau_C), r_{A_1}, \tfrac{1}{(\theta_-)^2}, C_l(\theta), C_l(\partial_x f),\|\partial_x g\|_\infty, C_l(g),T\right)$,
    \item $C_2 = C\left(C(\theta), r_+, |C(\tau_C) - C(N(0))| \right).$
\end{itemize}
}
{\noindent Moreover, we introduce the new constants $\tilde{C}_1(T)$ and $\tilde{C}_2$ independent of $A_2^0$ and $C(N(0))$ such that 
$$C_1(T) \leq \tilde{C_1}\left( C_l(r_{A_2}), \tfrac{1}{(m_2)^2}, C(\tau_C), r_{A_1}, \tfrac{1}{(\theta_-)^2}, C_l(\theta), C_l(\partial_x f),\|\partial_x g\|_\infty, C_l(g),T\right) $$
and 
$$ C_2 \leq \tilde{C}_2 \left(C(\theta), r_+, |C(\tau_C) - m_2| \right).$$
\noindent Since the map 
$$g:t\in[0,T] \mapsto g(t)\defeq \tilde{C}_1(t) t e^{\tilde{C}_2 t}$$ 
is continuous, increasing and $g(0)=0$, there exists $T_1>0$ such that 
\begin{equation}
\tilde{C}_1(T_1) T_1 e^{\tilde{C}_2 T_1}<1.
\label{eq:ineq_T_contraction}
\end{equation}
}
{Choosing $T_1>0$ satisfying  \ref{eq:ineq_T_linear_pb},\ref{eq:ineq_T_linear_pb_2} and \ref{eq:ineq_T_contraction} with $T_1< \tilde{T}$  {that is }
{$$T_1 < \min\left(\frac{\lambda_1}{r_+ C(\tau_C))}, \frac{e^{-\tilde{C}_2 \tilde{T}}}{\tilde{C}_1(\tilde{T})}, \frac{\lambda_2}{r_+ C(\tau_C))}\right)$$} 
it follows that $\mathcal{S}$ is a strict contraction on $\mathcal{P}$.}
Consequently, the contraction mapping Theorem implies there exists a unique $\mathcal{X}\in \mathcal{P}$ of \eqref{eq:model_abstract} on the time interval $[0,T_1]$.\\
\textit{Step 2.} {Now, we prove that we can extend the solution to $[0, T]$ for any $T>0$.
In order to do so, since $Q(t)\in \mathcal{C}^1(\Omega) \cap \mathcal{L}^1(\Omega) $ for $0\leq t \leq T_1$, we prove that we can repeat the previous arguments to extend the solution to the time interval $[T_1, 2T_1]$.
We first note that Lemma \ref{lem:sol_estimates} ensures that any solution $Q$ of \eqref{eq:model_abstract} satisfies 
$$N(0) \leq \int_\Omega Q(t,x) dx \leq C(\tau_C).$$
It implies that $N(T_1) \geq N(0)$. 
Hence, there exists a positive constant $C(N(T_1))$ such that 
$$ 0<C(N(T_1))<N(T_1) \leq C(\tau_C)$$
and that {
\begin{equation}
    C(N(T_1)) = N(T_1) - (N(0)-C(N(0)))\frac{C(\tau_C)-N(T_1)}{C(\tau_C)-N(0)}.
    \label{eq:iteration_lower_bounds_Q}
\end{equation}
}
Moreover, we define the space on which we apply the contraction mapping Theorem as the following
$$\mathcal{P}_1 = \mathcal{B}_1\times \left\lbrace A\in \mathcal{C}^1([T_1,2T_1])\; | \;0\leq A(t) \leq 1, \text{ for } t\in[T_1,2T_1] \right\rbrace^2,$$
where
\begin{equation*}  
\mathcal{B}_1 \defeq \left\lbrace Q \in {\mathcal{C}^1\left([T_1,2T_1];\mathcal{C}^1(\Omega)\cap\mathcal{L}^1(\Omega)\right)}  \, | \, C(N(T_1)) \leq \int_\Omega Q(s,x)dx \leq C(\tau_C) \, \text{for} \, s\in[T_1,2T_1]  \right\rbrace.
\end{equation*}
}{ 
Hence, we can proceed similarly as in \textit{Step 1}.
Thanks to \eqref{eq:iteration_lower_bounds_Q}, we have 
$$\lambda_1 = \frac{C(\tau_C)-N(0)}{C(\tau_C) -C(N(0))} = \frac{C(\tau_C)-N(T_1)}{C(\tau_C) -C(N(T_1))} $$
and the dependance to the initial condition which appears in \eqref{eq:ineq_T_linear_pb} does not cause an issue since 
$$\frac{\lambda_1}{r_+ C(\tau_C)}\leq \lambda_1\frac{\log\left(\tfrac{C(\tau_C)}{N(T_1)} \right)}{\|r\|_\infty |C(\tau_C) - N(T_1)|} = \frac{\lambda_1}{\|r\|_\infty \xi}$$
where $\xi \in (N(T_1); C(\tau_C)).$
As for the dependance to the initial condition which appears in \eqref{eq:ineq_T_linear_pb_2}, thanks to \eqref{eq:iteration_lower_bounds_Q}, we have
$$\lambda_2 = \frac{N(0) - C(N(0))}{C(\tau_C)-C(N(0))} = \frac{N(T_1) - C(N(T_1))}{C(\tau_C)-C(N(T_1))}$$
and 
$$\frac{\lambda_2}{r_+ C(\tau_C)}\leq \lambda_2\frac{\log\left(\tfrac{N(T_1)}{C(N(T_1))} \right)}{\|r\|_\infty |N(T_1) - C(N(T_1))|} = \frac{\lambda_2}{\|r\|_\infty \xi}$$
where $\xi \in (C(N(T_1)) ; N(T_1)).$\\
{Also, thanks to \eqref{eq:iteration_lower_bounds_Q} and the fact that $N(0) \leq N(T_1)$, the following holds 
$$C(N(T_1) - C(N(0)) = \frac{C(\tau_C) - C(N(0))}{C(\tau_C) - N(0)} \big( N(T_1) -N(0)\big) \geq 0. $$
Then, we have that 
$$ \frac{e^{-\tilde{C}_2 \tilde{T}}}{\tilde{C}_1(\tilde{T})} \leq \frac{e^{-{C}_2(C(N(T_1)) \tilde{T}}}{{C}_1(\tilde{T}, C(N(T_1)))}.$$
}
Using Lemmas \ref{lem:logistic_allee}, \ref{lem:logistic} and \ref{lem:wellpo_linear_pb} and since $2T_1-T_1 = T_1<\tilde{T}$, we obtain the strict contraction on the mapping on $\mathcal{P}_1$  with the same condition 
{$$T_1 < \min\left(\frac{\lambda_1}{r_+ C(\tau_C))}, \frac{e^{-\tilde{C}_2 \tilde{T}}}{\tilde{C}_1(\tilde{T})}, \frac{\lambda_2}{r_+ C(\tau_C))}\right).$$} 
Hence, by iterating, we extend the solution to the full interval $[0,T]$.\\
\textit{Step 3.} The uniqueness of the solution of the system is direct consequence of the stability properties of Lemmas \ref{lem:logistic_allee}, \ref{lem:logistic} and \ref{lem:wellpo_linear_pb} and the structural property of our system which implies that $\mathcal{X} = 0$ is a stationary solution.
}

\end{proof}

\section{Details about the numerics}\label{app:numerics}
The dynamical system is implemented in Python and the algorithm can be found at (\href{https://github.com/MarieJosec/PDE_Axons_Innerv}{\texttt{https://github.com/MarieJosec/PDE\_Axons\_Innerv}}).
\subsection{Scheme for the numerical approximation of the solution}

\noindent A classical upwind scheme in space and an explicit Euler scheme in time is proposed to approximate the system that can be written as
$$\begin{cases}
\partial_t Q(t,x)+\partial_x[F(x,A_1(t),A_2(t), N(t),\NC(t))Q(t,x)]=R(x,Q(t,x),A_1(t),A_2(t), N(t)),\\
\\
\frac{d}{dt} A_1(t)=G_1(A_1(t), N(t)),\\
\\
\frac{d}{dt} A_2(t)=G_2(A_2(t),\NC(t), N(t)),
\end{cases}$$
Note that the function $F$ always takes non negative values.
\medskip

\noindent Consider a constant step "time" discretization of the interval $[0,T]$ with $t_n=n dt$ and $n\in\{0,\cdots, N_T\}$, where $T=t_{N_T}$ and $dt$ is the time step. We also consider "space" discretization of the interval $[-L,L]$, $x_i=(i+\frac12) h$, and $i\in \{-p,\cdots, p-1\}$, $p \in \mathbb{N^*}$, where $h=\frac{L}{p}$ is the space step. We will use a finite volume approach and thus introduce  the points $x_{i+\frac12}=(i+1)h$ and $x_{i-\frac12}=ih$ that can be viewed as being the vertices of volume cells $M_{i}=[x_{i-\frac12},x_{i+\frac12}]$ whose centers are the $x_i$ where $L=x_{p-\frac12}$ .\\

\noindent We look for an approximation $Q^n_i$ of $Q(t_n,x_i)$, $N^n$ (resp. $\NC^n$) an approximation of $\int_{-L}^L Q(t_n,x)\, dx$ (resp.$\int_{0}^L Q(t_n,x)\, dx$), $A_1^n$ (resp. $A_2^n$) an approximation of $A_1(t_n)$ (resp. $A_2(t_n)$). To secure the positivity of $A_1^n$ and $A_2^n$, we approximate the logarithmic function $A_1^n$ and $A_2^n$ to then go back to the approximation of  $A_1^n$ and $A_2^n$  by taking the exponential function. 
\medskip

\noindent The scheme is then the following:
$$\frac h{dt}(Q^{n+1}_i-Q^n_i)+\left(F(x_{i+\frac12},A_1^n,A_2^n,N^n,N^n_c)Q^n_i-F(x_{i-\frac12},A_1^n,A_2^n,N^n,N^n_c)Q^n_{i-1}\right)=h R(x_i,Q_i^n,A_1^n,A_2^n,N^n),$$
for $\, n\geq 0$ and $ i\in\{-p+1,\cdots, p-1\}$ and 
$$
\begin{cases}
\frac{1}{dt}(\log(A_1^{n+1})-\log(A_1^n))=G_1(\exp({A_1^n}), N^n),\qquad n\geq 0,\\
\\
\frac1{dt}(\log(A_2^{n+1})-\log(A_2^n))=G_2(\exp{A_2^n}, N_c^n),\qquad n\geq 0,\\
\\
A_1^{n+1}  = \exp(\log(A_1^{n+1})),\qquad n\geq 0,\\
\\
A_2^{n+1}  = \exp(\log(A_2^{n+1})),\qquad n\geq 0.\\
\end{cases}
$$
We use a trapezoidal rule to approximate the integral terms:
$$\begin{cases}
N^n=\frac{h}{2}\left[Q_{-p}^n+\sum_{i=-p+1}^{p-2} 2Q_i^n + Q_{p-1}^n\right],\\
\\
N_c^n=\frac{h}{2}\left[Q_0^n+\sum_{i=i}^{p-2} 2Q_i^n+Q^n_{p-1}\right].
\end{cases}
$$
The approximation of the boundary conditions can be naturally written as
$$F(x_{i},A_1^n,A_2^n,N^n,N^n_c)Q^n_i=0,\, \hspace{3mm } n\geq 0, \hspace{3mm } i \in \{-p,p-1\},
$$
and the initial condition is approximated by 
$$\begin{cases}
Q^0_i=Q^0(x_i),\, i\in \{-p,\cdots, p-1\},\\
\\
A_1^0,\, A_2^0 \text{ given. }
\end{cases}$$
\noindent Note that the stability of the  transport scheme requires the CFL condition 
$$\frac{dt}h F(x_{i+\frac12},A_1^n,A_2^n,N^n,N^n_c)\leq 1, \quad \forall n\geq, \quad \forall i=-p,\cdots, p-1. $$

\subsection{Calibration through an optimization method}\label{app:calib}

For a given parameter set $\vartheta$, the computation of our 2-dimensional criterion $G(\vartheta) = (G_1(\vartheta), G_2(\vartheta))$ requires two numerical simulation approximations, one acting as if we were in the control group (AA) and a second one including a denervation treatment (OHDA). The integrals over time intervals are computed with a trapezoidal rule based on the discretization of the numerical scheme. Needless to say here that the computation of the gradient $\nabla_\vartheta G_i(\vartheta)$ of any coordinate of the criterion is not straightforward. This means that we have no guide to find the sets of parameters with small value of the criterion. The space we have to explore is the whole hypercube $\mathcal H$ of dimension $14$ defined by the range column of Table~\ref{tab_parameters}. That is why we rely on a two-stage algorithm, see Figure~\ref{fig:select param}. At first, we explore naively the whole range of parameters, from which we adjust an instrumental Gaussian distribution truncated to $\mathcal H$. In a second stage, the instrumental distribution is used to draw new parameter sets that are informed by the data and the biological knowledge and we seek among those draws for the best parameter sets according to our criterion.

We start the algorithm with a massive exploration of the whole hypercube $\mathcal H$ of dimension $14$ defined by the range column of Table~\ref{tab_parameters} drawn as a Sobol sequence, that is to say using a quasi-Monte Carlo algorithm. This gives us the collection $\mathcal J_0$ of size $n^\text{QMC}_0=2^{18}$ parameter sets that are distributed uniformly over the hypercube $\mathcal H$. The $2$-dimensional criterion is then computed for each parameter set $\vartheta$ in $\mathcal J_0$.

Among the parameter sets in $\mathcal J_0$, we select the $0.2\%$ best sets according to the first component of the criterion $G_1(\vartheta)$, i.e. the $0.2\%$ parameter sets that have the smallest value of $G_1(\vartheta)$. This collection is filtered again according to the second component of the criterion $G_2(\vartheta)$, i.e. we keep only the $20\%$ best sets according to $G_2(\vartheta)$. This gives us the collection $\mathcal J_1$ of parameter sets that are consistent with the cell data and the chronological knowledge, i.e. that have low values of both $G_1(\vartheta)$ and $G_2(\vartheta)$.

To adjust the instrumental distribution to the collection $\mathcal J_1$, we compute the mean and the covariance matrix of the collection, see Table~\ref{tab:meancov}. The instrumental distribution is then the multivariate Gaussian distribution centered at the means given in Table~\ref{tab:meancov}, with a diagonal covariance matrix set according to the observed variances of this Table, and truncated (or conditioned) to stay in $\mathcal H$. 

The second stage of the algorithm starts with many quasi-Monte Carlo draws from the instrumental distribution, using a transformation of another Sobol sequence of dimension $14$. The $n^\text{QMC}_2=2^{18}$ draw form the collection $\mathcal J_2$ of parameter sets. The $2$-dimensional criterion is then computed for each parameter set $\vartheta$ in $\mathcal J_2$. This collection is then filtering to keep only the $0.7\%$ best sets according to the first component of the criterion $G_1(\vartheta)$, and then among them the $50\%$ best sets according to the second component of the criterion $G_2(\vartheta)$. This gives us the final collection $\mathcal J_3$ of parameter sets that are consistent with the cell data and the chronological knowledge, i.e. that have low values of both $G_1(\vartheta)$ and $G_2(\vartheta)$. The final collection $\mathcal J_3$ is summarized in Figure~\ref{fig:step5}.

\begin{table}[H]

\centering
\begin{tabular}{l c c c c c c c c c c c c c c c c c c c c c }
 \hline
  & \scriptsize{$\pi_0$} & \scriptsize{$\beta$} & \scriptsize{$\delta$} & \scriptsize{$\gamma_r$} & \scriptsize{$s_r$} & \scriptsize{$\tau_C$} & \scriptsize{$\mu_1$} & \scriptsize{$\mu_2$} & \scriptsize{$r_{A_1}$} & \scriptsize{$\bar r_{A_2}$} &  \scriptsize{$x_{1,\pi}$} & \scriptsize{$\epsilon_{1,\pi}$} & \scriptsize{$s_\theta$} &  \scriptsize{$s_{A_2}$}  \\
  \hline
\scriptsize{Mean} & \scriptsize $3.069$ & \scriptsize  $0.526$ & \scriptsize $0.497$ &\scriptsize  $4.202$ & \scriptsize $5.046$ & \scriptsize $184.452$ & \scriptsize $0.234$ & \scriptsize $0.444$ & \scriptsize $0.032$ & \scriptsize $2.1$ & \scriptsize $34.249$ & \scriptsize $4.77$ & \scriptsize $16.006$ & \scriptsize $4.382$  \\
\scriptsize{Variance}  & \scriptsize $2.814$ & \scriptsize  $0.297$ & \scriptsize $0.29$ &\scriptsize  $2.93$& \scriptsize $2.833$ & \scriptsize $48.647$ & \scriptsize $0.516$ & \scriptsize $0.294$ & \scriptsize $0.024$ & \scriptsize $1.551$ & \scriptsize $3.706$ & \scriptsize $2.839$ & \scriptsize $2.134$ & \scriptsize $2.916$  \\
 \hline
\end{tabular}
\caption{Summary statistics of the collection $\mathcal J_1$ of parameter sets.}
\label{tab:meancov}
\end{table}

\begin{figure}[ht!]
\includegraphics[width=\textwidth]{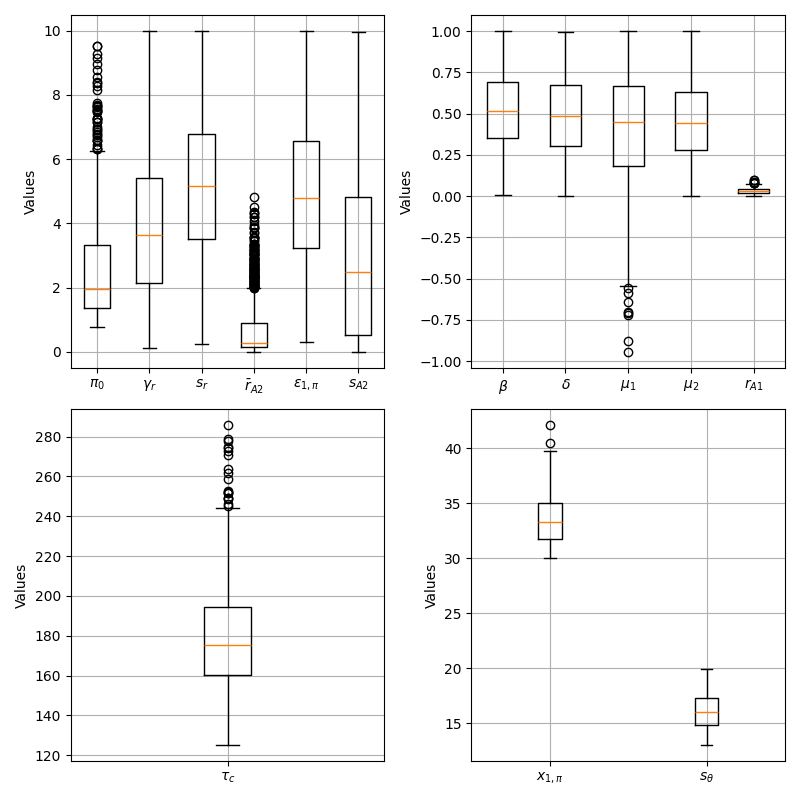}
\centering
\caption{Distribution of the collection $\mathcal{J}_3$ of parameter sets that are calibrated using our 2-dimensional criterion to fit the data and a few biological knowledge. These boxplots should be compared with the admissible ranges defined in Table~\ref{tab_parameters}.}\label{fig:step5}
\end{figure}

\end{appendix}

\bibliographystyle{apalike}
\bibliography{amu}

\begin{thebibliography}{}

\bibitem[Aguirre et~al., 2003]{aguirre2003activated}
Aguirre, A.~J., Bardeesy, N., Sinha, M., Lopez, L., Tuveson, D.~A., Horner, J.,
  Redston, M.~S., and DePinho, R.~A. (2003).
\newblock Activated kras and ink4a/arf deficiency cooperate to produce
  metastatic pancreatic ductal adenocarcinoma.
\newblock {\em Genes \& development}, 17(24):3112--3126.

\bibitem[Alexander and Cukierman, 2016]{ALEXANDER201680}
Alexander, J. and Cukierman, E. (2016).
\newblock Stromal dynamic reciprocity in cancer: intricacies of
  fibroblastic-ecm interactions.
\newblock {\em Current Opinion in Cell Biology}, 42:80--93.
\newblock Cell dynamics.

\bibitem[Banh et~al., 2020]{banh2020neurons}
Banh, R.~S., Biancur, D.~E., Yamamoto, K., Sohn, A.~S., Walters, B., Kuljanin,
  M., Gikandi, A., Wang, H., Mancias, J.~D., Schneider, R.~J., et~al. (2020).
\newblock Neurons release serine to support mrna translation in pancreatic
  cancer.
\newblock {\em Cell}, 183(5):1202--1218.

\bibitem[Biankin et~al., 2012]{biankin2012pancreatic}
Biankin, A.~V., Waddell, N., Kassahn, K.~S., Gingras, M.-C., Muthuswamy, L.~B.,
  Johns, A.~L., Miller, D.~K., Wilson, P.~J., Patch, A.-M., Wu, J., et~al.
  (2012).
\newblock Pancreatic cancer genomes reveal aberrations in axon guidance pathway
  genes.
\newblock {\em Nature}, 491(7424):399--405.

\bibitem[Chauvet et~al., 2023]{chauvet2023tumorigenesis}
Chauvet, S., Hubert, F., Mann, F., and Mezache, M. (2023).
\newblock Tumorigenesis and axons regulation for the pancreatic cancer: a
  mathematical approach.
\newblock {\em Journal of Theoretical Biology}, 556:111301.

\bibitem[Demir et~al., 2015]{demir2015neural}
Demir, I.~E., Friess, H., and Ceyhan, G.~O. (2015).
\newblock Neural plasticity in pancreatitis and pancreatic cancer.
\newblock {\em Nature reviews Gastroenterology \& hepatology}, 12(11):649--659.

\bibitem[Eftimie and Gibelli, 2020]{eftimie2020kinetic}
Eftimie, R. and Gibelli, L. (2020).
\newblock A kinetic theory approach for modelling tumour and macrophages
  heterogeneity and plasticity during cancer progression.
\newblock {\em Mathematical Models and Methods in Applied Sciences},
  30(04):659--683.

\bibitem[Guillot et~al., 2022]{guillot2022sympathetic}
Guillot, J., Dominici, C., Lucchesi, A., Nguyen, H. T.~T., Puget, A., Hocine,
  M., Rangel-Sosa, M.~M., Simic, M., Nigri, J., Guillaumond, F., et~al. (2022).
\newblock Sympathetic axonal sprouting induces changes in macrophage
  populations and protects against pancreatic cancer.
\newblock {\em Nature Communications}, 13(1):1985.

\bibitem[Hayakawa et~al., 2017]{hayakawa2017nerve}
Hayakawa, Y., Sakitani, K., Konishi, M., Asfaha, S., Niikura, R., Tomita, H.,
  Renz, B.~W., Tailor, Y., Macchini, M., Middelhoff, M., et~al. (2017).
\newblock Nerve growth factor promotes gastric tumorigenesis through aberrant
  cholinergic signaling.
\newblock {\em Cancer cell}, 31(1):21--34.

\bibitem[Klein et~al., 2002]{klein2002direct}
Klein, W.~M., Hruban, R.~H., Klein-Szanto, A.~J., and Wilentz, R.~E. (2002).
\newblock Direct correlation between proliferative activity and dysplasia in
  pancreatic intraepithelial neoplasia (panin): additional evidence for a
  recently proposed model of progression.
\newblock {\em Modern pathology}, 15(4):441--447.

\bibitem[Liddle, 2007]{liddle2007role}
Liddle, R.~A. (2007).
\newblock The role of transient receptor potential vanilloid 1 (trpv1) channels
  in pancreatitis.
\newblock {\em Biochimica et Biophysica Acta (BBA)-Molecular Basis of Disease},
  1772(8):869--878.

\bibitem[Lolas et~al., 2016]{lolas2016tumour}
Lolas, G., Bianchi, A., and Syrigos, K.~N. (2016).
\newblock Tumour-induced neoneurogenesis and perineural tumour growth: a
  mathematical approach.
\newblock {\em Scientific reports}, 6(1):1--10.

\bibitem[Perthame, 2006]{perthame2006transport}
Perthame, B. (2006).
\newblock {\em Transport equations in biology}.
\newblock Springer Science \& Business Media.

\bibitem[Renz et~al., 2018a]{renz2018beta2}
Renz, B.~W., Takahashi, R., Tanaka, T., Macchini, M., Hayakawa, Y., Dantes, Z.,
  Maurer, H.~C., Chen, X., Jiang, Z., Westphalen, C.~B., et~al. (2018a).
\newblock $\beta$2 adrenergic-neurotrophin feedforward loop promotes pancreatic
  cancer.
\newblock {\em Cancer cell}, 33(1):75--90.

\bibitem[Renz et~al., 2018b]{renz2018cholinergic}
Renz, B.~W., Tanaka, T., Sunagawa, M., Takahashi, R., Jiang, Z., Macchini, M.,
  Dantes, Z., Valenti, G., White, R.~A., Middelhoff, M.~A., et~al. (2018b).
\newblock Cholinergic signaling via muscarinic receptors directly and
  indirectly suppresses pancreatic tumorigenesis and cancer stemnesscholinergic
  signaling suppresses pancreatic tumorigenesis.
\newblock {\em Cancer discovery}, 8(11):1458--1473.

\bibitem[Saloman et~al., 2016]{saloman2016ablation}
Saloman, J.~L., Albers, K.~M., Li, D., Hartman, D.~J., Crawford, H.~C., Muha,
  E.~A., Rhim, A.~D., and Davis, B.~M. (2016).
\newblock Ablation of sensory neurons in a genetic model of pancreatic ductal
  adenocarcinoma slows initiation and progression of cancer.
\newblock {\em Proceedings of the National Academy of Sciences},
  113(11):3078--3083.

\bibitem[Sinha et~al., 2017]{sinha2017panin}
Sinha, S., Fu, Y.-Y., Grimont, A., Ketcham, M., Lafaro, K., Saglimbeni, J.~A.,
  Askan, G., Bailey, J.~M., Melchor, J.~P., Zhong, Y., et~al. (2017).
\newblock Panin neuroendocrine cells promote tumorigenesis via neuronal
  cross-talk.
\newblock {\em Cancer research}, 77(8):1868--1879.

\bibitem[Winkler et~al., 2023]{winkler2023cancer}
Winkler, F., Venkatesh, H.~S., Amit, M., Batchelor, T., Demir, I.~E., Deneen,
  B., Gutmann, D.~H., Hervey-Jumper, S., Kuner, T., Mabbott, D., et~al. (2023).
\newblock Cancer neuroscience: state of the field, emerging directions.
\newblock {\em Cell}, 186(8):1689--1707.

\end{thebibliography}

\end{document}